\documentclass[10pt, oneside]{amsart}   	
\usepackage{geometry}                		
\usepackage{graphicx}				
	
\usepackage{amssymb}
\usepackage{amsmath}
\usepackage{amsfonts}
\usepackage{amsthm}
\usepackage{amscd}
\usepackage{esint}
\usepackage{mathrsfs}

\usepackage{color}
\usepackage[usenames,dvipsnames,svgnames,table]{xcolor}

\usepackage[utf8]{inputenc}
\usepackage[parfill]{parskip}
\usepackage[colorlinks=true, pdfstartview=FitV,linkcolor=ForestGreen,citecolor=ForestGreen, urlcolor=black]{hyperref}

\usepackage{enumerate}
\usepackage{cite}
\usepackage{bbm}
\usepackage{comment}
\usepackage{url} 

\usepackage{tikz}

\newcommand{\dd}{\, {\rm d}}

\newcommand{\abs}[1]{\left\vert#1\right\vert}
\newcommand{\norm}[1]{\left\|#1\right\|}  
\newcommand{\set}[1]{\left\{ #1 \right\}}
\newcommand{\brak}[1]{\left\langle #1 \right\rangle}

\newcommand{\R}{\ensuremath{{\mathbb R}}}
\newcommand{\N}{\ensuremath{{\mathbb N}}}

\DeclareMathOperator{\eps}{\varepsilon}
\DeclareMathOperator{\embeds}{\hookrightarrow}
\DeclareMathOperator{\tens}{\otimes}

\newcommand{\weakstar}{\ensuremath{{\overset{\ast}{\rightharpoonup}}}}

\newcommand{\beq}{\begin{equation}}
\newcommand{\eeq}{\end{equation}}
\newcommand{\beqs}{\begin{equation*}}
\newcommand{\eeqs}{\end{equation*}}
\newcommand{\bal}{\begin{equation}\begin{aligned}}
\newcommand{\eal}{\end{aligned}\end{equation}}
\newcommand{\bals}{\begin{equation*}\begin{aligned}}
\newcommand{\eals}{\end{aligned}\end{equation*}}

\newcounter{num} \numberwithin{num}{section}

\newtheorem{theorem}[num]{Theorem}
\newtheorem{proposition}[num]{Proposition}
\newtheorem{lemma}[num]{Lemma}
\newtheorem{corollary}[num]{Corollary}
\theoremstyle{definition}

\theoremstyle{remark}
\newtheorem{remark}[num]{Remark}

\numberwithin{equation}{section}

\title[The Landau-Coulomb equation in $L^{3/2}$]{Global smooth solutions to the Landau-Coulomb equation in $L^{3/2}$}

\author{William Golding, Maria Gualdani, Amélie Loher}
\date{Nov 13, 2024}

\address[William Golding]{\newline Department of Mathematics, \newline The University of Texas at Austin, Austin, TX 78712, USA}
\email{wgolding@utexas.edu}
\address[Maria Gualdani]{\newline Department of Mathematics, \newline The University of Texas at Austin, Austin, TX 78712, USA}
\email{gualdani@math.utexas.edu}
\address[Amélie Loher]{\newline Department of Pure Mathematics and Mathematical Statistics, \newline University of Cambridge, Cambridge, UK}
\email{ajl221@cam.ac.uk}

\thanks{\textbf{Funding:} W.G. is partially supported by NSF grant DMS 1840314. A.L. is funded by the Cambridge Trust. M.P.G. is partially supported by NSF Grant DMS-2206677. }
\thanks{\textbf{Acknowledgments:} The authors would like to thank Dallas Albritton and Nestor Guillen for several enlightening discussions. The authors would like to thank the Isaac Newton Institute in Cambridge UK for their kind hospitality during the thematic program in Spring 2022.
W.G. is partially supported by NSF grant DMS 1840314. A.L. is funded by the Cambridge Trust. M.P.G. is partially supported by NSF Grant DMS-2206677. }

\begin{document}

\begin{abstract}

We consider the homogeneous Landau equation in $\mathbb{R}^3$ with Coulomb potential and initial data in polynomially weighted $L^{3/2}$. We show that there exists a smooth solution that is bounded for all positive times. The proof is based on short-time regularization estimates for the Fisher information, which, combined with the recent result of Guillen and Silvestre, yields the existence of a global-in-time smooth solution. Additionally, if the initial data belongs to $L^p$ with $p>3/2$, there is a unique solution. 

   At the crux of the result is a new $\varepsilon$-regularity criterion in the spirit of the Caffarelli-Kohn-Nirenberg theorem: a solution which is small in weighted $L^{3/2}$ is regular.  Although the $L^{3/2}$ norm is a critical quantity for the Landau-Coulomb equation, using this norm to measure the regularity of solutions presents significant complications. For instance, the $L^{3/2}$ norm alone is not enough to control the $L^\infty$ norm of the competing reaction and diffusion coefficients. These analytical challenges caused prior methods relying on the parabolic structure of the Landau-Coulomb  to break down. 
   
   Our new framework is general enough to handle slowly decaying and singular initial data, and provides the first proof of global well-posedness for the Landau-Coulomb equation with rough initial data.

\end{abstract}

\maketitle

\tableofcontents

We consider the homogeneous Landau equation on $\R^3$ with Coulomb potential,
\begin{equation}\label{eq:landau}
    \partial_t f = \nabla \cdot \left(A[f]\nabla f - \nabla a[f]f \right).
\end{equation}
Here the non-local coefficients $A[f]$ and $a[f]$ are given by convolutions:
\begin{equation*}
    A[f] = \frac{\Pi(v)}{8\pi|v|} \ast f, \qquad a[f] = (-\Delta)^{-1}f = \frac{1}{4\pi|v|} \ast f, \qquad \Pi(v) = Id - \frac{v\tens v}{\abs{v}^2}.
\end{equation*}
The mathematical investigation of equation (\ref{eq:landau}) has witnessed remarkable activity over the past three decades, leading to significant advancements and culminating in the recent breakthrough result by Guillen and Silvestre in \cite{GuillenSilvestre}. In \cite{GuillenSilvestre}, they consider $C^1$ initial data with Maxwellian upper bounds and show the existence and uniqueness of a global-in-time classical solution to \eqref{eq:landau}. Their proof relies on a novel approach to the equation that enables them to prove that the Fisher information, defined via
\begin{equation*}
    i(f) := \int_{\R^3} \frac{\abs{\nabla f}^2}{f} \dd v = 4\norm{\sqrt{f}}_{\dot{H}^1}^2,
\end{equation*}
is monotone decreasing. 
Then, by a standard Sobolev embedding, the Fisher information controls the $L^3$ norm uniformly in time. Consequently, the coefficients $A[f]$ and $\nabla a[f]$ are bounded\footnote{Technically, $\nabla a[f]$ belongs to $L^\infty(0,T;L^p)$ with $3/2 < p < +\infty$ if $f\in L^\infty(0,T;L^3)$.} for all times $t\ge 0$. Combined with the known lower bound on $A[f]$, \eqref{eq:landau} becomes a uniformly parabolic equation. Global regularity then follows via several approaches. One approach, taken in \cite{GuillenSilvestre}, is to use the maximum principle to propagate bounds on the initial data. Another approach, taken here, is to use parabolic regularization estimates to show smoothing effects for less regular initial data (see Theorem \ref{thm:long-cor-GS-GL}). The important point is that the a priori estimate on the Fisher information rigorously justifies that \eqref{eq:landau} is a parabolic equation with bounded (uniformly in time) measurable coefficients. This satisfactorily addresses the well-posedness theory for \eqref{eq:landau} for smooth and rapidly decaying initial data.

By contrast, the global well-posedness  for rough, slowly decaying initial data, meaning initial data with bounded mass, temperature, and Boltzmann entropy, remains unknown. Notably, regularity and uniqueness for $H$-solutions, first constructed by Villani in \cite{Villani_Hsols} and later revisited by Desvillettes in \cite{Desvillettes_HSolutions}, are still unknown.  One approach in this program is to prove existence of strong solutions under minimal assumptions on the initial data and show uniqueness of these strong solutions among a larger class of weak solutions. 

In this manuscript, we establish the global in time existence of smooth solutions to \eqref{eq:landau} for slowly decaying initial data in  $L^{3/2}$. We show that such solutions satisfy the quantitative short time regularization estimates:
\begin{equation*}
    \sup_{0 < t < 1} \norm{f(t)}_{L^\infty} \le \frac{C}{t} \qquad \text{and} \qquad \sup_{0 < t < 1} i(f(t)) \le \frac{C}{t}.
\end{equation*}
Since the Fisher information becomes instantaneously finite, we adapt  the main result of \cite{GuillenSilvestre} to our setting and continue the local-in-time solutions to global-in-time solutions.  We also prove existence of smooth solutions with initial data in $L^p$ for any $p > 3/2$. In this case, we additionally show uniqueness.

Previous existence results for smooth solutions considered $L^\infty$ data \cite{ArsenevPeskov, KimGuoHwang, HendersonSnelsonTarfulea} , $L^2$ data \cite{CarrapatosoMischler}, $\dot{H}^1$ data \cite{DesvillettesHeJiang}, and recently $L^p$ data with $p > 3/2$ \cite{GoldingLoher, GoldingGualdaniLoher}.\footnote{Technically, each of these results with the exception of Arsenev-Peskov \cite{ArsenevPeskov} requires a weighted a version of the space or additional $L^1$ moments.} Each of these works considers either global-in-time solutions for initial data close to equilibrium or local-in-time solutions for general initial data. Furthermore, each of these results requires the initial data to belong in a space which embeds into $L^p$ for some $p>3/2$.
The value $p=3/2$ is a threshold for the boundedness of the coefficients ($A[f]$ can be unbounded for $f \in L^{3/2}$ ) and also a threshold for time integrability of $\|f(t)\|_{L^\infty}$.  Consequently, $p = 3/2$ is the threshold at which most known analytical techniques break down and no longer apply.

\vspace{0.25cm}

{\bf{Literature:}} Following its introduction by Landau in \cite{Landau}, the Landau equation rapidly became one of the most important mathematical equations for modeling collisional plasmas. The Landau equation first appeared in the literature as a variant of the Boltzmann equation for particles interacting by a Coulomb potential. Indeed, using methods developed for the Boltzmann equation in \cite{Lions, DipernaLions, DipernaLions2}, one can construct renormalized solutions to the Landau equation and show rigorously that the Boltzmann equation converges to the Landau equation in a grazing collision limit \cite{Alexandre_Villani, DesvillettesVillani, Villani}. Separately, the Landau collision operator was also derived as a nonlinear version of a Fokker-Planck operator to describe particles that interact and exchange only small amount of momentum during collisions \cite{LandauLifshitz}.

The mathematical understanding of the Landau equation has come quite far in the past few decades, with fundamental contributions in many different directions. Given the extensive body of work on the Landau equation, our overview focuses primarily on works concerning the homogeneous Landau equation with Coulomb potential. For a discussion of other potentials and the inhomogeneous equation, we direct readers to the review paper of Silvestre \cite{Silvestre_Review}. Relevant to the present work, particularly from an analytical standpoint, are noteworthy contributions that can be summarized, albeit incompletely, as follows: (i) The existence and uniqueness of smooth solutions for short times \cite{GoldingLoher}. (ii) Global existence and uniqueness of smooth solutions with initial data near Maxwellian. Most relevant here are the ones by  Carrapatoso and Mischler \cite{CarrapatosoMischler}, Desvillettes, He and Jiang \cite{DesvillettesHeJiang}, Golding, Gualdani, and Loher \cite{ GoldingGualdaniLoher}, along with Guo \cite{Guo} and related references.  (iii) Conditional uniqueness and regularity, focusing on conditions ensuring the global well-posedness of solutions for arbitrarily large times. This line of research includes investigations by Silvestre \cite{Silvestre}, Gualdani and Guillen \cite{GualdaniGuillen1, GualdaniGuillen2}, Alonso, Bagland, Desvillettes, Lods \cite{AlonsoBaglandDesvillettesLods_ProdiSerrin}, Fournier \cite{Fournier}, and  Chern and Gualdani \cite{ChernGualdani}.  (iv) Partial regularity \cite{GolseGualdaniImbertVasseur_PartialRegularity1, GolseImbertVasseur_PartialRegularity2}. (v) Study of modified models that pertain the same difficulties of the Landau equation but seem analytically more tractable, explored in works by Gressman, Krieger and Strain  \cite{KriegerStrain, GressmanKriegerStrain} and Gualdani and Guillen \cite{GualdaniGuillen1}. (vi) Efforts directed at ruling out blow-up profiles initiated with the work of Bedrossian, Gualdani, and Snelson \cite{BedrossianGualdaniSnelson} and culminated in the recent work of Guillen and Silvestre \cite{GuillenSilvestre}. Noteworthy is the very recent proof of blow-up for a modified Landau equation by Chen in \cite{Chen}.  
  
\vspace{0.25cm}

{\bf{Main results:}} Throughout the manuscript we assume that  
\begin{equation}\label{eq:normalisation}
    f_{in} \ge 0, \qquad \int_{\R^3} f_{in} = 1, \qquad \int_{\R^3} v f_{in} = 0, \qquad \text{and} \qquad \int_{\R^3} |v|^2 f_{in} = 3.
\end{equation}
We also denote by $L^p_{m}$ the space of measurable functions such that
\begin{equation*}
\norm{f}_{L^p_{m}}:= \left(\int_{\R^3} f^p \brak{v}^m \;dv\right)^{1/p} < +\infty. 
\end{equation*}
Our main results are summarized in the following theorems: 
\begin{theorem}\label{thm:short} (Local-in-Time well-posedness) 
Let $\frac{3}{2}\le  p <  +\infty $. If the initial data $f_{in}$ is such that $f_{in}\langle v \rangle^{\frac{9}{2p}} \in L^p(\mathbb{R}^3)$, then there exists a smooth solution $f:[0,T]\times \R^3 \to \R^+$ to (\ref{eq:landau}) such that 
\begin{equation}\label{eq:Linfty}
\norm{f}_{C([0,T];L^p_{9/2})} + \sup_{0 < t \le T} t^{\frac{3}{2p}}\|f(t)\|_{L^\infty} \lesssim C(f_{in}). 
\end{equation}
Moreover, if $p > 3/2$, this solution is unique among all smooth solutions\footnote{More precisely, following \cite{GoldingLoher}, we assume $g$ is a locally bounded $H$-solution: we assume $g$ is a distributional solution; we assume the entropy dissipation of $g$ belongs to $L^1(0,T)$; and we assume $g$ satisfies the additional regularity assumption $g\in L^\infty(s,T;L^\infty(\Omega))$ for each $s > 0$ and $\Omega \subset \R^3$ compact. It is well-known that such solutions are necessarily qualitatively smooth. This technical regularity assumption is solely to ensure that truncations of $g$ are admissible as test functions for \eqref{eq:landau}. We expect this assumption is unnecessary and can be removed.} $g:[0,T]\times \R^3 \to \R^+$ to \eqref{eq:landau}
which obtain the initial data in the following sense:
\begin{enumerate}[(1)]
    \item $\begin{aligned}
        \lim_{t \to 0^+} g(t) = f_{in} \quad \text{in the sense of distributions}
    \end{aligned}$
    \item $\begin{aligned}
        g\in L^r(0,T;L^q_{9/2}) \quad \text{for some pair }(r,q),\quad 3/2 < q \le \infty,\quad \frac{2q}{2q-3} < r \le \infty.
    \end{aligned}$
\end{enumerate}
\end{theorem}

The $L^{3/2} \to L^\infty$ regularization estimate in \eqref{eq:Linfty} agrees with the regularization rate observed for the heat equation. However, we conjecture that the $L^p \to L^\infty$ regularization rate of the heat equation is subobtimal for the Landau diffusion operator, when $p < 3/2$, as predicted in \cite{CabreraGualdaniGuillen,GualdaniGuillen2}. The enhanced regularization rate is expected because the diffusion operator $\nabla \cdot (A[f]\nabla f)$  behaves more nonlinearly if $p<3/2$. 
Let us comment further on the meaning of the threshold $p = 3/2$. First, we recall that for any function $f$, the diffusion matrix is uniformly bounded below as 
\begin{equation*}
A[f](v,t) \ge c_0 \brak{v}^{-3} I, 
\end{equation*}
where $c_0$ only depends on the mass, energy, and entropy of $f$. One is then led to the toy problem:
\begin{equation}\label{eq:semilinear_heat}
u_t = \Delta u + u^2.    
\end{equation}
and the associated $1$-parameter scaling $u_\lambda = \lambda^2u(\lambda^2 t, \lambda x)$. This scaling preserves \eqref{eq:semilinear_heat} and leaves the $L^\infty(0,T,L^{3/2})$ norm invariant. In this sense, the $L^{3/2}$ norm is critical with respect to the scaling of this toy problem. This analogy is, however, simultaneously both helpful and misleading. This comparison is helpful in identifying the $L^{3/2}$ as an important norm for Landau, but not critical in the usual scaling sense. This is the motivation for the $\varepsilon$-regularity criterion, presented in Section \ref{sec:DeGiorgi}, which shows that if a solution to the Landau equation (\ref{eq:landau}) is small in a (critical) weighted $L^{3/2}$ norm, then the solution is smooth. Because the smallness required by the $\varepsilon$-regularity criterion cannot be directly removed, we are not able to show regularization estimates of the form $L^\infty(0,T,L^{3/2}) \to L^{\infty}(t,T, L^{\infty})$, nor are we able to propagate $L^{3/2}$ directly. Instead, we can replace the smallness by estimates of the form  $C(0,T,L^{3/2}) \to L^\infty(t,T, L^{\infty})$ and quantitatively propagate continuity in $L^{3/2}$ by solving a system of coupled differential inequalities. While the comparison to \eqref{eq:semilinear_heat} provides inspiration for the smoothing estimates, the standard techniques for the semilinear heat equation and for the incompressible three-dimensional Navier-Stokes equation do not apply here, due to the nonlinear and nonlocal nature of \eqref{eq:landau} and the presence of the unbounded coefficients.

For all positive times, the solution constructed in Theorem \ref{thm:short} is smooth and the notion of solution is unambiguous. However, the initial data is rough, and the sense in which the initial data is obtained is a delicate point, especially when considering uniqueness. 
For $p=3/2$, the solution $f$ does not belong to $L^1(0,T,L^\infty)$, and the uniqueness result by Fournier \cite{Fournier} does not apply. Nevertheless, we expect this solution to be unique among smooth solutions in the class $C([0,T];L^{3/2}_{9/2})$, but the proof is still open. For readers interested in investigating uniqueness, we note that the solutions constructed here in the $p = 3/2$ case additionally satisfy:
\begin{equation*}
\begin{aligned}
   &f \in L^{\frac{2q}{2q-3}}(0,T;L^q_{9/2}) \quad \text{if }q\in [3/2, \infty), \\
    &\nabla f^{3/4} \in L^2(0,T;L^2_{3/2}), \quad \text{and} \quad \lim_{t\to 0^+} t\norm{f(t)}_{L^\infty} = 0.
\end{aligned}
\end{equation*}

Note that the unweighted $p=\infty$ result of Arsenev-Peskov \cite{ArsenevPeskov} fits naturally in our framework. Our decay assumptions are less restrictive in terms of weights than the ones in  \cite{AlonsoBaglandDesvillettesLods_ProdiSerrin} used to show conditional regularity. This is a result of a new (weighted) $\eps$-Poincar\'{e} inequality used to rigorously show uniqueness (see Lemma \ref{lem:eps_poincare}). The proof of Lemma \ref{lem:eps_poincare} follows from the Caffarelli-Kohn-Nirenberg inequality and is located in Appendix \ref{sec:eps_poincare}. The integrability condition $f\in L^1_2$ and the normalization (\ref{eq:normalisation}) for $f$ are used solely in the lower bound for $A[f]$. 

\vspace{0.25cm}

To obtain global-in-time solutions from Theorem \ref{thm:short} we use the monotonicity of the Fisher information proved by Guillen and Silvestre in \cite{GuillenSilvestre}. Their framework, as stated, requires $C^1$ initial data with Maxwellian tails. We cannot expect that the solutions constructed in Theorem \ref{thm:short} satisfy these assumptions for generic initial data. Instead, we use the following adaptation of their result:


\begin{theorem}\label{thm:long-cor-GS-GL}
    Suppose $g_{in} \in L^1_m$ for some $m > 6$ and $g_{in} \in W^{1,1}_{loc}$ has finite Fisher information\footnote{There are two reasonable definitions for the Fisher information of a general nonnegative function $g\in L^1(\R^3)$ that will be used frequently, namely $\norm{\sqrt{g}}_{\dot{H}^1}^2$ and $\norm{g^{-1}\abs{\nabla g}^2}_{L^1}$, with the convention that the integrand is $0$ on the set $\set{g = 0}$. We remind the reader that under the mild assumption that $g \in W^{1,1}_{loc}$ these definitions are equivalent up to a factor of four.}. Then, there is a unique global-in-time smooth solution to \eqref{eq:landau} with initial datum $g_{in}$ and decreasing Fisher information.
\end{theorem}

The proof of Theorem \ref{thm:long-cor-GS-GL} uses an approximation argument that ensures the Guillen-Silvestre a priori estimates can be applied here.
Consequently, applying Theorem \ref{thm:long-cor-GS-GL}, we obtain the following global-in-time existence result: 
\begin{theorem}\label{thm:long}(Global-in-time Existence)
Fix $\frac{3}{2} \le p < \infty$, $m > 6$, and $m_0 \ge 9/2$, and initial data $f_{in}$ such that 
\begin{equation}\label{eq:initial_data_moments}
    f_{in} \in L^p_{m_0} \cap L^1_m \quad \text{and} \quad f_{in} \ge \frac{a}{\brak{v}^{k}}, \quad \text{for some }a > 0 \text{ and }k \le \frac{m_0 - 3}{p-1}.
\end{equation} 
Then, \eqref{eq:landau} admits a global-in-time smooth solution $f:[0,\infty) \times \R^3 \to \R^+$ with initial data $f_{in}$. Moreover, $f$ has decreasing Fisher information, which satisfies the short-time smoothing estimate:
\begin{equation}\label{eq:Fisher}
    \int_{\R^3} \frac{\abs{\nabla f}^2}{f} \dd v \le \frac{C(a,p,k,m_0,f_{in})}{t} \qquad \text{for}\qquad 0 < t < 1.
\end{equation}
\end{theorem}

\begin{remark}
    There is an implicit compatibility condition hidden within assumption \eqref{eq:initial_data_moments}. Namely, the lower bound $\brak{v}^{-k}$ must belong to $L^p_{m_0} \cap L^1_m$. This holds if 
    \begin{equation*}
        \frac{m_0 - 3}{p - 1} \ge k > \max\left(\frac{m_0 + 3}{p}, m + 3\right).
    \end{equation*}
    In particular, this compatibility condition is satisfied for initial data $f_{in}$ satisfying
    \begin{equation*}
        f_{in} \in L^{3/2}(B_R), \qquad f_{in} \gtrsim 1 \quad \text{on }B_R, \qquad f_{in} \sim \abs{v}^{-k} \quad \text{for }\abs{v} \ge R,
    \end{equation*}
    provided $k > 9$.
\end{remark}

Theorem \ref{thm:long} is \emph{NOT} proved by iterating Theorem \ref{thm:short}. Instead, Theorem \ref{thm:short} is used to show that solutions to \eqref{eq:landau} with initial data satisfying \eqref{eq:initial_data_moments} satisfy the regularization estimates \eqref{eq:Linfty} and \eqref{eq:Fisher} for a short time. We then apply Theorem \ref{thm:long-cor-GS-GL} with $g_{in} = f(t)$ for any small time $t > 0$ to obtain a global solution. Global existence fundamentally relies upon the $L^1$ moments and the Fisher information, which are the only quantities besides the entropy we know how to propagate for all time.

Let us comment briefly on \eqref{eq:Fisher}. In fact, \eqref{eq:Fisher} is a consequence of the decay of the Fisher information $i(f)$ and the bound $i(f)\in L^1_{loc}(\R^+)$. The $L^1$ bound on the Fisher information is obtained by 
combining pointwise lower bounds for $f$ with the weighted $L^2$ estimates for $\nabla f$. Because the Landau equation does not generate polynomial lower bounds, we need to propagate a pointwise lower bound from the initial data. We note that this assumption has the undesirable side effect of precluding vacuum regions in the initial data. We expect that the pointwise lower can be weakened to a lower bound in an $L^p$ sense, but we do not pursue this here. 
Moreover, to overcome the technical difficulties created by $A[f]$ being unbounded at time $0$ for $p = 3/2$, we take a Stampacchia approach to the maximum principle.  Additionally, 
one could conceivably obtain a version of Theorem \ref{thm:long} for initial data with Gaussian or stretched exponential tails, using the extensive study of weighted $L^2$ estimates in \cite{CarrapatosoMischler}.

It is natural to ask whether the regularization estimates found in Theorem 1.1 hold for all time. Since we are not iterating Theorem 1.1 to produce a global solution, it is unclear in what form the estimates fond in Theorem 1.1 hold for arbitrarily large times. Using known regularization techniques, one can show that for $t >0$ the function $f$ can be bounded as
\begin{equation}\label{eq:long_time_far_from_equilibrium}
    \norm{f(t+1)}_{L^\infty} \le C(\norm{f(t)}_{L^1_{6+}},\norm{f(t)}_{L^3}) \le C(\norm{f(t)}_{L^1_{6+}},i(f(t))). 
\end{equation}
While the Fisher information decays, the $L^1_m$ moments may grow at most linearly in time, raising the possibility of infinite time blow-up. However, taking $m > 16$ in \eqref{eq:initial_data_moments}, one can show that the $L^1_{6^+}$ norm of $f(t)$ is bounded uniformly in time (for a precise reference see \cite[Lemma 2.1]{DesvillettesHeJiang}). As expected, this rules out infinite time blow-up for initial data satisfying \eqref{eq:initial_data_moments}. Previously, the results in \cite{DesvillettesHeJiang} quantitatively ruled out infinite time blow up for $H^1$ initial datum using a novel monotonicity formula connecting the $\dot{H}^1$ norm to the relative entropy. Alternatively, \cite{GolseGualdaniImbertVasseur_PartialRegularity1} qualitatively ruled out infinite time blow up via estimates on the Hausdorff dimension of the singular set of suitable $H$-solutions. See also \cite{GolseImbertVasseur_PartialRegularity2} for related partial regularity results in space-time.

While \eqref{eq:long_time_far_from_equilibrium} shows non-degeneracy for large time, more precise estimates can be obtained to study the convergence to equilibrium. Using the estimates in Section \ref{sec:short_time2} with $p = 2$ modified for $h = f - \mu$ (instead of $f$), where $\mu$ is the Maxwellian corresponding to the normalisation of $f$ \eqref{eq:normalisation}, we extend known convergence to equilibrium results in $L^1$ (see \cite[Theorem 2]{CarrapatosoDesvillettesHe}) and $L^2$ (see \cite[Corollary 1.3]{CarrapatosoMischler}) to $L^\infty$: 
\begin{corollary}\label{cor:long_time}
    For any smooth, rapidly decaying $f$ that solves Landau, and any $m \ge 9/2$,
    \begin{equation}\label{eq:long_time1}
         \norm{f(T) - \mu}_{L^\infty} \lesssim_{m}  \sup_{T-1 < t < T}\norm{f(t) - \mu}_{L^2_m}^4 + \norm{f(t) - \mu}_{L^2_m}^{4/7},
    \end{equation}
    where $\mu$ is the Maxwellian associated to the normalisation \eqref{eq:normalisation}.
    If $f_{in}$ has finite Fisher information and $f_{in} \in L^2(\exp(\lambda \brak{v}^s))$ for some $s \in (0,1/2)$ and $\lambda > 0$, then
    \begin{equation}\label{eq:long_time2}
        \norm{f(t) - \mu}_{L^\infty} \lesssim_{s,\lambda,f_{in}} \exp\left(-\lambda_0 t^{\frac{s}{3}}\right),
    \end{equation}
    for some $\lambda_0 = \lambda_0(s,\lambda, f_{in})$.
\end{corollary}

The regularization estimate \eqref{eq:long_time1} is proved using a De Giorgi iteration that has essentially appeared in the authors' previous work as \cite[Proposition 4.1]{GoldingGualdaniLoher}. A detailed proof can be found in the first author's dissertation as Proposition 4.12 and Proposition 4.13.

The asymptotic behavior in \eqref{eq:long_time2} then follows combining known results on convergence to equilibrium in weighted $L^1$ (specifically \cite[Theorem 2]{CarrapatosoDesvillettesHe}), weighted $L^2$ (specifically \cite[Theorem 1.1]{CarrapatosoMischler}), and the monotonicity of the Fisher information for polynomial localized solutions proved in Theorem \ref{thm:long-cor-GS-GL}) above.

The construction of the global solution in Theorem \ref{thm:long} relied on propagating the polynomial lower bound that is assumed on the initial data for some amount of time. Interestingly, compared to the statement in Corollary \ref{cor:long_time}, Maxwellian behaviour only forms at larger times.

\vspace{0.25cm}

{\bf Open Problems:} Given the rapid recent development of the theory for the homogeneous Landau equation, we collect here some interesting open problems. We emphasize problems whose resolution, in our opinion, would advance the understanding of the collision operator in a manner that can be transferred to the physical inhomogeneous Landau equation.
\begin{itemize}
    \item Regularization estimates for smooth solutions using only the mass, energy, and entropy. That is, do suitably nice solutions to \eqref{eq:landau} satisfy $L^\infty(0,T;L^1_2 \cap L \log L) \to L^\infty(t,T;L^\infty)$ estimates? If so, can one verify the conjectured regularization rate of $\frac{1}{t}$? 
    \item A study of $L^\infty$ solutions, with only locally finite mass: Such solutions appear naturally in the study of local regularity as the limit of suitable rescaling procedures (zooming in about a single point). Additionally, self-similar solutions necessarily have infinite mass. See \cite{Silvestre_Review} for a more in depth discussion.
    \item Existence of additional Lypanov functionals: For example, are there monotone weighted quantities? Are there higher order functionals, such as the derivatives of the Fisher information, which are monotone? 
\end{itemize}

\vspace{0.25cm}
    
{\bf{Outline:}} The rest of the paper is divided as follows: In Section \ref{sec:preliminaries} we introduce several known results without proof, as well as frequently used technical tools (i.e. embeddings and interpolation estimates). In Section \ref{Sec:3}, we recall short time well posedness results and prove Theorem \ref{thm:long-cor-GS-GL}.  The rest of the paper is then devoted to proofs of Theorem \ref{thm:short} and Theorem \ref{thm:long}. The proofs split into two distinct cases, $p=3/2$ and $p > 3/2$. Section \ref{sec:DeGiorgi}, Section \ref{sec:short_time1}, and Section \ref{sec:long_time1} are devoted to the case $p = 3/2$. In Section \ref{sec:DeGiorgi}, we prove novel $L^{3/2} \to L^\infty$ regularization estimates for smooth solutions using the De Giorgi method. In Section \ref{sec:short_time1}, we use an ODE argument to propagate the $L^{3/2}$ norm for a short time. Combined with the $L^\infty$ regularization estimates, additional higher regularity estimates, and a compactness argument we deduce the $p = 3/2$ case of Theorem \ref{thm:short}. In Section \ref{sec:long_time1}, we propagate the lower bound of the initial data using a barrier argument and perform additional weighted estimates to show that under the hypotheses of Theorem \ref{thm:long}, the Fisher information becomes instantaneously finite, which concludes the proof of Theorem \ref{thm:long} for the $p = 3/2$ case. In Section \ref{sec:short_time2}, we show the analogous regularization estimates for $p > 3/2$ and show Theorem \ref{thm:short} for the $p > 3/2$ case. In Section \ref{sec:long_time2}, we show Theorem \ref{thm:long} for the $p > 3/2$ case.

\section{Known results and technical lemmas}\label{sec:preliminaries}

It is known that $L^1$-moments of any order $s > 2$ grow at most linearly in time.
\begin{lemma}[\protect{Propagation of \texorpdfstring{$L^1$}{} moments \cite[Lemma 2.1]{CarrapatosoDesvillettesHe}}]\label{lem:moments}
Let $k > 2$. Fix a non-negative initial datum $f_{in} \in L^1_k \cap L\log L$ satisfying the normalization \eqref{eq:normalisation}. Suppose $f:\R^+ \times \R^3 \rightarrow \R^+$ is any weak solution of \eqref{eq:landau} with initial datum $f_{in}$. 
Then, for each $t \ge 0$,
\begin{equation*}
\int_{\R^3}f(t,v)\brak{v}^k\dd v \leq C(1+t).
\end{equation*}
\end{lemma}

We will often use the uniqueness result due to Fournier:
\begin{theorem}[Uniqueness \protect{\cite[Theorem 2]{Fournier}}]\label{thm:Fournier}
    Suppose $f,\ g:[0,T]\times \R^3 \rightarrow \R^+$ are solutions to \eqref{eq:landau} in the sense of distributions with the same initial datum $f_{in} \in L^1_2$ and both belong to $L^\infty(0,T;L^1_2) \cap L^1(0,T;L^\infty)$. Then, $f(t) = g(t)$ for all $0 < t < T$.
\end{theorem}

Moreover, we collect all bounds on the coefficients in the following lemmas. We generally require mass, energy and entropy of $f$ to be bounded uniformly. If we know that $f \in L^p$ for some $p > 1$, then the entropy bound follows naturally from the $L^p$ bound.
\begin{lemma}\label{lem:entropy-bound}
Suppose $f:\R^3 \rightarrow \R^+$ satisfies for any $p > 1$
    \begin{equation*}
        0 < m_0 \le \int_{\R^3} f \dd v \le M_0, \qquad \int_{\R^3} |v|^2 f \dd v \le E_0, \qquad \int_{\R^3} f^p  \dd v \le K.
    \end{equation*}
    Then, 
    \begin{equation*}
        \int_{\R^3} f \abs{\log(f)} \dd v \le C(m_0,M_0,E_0,K,p).
    \end{equation*}
\end{lemma}
We then state the uniform lower bound satisfied by the diffusion coefficient. Note that by the previous Lemma, we can relax the entropy bound by the $L^p$ bound on $f$, if $p > 1$. This ellipticity result is by now standard in the literature.
\begin{lemma}\label{lem:ellipticity}
    Suppose $f:\R^3 \rightarrow \R^+$ satisfies
    \begin{equation*}
        0 < m_0 \le \int_{\R^3} f \;\dd v \le M_0, \qquad \int_{\R^3} |v|^2 f\; \dd v \le E_0, \qquad \int_{\R^3} f \abs{\log(f)} \;\dd v \le H_0.
    \end{equation*}
    Then, there is a constant $c_0 = c_0(m_0, M_0, E_0, H_0)$ such that
    \begin{equation*}
        A[f](v) \ge \frac{c_0}{\brak{v}^3} I.
    \end{equation*}
\end{lemma}

Moreover, we collect standard upper bounds on the coefficients in the following statement. A reference can for example be found in \cite[Lemma 4.2]{CarrapatosoMischler}.
\begin{lemma}\label{lem:coefficient_bounds}
    For each $1 \le p < 3/2 < q \le \infty$,
    \begin{equation*}
        \norm{A[f]}_{L^\infty} \lesssim \norm{f}_{L^1}^{\frac{2q-3}{3q-3}}\norm{f}_{L^q}^{\frac{q}{3q-3}} \qquad \mathrm{and}\qquad \norm{A[f]}_{L^\infty} \lesssim \norm{f}_{L^p}^{\frac{2p}{3}}\norm{f}_{L^\infty}^{\frac{3-2p}{3}}.
    \end{equation*}
    For each $1 \le p < 3 < q \le \infty$,
    \begin{equation*}
        \norm{\nabla a[f]}_{L^\infty} \lesssim \norm{f}_{L^1}^{\frac{q-3}{3q-3}}\norm{f}_{L^q}^{\frac{2q}{3q-3}} \qquad \text{and} \qquad \norm{\nabla a[f]}_{L^\infty} \lesssim \norm{f}_{L^p}^{\frac{p}{3}}\norm{f}_{L^\infty}^{\frac{3 - p}{3}}.
    \end{equation*}
    For any $3 < p \leq \infty$ and $m > 3\left(p - 1\right)$,
\begin{equation*}
        \abs{\nabla a[f](v)} + \abs{\nabla A[f](v)} \lesssim \brak{v}^{-2}\norm{f}_{L^p_m}.
    \end{equation*}
\end{lemma}

Finally, we repeatedly use weighted interpolation and weighted Sobolev estimates. For the convenience of the reader, we collect these estimates in the following lemma. Its proof may be found in the Appendix \ref{appendix}. 
\begin{lemma}[Sobolev and Interpolation Inequalities] \label{lem:sobolev}
    For each $k \ge 3$, there are $C_1 = C_1(k) > 0$ and $C_2 = C_2(k) > 0$ so that for all $f\in \S(\R^3)$,
    \begin{equation}\label{eq:sobolev}
        \left(\int_{\R^3} \brak{v}^{3k-9} f^6 \dd v\right)^{\frac{1}{3}} + C_1 \int_{\R^3} \brak{v}^{k-5} f^2 \dd v \le C_2 \int_{\R^3} \brak{v}^{k-3}|\nabla f|^2 \dd v.
    \end{equation}
    Moreover, let $1 < p < q < 3p < \infty$. Then, for any $f \in \S(\R^3)$,
    \begin{equation}\label{eq:interpolation-1}
        \int_{\R^3} \brak{v}^k f^{q} \le C(k,p) \left(\int_{\R^3} \brak{v}^m f^{p} \dd v\right)^{\frac{3p-q}{2p}}\left(\int_{\R^3} \brak{v}^{k-3} \abs{\nabla f^{p/2}}^2 \dd v\right)^{\frac{3(q-p)}{2p}},
    \end{equation}
    where
    \begin{equation*}
         m = \frac{2kp - (k-3)(3q-3p)}{3p-q}.
    \end{equation*}
\end{lemma}

\section{Prior well-posedness results and proof of Theorem {\ref{thm:long-cor-GS-GL}}} \label{Sec:3}

First, we state a local well-posedness result for subcritical norms, that is for $p > 3/2$, without any velocity weights. The result is taken from \cite{GoldingLoher}. 
\begin{theorem}[\protect{From \cite[Theorem 1.1, Corollary 1.2, Theorem 1.3]{GoldingLoher}}]\label{thm:golding-loher}
Fix $p > \frac{3}{2}$ and let $m >\frac{9}{2}\frac{p-1}{p - \frac{3}{2}}$. For any given $M$, $H$ and $f_{in} \in L^1_m \cap L^p(\R^3)$ satisfying 
\begin{equation*}
        \norm{f_{in}}_{L^1_m} \le M \qquad \text{and} \qquad \int_{\R^3} f_{in} \abs{\log(f_{in})} \;dv \le H,
\end{equation*}
there exists $0 < T = T(p, m, M, H)$ such that equation \eqref{eq:landau} admits a unique smooth solution on $(0,T) \times \R^3$ such that $\lim_{t \rightarrow 0^+} \norm{f(t) - f_{in}}_{L^p} = 0$.
This solution satisfies, for $t \in (0, T)$, the smoothing estimate
\begin{equation}\label{eq:inst_reg}
\|f(t)\|_{L^\infty} \le C(p,m,M,H)\left(1 + t^{-\beta^*}\right),
\end{equation}
where $\beta^* = \beta^*(p,m) \in (0,1)$ is explicitly computable.
The maximal time $T^*$ for which $f$ can be uniquely continued as a smooth solution is characterized by
        $\lim_{t \nearrow T^*} \norm{f(t)}_{L^p(\R^3)} = +\infty.$
\end{theorem}

Next, we state a very recent result by Guillen-Silvestre \cite{GuillenSilvestre} on the global well-posedness for $C^1$ initial datum with sufficient decay.  The proof is based on showing that the Fisher information $i(f)$ is monotone decreasing as a function of time: 
\begin{theorem}[\protect{From \cite[Theorem 1.2]{GuillenSilvestre}}]\label{thm:guillen-silvestre}
Let $f_{0} :\R^3 \to [0, \infty)$ be an initial datum that is $C^1$ and satisfies $f_0(v) \leq C_0 \exp\left({-\beta \abs{v}^2}\right)$ for some positive parameters $C_0$ and $\beta$. Assume also that $f_0$ has finite Fisher information. 
Then there is a unique $C^1$ solution $f: [0, \infty) \times \R^3 \to [0, \infty)$ to the Landau equation \eqref{eq:landau} with initial datum $f(0, v) = f_0(v)$. For any positive time, this function f is strictly positive, in the Schwartz space, and bounded above by a Maxwellian. The Fisher information $i(f)$ is non-increasing.
\end{theorem}

The proof of Theorem \ref{thm:long-cor-GS-GL} relies on approximating rough slowly, decaying initial data by initial data satisfying the hypotheses of Theorem \ref{thm:guillen-silvestre}. However, because the Fisher information is a nonlinear quantity, we will need to use a particular approximation. The following lemma will be used to construct the approximation:
\begin{lemma}\label{lem:cutoff}
    For any $R > 0$, there is a radial, monotone decreasing $\eta \in C^\infty_c(\R^3)$ that satisfies:
    \begin{equation*}
        0 \le \eta \le 1, \quad \eta(x) = \begin{cases}
            1 & \text{if } x \in B_R(0)\\
            0 & \text{if } x \notin B_{2R}(0),
        \end{cases}
        \quad \abs{\nabla \eta} \le \frac{C\sqrt{\eta}}{R}, \quad \text{and} \quad \abs{\nabla \eta} \le \frac{C\sqrt{1-\eta}}{R},
    \end{equation*}
    where $C$ is a universal constant.
\end{lemma}

The proof is based on the standard construction of a smooth partition of unity in differential geometry. For completeness, we provide a proof in Appendix \ref{appendix2}. We are now ready to prove an analogue of Theorem \ref{thm:guillen-silvestre} for polynomially decaying initial data: 

\begin{proof}[Proof of Theorem \ref{thm:long-cor-GS-GL}]
    We prove the theorem first for $g_{in} \in C^\infty$. For $R > 0$, fix $\eta_R$ as in Lemma \ref{lem:cutoff} and approximate $g_{in}$ by
    \begin{equation*}
        g_{in, R} := g_{in}(v)\eta_R(v) + (1-\eta_R(v))e^{-|v|^2}.
    \end{equation*}
    Note that $g_{in,R} \in C^\infty$, $g_{in,R}$ has Gaussian tails, and $g_{in, R} \to g_{in}$ pointwise as $R \to \infty$. Also,
    \begin{equation}
        \int_{\R^3} \brak{v}^m \abs{g_{in, R} - g_{in}} \dd v \le \int_{\R^3 \setminus B_{R}(0)} \brak{v}^m\left(g_{in}(v) + e^{-|v|^2}\right) \dd v,
    \end{equation}
    which converges to $0$ since $g_{in} \in L^1_m$. A similar argument gives $g_{in, R} \to g_{in}$ in $L^3$ as a consequence of the finite Fisher information and Sobolev's embedding.
    We show next that the Fisher information converges.
    For $|v| \le R$, we have $g_{in, R} = g_{in}$. On the other hand, for $|v| \ge 2R$, $g_{in,R} = e^{-|v|^2}$, so we have
    \begin{equation*}
        \int_{\R^3 \setminus B_{2R}} \frac{\abs{\nabla g_{in, R}}^2}{g_{in, R}} \dd v \le 4 \int_{\R^3 \setminus B_{R}} |v|^2e^{-|v|^2} \dd v,
    \end{equation*}
    which evidently converges to $0$ as $R \to \infty$. It remains to bound terms on the intermediate annulus $A_R$ given by $R \le |v| \le 2R$. We start by computing
    \begin{equation*}
        \nabla g_{in, R} = \nabla g_{in}(v)\eta_R(v) + g_{in}(v)\nabla \eta_R(v) + \nabla\left[(1-\eta_R(v))e^{-|v|^2}\right]
    \end{equation*}
    and so by Cauchy Schwarz inequality,
    \begin{equation*}
    \begin{aligned}
        \int_{A_R} &\frac{\abs{\nabla g_{in, R}}^2}{g_{in, R}} \dd v \\
        &\le \int_{A_R} \frac{\abs{\nabla g_{in}(v)}^2\eta_R(v)^2 +  g_{in}(v)^2\abs{\nabla\eta_R(v)}^2 + \abs{\nabla\left[(1-\eta_R(v))e^{-|v|^2}\right]}^2}{g_{in, R}} \dd v \\
        &=: I_1 + I_2 + I_3.
        \end{aligned}
    \end{equation*}
    The first term we control as
    \begin{equation*}
        I_1 \le \int_{A_R} \frac{\abs{\nabla g_{in}(v)}^2\eta_R(v)^2}{g_{in}\eta_R(v)} \dd v \le \int_{A_R} \frac{\abs{\nabla g_{in}(v)}^2}{g_{in}} \dd v.
    \end{equation*}
    Because $i(g_{in}) < \infty$, the Lebesgue dominated convergence theorem implies $I_1$ converges to $0$ as $R \to \infty$.
    We control the second term using the specific form of $\eta_R$:
    \begin{equation*}
        I_2 \leq \int_{A_R} \frac{g_{in}(v)^2\abs{\nabla\eta_R(v)}^2}{g_{in}\eta_R(v)} \dd v \le CR^{-2} \int_{A_R} g_{in}(v) \dd v \leq CR^{-2},
    \end{equation*}
    which converges to $0$ as $R \to \infty$.
    We control the third term similarly:
     \begin{equation*}
     \begin{aligned}
        I_3 &\le \int_{A_R} \frac{\abs{\nabla\left[(1-\eta_R(v))\exp(-|v|^2)\right]}^2}{(1-\eta_R(v))\exp(-|v|^2)} \dd v \\
        &\le 4\int_{A_R} \abs{v}^2 e^{-\abs{v}^2} \dd v + \int_{A_R} \exp(-|v|^2)\frac{\abs{\nabla\eta_R(v)}^2}{1-\eta_R(v)} \dd v.
        \end{aligned}
    \end{equation*}
    The first term converges to $0$ again as $R \to \infty$, since the exponential dominates any polynomial. The second term also converges to $0$ as $R\to \infty$, where we use the specific form of $1 - \eta_R$ to conclude.
    Thus, we we have shown
    \begin{equation*}
        \lim_{R\to \infty} i(g_{in,R}) = i(g_{in}).
    \end{equation*}

    We now use Theorem \ref{thm:guillen-silvestre} to deduce that for any $R > 1$, there exists $g_R : [0, \infty) \times \R^3 \to [0, \infty)$ solving \eqref{eq:landau} with initial datum $g_{in,R}$, and whose Fisher information is monotonically decreasing. The Sobolev embedding $\dot{H}^1(\R^3) \embeds L^6(\R^3)$ implies that $\norm{g_R(t)}_{L^3} \le C i(g_{in,R})$ for each $t \ge 0$ and $R > 1$. Using the propagation of $L^1$ moments in Lemma \ref{lem:moments}, we apply Theorem \ref{thm:golding-loher} with initial data $g_{R}(t)$ for any $t \in [0,T]$. By uniqueness of $L^1_{loc}(\R^+;L^\infty)$ weak solutions, we conclude that $g_R$ satisfies the smoothing estimate \eqref{eq:inst_reg}: for any $T > 0$ and any $t \in (0,T)$,
    \begin{equation*}
       \norm{g_R(t)}_{L^\infty} \le C\left(\sup_{0<t<T} \norm{g_{R}(t)}_{L^3}, \sup_{0<t<T}\norm{g_{R}(t)}_{L^1_m},m\right)\left(1 + t^{-\beta^*}\right),
    \end{equation*}
    for some explicit $\beta^* \in (0,1)$. Therefore, there is an $g:[0,\infty) \times \R^3 \to [0,\infty)$ such that as $R \to \infty$, $g_R \to g$ weak starly in $L^\infty(0,T;L^3)$, and weakly in $L^1_{loc}(\R^+;L^\infty)$. By well-known higher regularity estimates (used in more detail below as Lemma \ref{lem:higher-regularity}), $g_R \to g$ in $C^k(\Omega)$ for any $k \in \N$ and any $\Omega \subset (0,\infty) \times \R^3$ compact. This is certainly sufficient to conclude $g$ is a smooth solution to \eqref{eq:landau} with initial data $g_{in}$.

    Further, by the Banach-Alaoglu theorem and monotonicity of the $i(g_R(t))$, $\sqrt{g} \in L^\infty(\R^+;\dot{H}^1)$ and 
    \begin{equation*}
        \sup_{t\in \R^+} i(g(t)) = \sup_{t\in \R^+} \norm{\sqrt{g}}_{\dot{H}^1}^2 \le \liminf_{R\to \infty} \sup_{t\in \R^+} \norm{\sqrt{g_R(t)}}_{\dot{H}^1}^2 \le \liminf_{R\to \infty} i(g_{in,R}) = i(g_{in}).
    \end{equation*}
    To summarize, we have shown for any smooth $g_{in}$ with finite Fisher information, there is a unique global-in-time smooth solution $g$ to Landau with initial datum $g_{in}$. This smooth solution satisfies the bounds
    \begin{equation*}
        i(g(t)) \le i(g_{in}),
             \end{equation*}
             and
               \begin{equation*}
\sup_{0 < t < T} t^{\beta^*} \norm{g(t)}_{L^\infty} \le C(\norm{g_{in}}_{L^1_m},i(g_{in}),m,T) \quad\text{for each }T > 0.
    \end{equation*}
    Finally, using uniqueness, we can show that $g$ has decreasing Fisher information. In particular, fix $0 \le s \le t$ and apply this result to $g(s)$ to construct another solution $\tilde g$ with initial data  $\tilde g_{in} = g(s)$. By construction, $i(\tilde g(t)) \leq i(\tilde g_{in}) = i(g(s))$. By the uniqueness from Theorem \ref{thm:Fournier}, $\tilde g(\tau) = g(s + \tau)$ for all $\tau \ge 0$. Taking $\tau = t - s$, we conclude $i(g(t)) = i(\tilde g(t)) \leq i( g(s))$.

    This concludes the proof, if $g_{in}$ is qualitatively smooth. If $g_{in} \in W^{1,1}_{loc}$ is not smooth, then one can approximate $g_{in}$ by
    \begin{equation*}
        g_{in,\delta} = \left(\varphi_\delta \ast \sqrt{g_{in}(\cdot)}\right)^2,
    \end{equation*}
    where $\varphi_\delta$ is a standard mollifier. The seemingly strange choice of approximation is exactly to easily imply $i(g_{in,\delta}) \to i(g_{in})$ as $\delta \to 0^+$. Sending $\delta \to 0^+$ one then concludes in the same manner as when taking $R \to \infty$.
    \qedhere  
\end{proof}

\section{An \texorpdfstring{$\eps$}{epsilon}-regularity criterion for Landau-Coulomb}\label{sec:DeGiorgi}
In this section, we derive a smoothing estimate based on a De Giorgi argument for smooth functions. Variants of the De Giorgi method have been used previously in \cite{AlonsoBaglandsDesvillettesLods, AlonsoBaglandDesvillettesLods_ProdiSerrin, GoldingGualdaniLoher, GoldingLoher, GolseGualdaniImbertVasseur_PartialRegularity1, GolseImbertVasseur_PartialRegularity2} to obtain $L^\infty$ regularization estimates for homogeneous kinetic equations. The novelty of the method presented here is that we use an $L^{3/2}$-based control quantity and use crucially the translation invariance of the estimates in the level set parameter. 
The main result of this section is the following $\eps$-regularity criterion:
\begin{proposition}\label{prop:critical_degiorgi}
    Let $c_0 > 0$ arbitrary and let $f: [0,T]\times \R^3 \to \R^+$ a smooth, rapidly decaying solution to \eqref{eq:landau} satisfying \eqref{eq:normalisation} with $A[f] \ge c_0\brak{v}^{-3}I$. Then, there is a constant $\eps_0 > 0$, depending only on $c_0$, such that if $f$ satisfies
    \begin{equation*}
        \eps(f) := \sup_{0 < \tau < T} \int_{\R^3} \brak{v}^{9/2} (f-K)_+^{3/2} \dd v + \int_0^{T}\int_{\R^3} \brak{v}^{3/2}\abs{\nabla (f-K)_+^{3/4}}^2 \dd v < \eps_0,
    \end{equation*}
    for some threshold $K \ge 0$, then 
    \begin{equation*}
        f(\tau, v) \le C^*(K + 1) + C^*\eps^{2/3} t^{-1} \qquad \text{for each }(\tau,v) \in [t,T]\times \R^3,
    \end{equation*}
    where $C^* > 0$ depends only on $c_0$. 
\end{proposition}
We start by noting that for any smooth, rapidly decaying solution to \eqref{eq:landau}, Lemma \ref{lem:ellipticity} and Lemma \ref{lem:entropy-bound} imply that there is some $c_0 > 0$ such that 
\begin{equation*}
    A[f](v) \ge \frac{c_0}{\brak{v}^3} I.
\end{equation*}

\medskip

\begin{flushleft}
    {\bf \underline{Step 1: Introduction of Level Sets and Weighted Energy Functional}}
\end{flushleft}

For $f$ a solution to \eqref{eq:landau} and for any $\ell \in \R^+$, we denote the level set function
\begin{equation*}
    f_\ell(t,v) := (f(t,v) - \ell)_+ = \max(f(t,v) - \ell, 0).
\end{equation*}
We now define our main weighted energy functional associated to $f$: for $0\le T_1 < T_2 \le T$ and $\ell \in \R^+$,
\begin{equation*}
\begin{aligned}
    \mathcal{E}_\ell(T_1,T_2) :&= \sup_{T_1 < t < T_2} \int_{\R^3} \brak{v}^{9/2} f_\ell^{3/2} \dd v + \int_{T_1}^{T_2}\int_{\R^3} \brak{v}^{3/2} \abs{\nabla f_\ell^{3/4}}^2 \dd v.\\[7pt]
\end{aligned}
\end{equation*}
Our first main goal is showing the following nonlinear inequality: for any $\ell > k \ge 0$ and $0 \le T_1 < T_2 < T_3 \le T$,
\begin{equation}\label{eq:energy_inequality}
\begin{aligned}
    &\mathcal{E}_{\ell}(T_2,T_3) \\
    &\le C \left[1 + \frac{1}{(\ell - k)(T_2 - T_1)} + \frac{\mathcal{E}_k^{1/2}(T_1, T_3)}{(\ell - k)^{3/4}} + \frac{1 + \ell}{(\ell - k)} + \frac{\ell^2}{(\ell - k)^2}\right]\mathcal{E}_k^{5/3}(T_1,T_3),
\end{aligned}
\end{equation}
where $C$ depends on $c_0$. 
To prove \eqref{eq:energy_inequality}, we will show a slightly more precise bound. It will be convenient at times to work with each half of $\mathcal{E}$ separately, therefore we introduce the following auxiliary functionals:
\begin{equation}\label{eq:functionals}
\begin{aligned}
    \mathcal{A}_\ell(T_1,T_2) :&= \sup_{T_1 < t < T_2} \int_{\R^3} \brak{v}^{9/2} f_\ell^{3/2} \dd v\\[7pt]
    \mathcal{B}_\ell(T_1,T_2) :&= \int_{T_1}^{T_2}\int_{\R^3} \brak{v}^{3/2} \abs{\nabla f_\ell^{3/4}}^2 \dd v\dd t. \\[7pt]
\end{aligned}
\end{equation}

\begin{flushleft}
    {\bf \underline{Step 2: Differential Inequality}}
\end{flushleft}

We begin with the computation of the $L^p$ norm of the level set function $f_\ell$. This computation will be used frequently throughout the remainder of the paper. 

\begin{lemma}\label{lem:level_set_inequality}
    For $f:[0,T]\times \R^3 \to \R^+$ a solution to \eqref{eq:landau}, the following differential inequality holds for each $\ell \ge 0$, for each $1 < p < \infty$ and $m \ge 0$,
    \begin{equation}\label{eq:level_set_inequality}
    \begin{aligned}
        \frac{\dd}{\dd t}\int_{\R^3} \brak{v}^m f_\ell^p \dd v &+ \int_{\R^3}\brak{v}^{m-3}\abs{\nabla f_\ell^{p/2}}^2 \dd v \\
        &\le C \int_{\R^3} \brak{v}^m f_\ell^{p+1} \dd v + C(\norm{A[f]}_{L^\infty} + \ell)\int_{\R^3} \brak{v}^m f_\ell^p \dd v\\
            &\qquad + C\ell^2 \int_{\R^3} \brak{v}^m f_\ell^{p-1} \dd v,
    \end{aligned}
    \end{equation}
    where the constant $C$ depends only on $c_0$, $p$, and $m$.
\end{lemma}

Testing \eqref{eq:landau} by $\brak{v}^m f_\ell^{p-1}$, and using $\nabla f_\ell = \chi_{\set{f > \ell}}\nabla f$ and $\partial_t f_\ell = \chi_{\set{f > \ell}} \partial_t f$, we obtain
\begin{align*}
    \frac{1}{p}\frac{\dd }{\dd t}\int_{\R^3} \brak{v}^m f_\ell^p &+ \frac{4(p-1)}{p^2}\int_{\R^3} \brak{v}^m A[f]\nabla f_\ell^{p/2}\cdot \nabla f_\ell^{p/2} \dd v \\
    &\le -\int_{\R^3} f_\ell^{p-1} (\nabla \brak{v}^{m}) \cdot A[f]\nabla f_\ell \dd v\\
    &+ \int_{\R^3} \brak{v}^m \nabla f_\ell^{p-1} \cdot \nabla a[f]f \dd v\\
    &+ \int_{R^3} f_\ell^{p-1}\nabla(\brak{v}^m) \cdot \nabla a[f]f \dd v\\
    &=: I_1 + I_2 + I_3
\end{align*}
We bound $I_1$ using Young's inequality as
\begin{equation*}
\begin{aligned}
    I_1 = -\frac{2}{p}\int_{\R^3} f_\ell^{p/2}(\nabla\brak{v}^m) \cdot A[f]\nabla f_\ell^{p/2} \dd v &\le \frac{p-1}{2p^2} \int_{\R^3} \brak{v}^m A[f]\nabla f_\ell^{p/2} \cdot \nabla f_\ell^{p/2} \dd v\\
        &\qquad + C(p)\norm{A[f]}_{L^\infty}\int_{\R^3} \brak{v}^{m-2} f_\ell^p \dd v.
\end{aligned}
\end{equation*}
Next, by integrating by parts, we rewrite $I_2$, which contains the highest order terms:
\begin{equation}
\begin{aligned}\label{eq:I1}
    I_2 &= -I_3 + \int_{\R^3}\brak{v}^m f_\ell^{p-1}f^2 \dd v - \frac{1}{p}\int_{\R^3} \brak{v}^m \nabla a[f]\cdot \nabla f_\ell^p \dd v\\
        &= -I_3 + \int_{\R^3}\brak{v}^m f_\ell^{p+1} \dd v + 2\ell\int_{\R^3}\brak{v}^m f_\ell^p \dd v + \ell^2\int_{\R^3}\brak{v}^m f_\ell^{p-1} \dd v \\
        &\qquad-  \frac{1}{p} \int_{\R^3}\brak{v}^m f_\ell^p f \dd v +\frac{1}{p} \int_{\R^3} f_\ell^p (\nabla\brak{v}^m) \cdot \nabla a[f] \dd v \\
        &= -I_3 + \frac{p-1}{p}\int_{\R^3}\brak{v}^m f_\ell^{p+1} \dd v + \frac{(2p - 1)\ell}{p}\int_{R^3} \brak{v}^m f_\ell^{p} \dd v\\
        &\qquad + \ell^2 \int_{\R^3}\brak{v}^m f_\ell^{p-1} \dd v + \frac{1}{p}\int_{\R^3} (\nabla \brak{v}^m) \cdot \left[\nabla \cdot \left(A[f]f_\ell^p\right) - A[f]\nabla f_\ell^p\right] \dd v
\end{aligned}
\end{equation}
Therefore, we bound $I_2$ and $I_3$ together as
\begin{align*}
    I_2 + I_3 &\le C\int_{\R^3} \brak{v}^m f_\ell^{p+1} \dd v + C\ell \int_{\R^3} \brak{v}^m f_\ell^p \dd v + C\ell^2 \int_{\R^3} \brak{v}^m f_\ell^{p-1} \dd v\\
        &\qquad -\frac{1}{p}\int_{\R^3} f_\ell^p A[f]:(\nabla^2 \brak{v}^m) \dd v - \frac{2}{p}\int_{\R^3} (\nabla \brak{v}^m)f_\ell^{p/2} \cdot A[f]\nabla f_\ell^{p/2} \dd v\\
        &\le C\int_{\R^3} \brak{v}^m f_\ell^{p+1} \dd v + C(\norm{A[f]}_{L^\infty} + \ell)\int_{\R^3} \brak{v}^m f_\ell^p \dd v \\
        &\qquad+ C\ell^2 \int_{\R^3} \brak{v}^m f_\ell^{p-1} \dd v + \frac{p-1}{2p^2}\int_{\R^3} \brak{v}^{m}A[f]\nabla f_\ell^{p/2} \cdot \nabla f_\ell^{p/2} \dd v,
\end{align*}
where the last inequality follows from Young's, as in  \eqref{eq:I1}.
After reabsorbing terms and using the lower bound for $A[f]$, we find the desired inequality:
\begin{equation*}
\begin{aligned}
    \frac{\dd}{\dd t}\int_{\R^3} &\brak{v}^m f_\ell^p \dd v + \int_{\R^3}\brak{v}^{m-3}\abs{\nabla f_\ell^{p/2}}^2 \dd v \\
    &\le C\int_{\R^3} \brak{v}^m f_\ell^{p+1} \dd v + C(\norm{A[f]}_{L^\infty} + \ell) \int_{\R^3} \brak{v}^m f_\ell^p \dd v+ C\ell^2 \int_{\R^3} \brak{v}^m f_\ell^{p-1} \dd v.
\end{aligned}
\end{equation*}
This concludes the proof of Lemma \ref{lem:level_set_inequality}.

\begin{flushleft}
    \underline{{\bf Step 3: Gain of Integrability}}
\end{flushleft}

In this step, we use  Lemma \ref{lem:level_set_inequality} to obtain a slightly more precise form of the nonlinear inequality \eqref{eq:energy_inequality} relating the energy functionals defined in \eqref{eq:functionals}.

\begin{lemma}\label{lem:gain_of_integrability}
    Under the assumptions of Proposition \ref{prop:critical_degiorgi} and notation introduced in \eqref{eq:functionals} of Step 1, the energy functional $\mathcal{E}$ satisfies for each $0 \le T_1 < T_2 < T_3 \le T$, each $0 \le k < \ell$,
    \begin{equation*}
    \begin{aligned}
    	&\mathcal{E}_{\ell}(T_2,T_3) \\
	&\le C\left[1 + \frac{\mathcal{A}_k^{1/2}(T_1,T_3)}{(\ell - k)^{3/4}} + \frac{1}{(T_2 - T_1)(\ell - k)} + \frac{1 + \ell}{(\ell - k)} + \frac{\ell^2}{(\ell - k)^2}\right] \mathcal{A}_k^{2/3}(T_1,T_3)\mathcal{B}_k(T_1,T_3),
    \end{aligned}
    \end{equation*}
    where the constant $C$ depends only on $c_0$, the coercivity constant of $A[f]$.
\end{lemma}

We begin by integrating \eqref{eq:level_set_inequality} (with $m = 9/2$, $p = 3/2$) over $\tau \in [t_1,t_2] \subset[0,T]$, to obtain
\begin{equation*}
\begin{aligned}
    \int_{\R^3} \brak{v}^{9/2} f_\ell^{3/2}(t_2) \dd v &+ \int_{t_1}^{t_2}\int_{\R^3}\brak{v}^{3/2}\abs{\nabla f_\ell^{3/4}}^2 \dd v \dd \tau \\
    &\le \int_{\R^3} \brak{v}^{9/2} f_\ell^{3/2}(t_1) \dd v+C\int_{t_1}^{t_2}\int_{\R^3} \brak{v}^{9/2} f_\ell^{5/2} \dd v\dd\tau \\
    &+ C\ell\int_{t_1}^{t_2}\int_{\R^3} \brak{v}^{9/2} f_\ell^{3/2} \dd v + C\ell^2 \int_{t_1}^{t_2}\int_{\R^3} \brak{v}^{9/2}f_\ell^{1/2} \dd v\dd \tau \\
    &+ C\int_{t_1}^{t_2}\int_{\R^3} \brak{v}^{9/2}\norm{A[f(\tau)]}_{L^\infty} f_\ell^{3/2} \dd v\dd\tau .
\end{aligned}
\end{equation*}
Taking a supremum over $t_2 \in [T_2, T_3]$, then averaging over $t_1 \in [T_1,T_2]$, and recalling the definition of $\mathcal{E}_\ell$ in \eqref{eq:functionals}, we find
\begin{equation}\label{eq:intergral_level_set_inequality}
\begin{aligned}
    \mathcal{E}_\ell(T_2,T_3) &\le \frac{C}{T_2 - T_1}\int_{T_1}^{T_3}\int_{\R^3} \brak{v}^{9/2} f_\ell^{3/2} \dd v\dd\tau + C\int_{T_1}^{T_3}\int_{\R^3} \brak{v}^{9/2} f_\ell^{5/2} \dd v\dd\tau \\&\qquad +C\ell\int_{T_1}^{T_3}\int_{\R^3} \brak{v}^{9/2} f_\ell^{3/2} \dd v\dd \tau +
    	 C\ell^2 \int_{T_1}^{T_3}\int_{\R^3} \brak{v}^{9/2}f_\ell^{1/2} \dd v\dd \tau \\
      &\qquad + C\int_{T_1}^{T_3}\int_{\R^3} \brak{v}^{9/2}\norm{A[f(\tau)]}_{L^\infty} f_\ell^{3/2} \dd v\dd\tau.
\end{aligned}
\end{equation}
Now, we note a simple consequence of the definition of $f_\ell$: if $0 \le k < \ell$, then
\begin{equation*}
    \set{f_{\ell} > 0} = \set{f_{k} > \ell - k}.
\end{equation*}
Therefore, using Hölder's inequality, Chebychev's inequality for the measure $\dd\nu(v) = \brak{v}^{9/2} \dd v$, for any $0 < \alpha_1 < \alpha_2$, and the fact that $f_{\ell} \leq f_k$ for any $k \leq \ell$, 
\begin{equation}\label{eq:Chebychev}
\begin{aligned}
    \int_{\R^3} \brak{v}^{9/2} f^{\alpha_1}_{\ell} \dd v &\le \nu\left(\set{f_{\ell} > 0}\right)^{1-\frac{\alpha_1}{\alpha_2}}\left(\int_{\R^3} f^{\alpha_2}_{\ell} \dd\nu(v)\right)^{\frac{\alpha_1}{\alpha_2}}\\
        &\le \nu\left(\set{f_{k}^{\alpha_2} > (\ell - k)^{\alpha_2}}\right)^{1-\frac{\alpha_1}{\alpha_2}}\left(\int_{\R^3} f^{\alpha_2}_{\ell} \dd\nu(v)\right)^{\frac{\alpha_1}{\alpha_2}}\\
        &\le (\ell - k)^{\alpha_1 - \alpha_2} \left(\int_{\R^3} \brak{v}^{9/2} f^{\alpha_2}_{k} \dd v\right).
\end{aligned}
\end{equation}
Combining \eqref{eq:intergral_level_set_inequality} and \eqref{eq:Chebychev} (with $\alpha_1 \in \set{1/2,\,3/2}$ and $\alpha_2 = 5/2$ and integrated over $\tau \in [T_1,T_3]$), we have shown for $0 \le k < \ell$
\begin{equation}\label{eq:energy-aux}
\begin{aligned}
    &\mathcal{E}_{\ell}(T_2,T_3) \\
    &\le C\left[1 + \frac{1}{(T_2 - T_1)(\ell - k)} + \frac{\ell}{(\ell - k)} + \frac{\ell^2}{(\ell - k)^2}\right]\int_{T_1}^{T_3}\int_{\R^3} \brak{v}^{9/2} f_k^{5/2} \dd v\dd\tau\\
    &\qquad + C\int_{T_1}^{T_3}\int_{\R^3} \brak{v}^{9/2}\norm{A[f(\tau)]}_{L^\infty} f_{\ell}^{3/2} \dd v\dd\tau.
\end{aligned}
\end{equation}
We now bound the last term, using the identity $f = f_\ell + \min(f, \ell)$, the linearity of $A$, the coefficient bounds in Lemma \ref{lem:coefficient_bounds}, and the embedding $L^2_m\embeds L^1$ for $m > 1$, and \eqref{eq:Chebychev} (with $\alpha_1 \in \set{3/2,2}$ and $\alpha_2 \in \set{13/6,5/2}$):
\begin{equation}\label{eq:coefficient_bound}
\begin{aligned}
    &\norm{A[f]}_{L^\infty}\norm{f_\ell}_{L^{3/2}_{9/2}}^{3/2} \\
    &\le \left(\norm{A[f_\ell]}_{L^\infty} + \norm{A[\min(f,\ell)]}_{L^\infty}\right)\norm{f_\ell}_{L^{3/2}_{9/2}}^{3/2} \\
    &\le C\left(\norm{f_\ell}_{L^1}^{1/3}\norm{f_\ell}_{L^2}^{2/3} + \norm{\min(f,\ell)}_{L^1}^{2/3}\norm{\min(f,\ell)}^{1/3}_{L^\infty}\right)\norm{f_\ell}_{L^{3/2}_{9/2}}^{3/2}\\
    &\le C\norm{f_\ell}_{L^{2}_{9/2}}\norm{f_\ell}_{L^{3/2}_{9/2}}^{3/2} + C\ell^{1/3}\norm{f_\ell}_{L^{3/2}_{9/2}}^{3/2}\\
    &\le \frac{C}{(\ell - k)^{3/4}}\left(\int_{\R^3} \brak{v}^{9/2} f_k^{13/6} \dd v\right)^{\frac{3}{2}} + \frac{C\ell^{1/3}}{(\ell - k)}\left(\int_{\R^3} \brak{v}^{9/2}f_k^{5/2} \dd v\right)
\end{aligned}
\end{equation}
Next, we use Lemma \ref{lem:sobolev} with $q \in \set{5/2,13/6}$, $p = 3/2$, and $9/2$ moments:
\begin{equation*}
\begin{aligned}
    \int_{\R^3} \brak{v}^{9/2} f_k^{5/2} \dd v &\le \left(\int_{\R^3}\brak{v}^{9/2} f_k^{3/2} \dd v\right)^{2/3}\left(\int_{\R^3} \brak{v}^{3/2} \abs{\nabla f_k^{3/4}}^2 \dd v\right),\\[7pt]
    \int_{\R^3} \brak{v}^{9/2} f_k^{13/6} \dd v &\le \left(\int_{\R^3}\brak{v}^{9/2} f_\ell^{3/2} \dd v\right)^{7/9}\left(\int_{\R^3} \brak{v}^{3/2} \abs{\nabla f_k^{3/4}}^2 \dd v\right)^{2/3}.
\end{aligned}
\end{equation*}
Integrating in time and recalling the definitions of $\mathcal{A}_k$ and $\mathcal{B}_k$ in \eqref{eq:functionals}, we find
\begin{equation}\label{eq:interpolation}
\begin{aligned}
    \int_{T_1}^{T_3}\int_{\R^3} \brak{v}^{9/2} f_k^{5/2} \dd v\dd\tau &\le \mathcal{A}_k^{2/3}(T_1,T_3)\mathcal{B}_k(T_1,T_3),\\[7pt]
    \int_{T_1}^{T_3}\left(\int_{\R^3} \brak{v}^{9/2} f_k^{13/6} \dd v\right)^{3/2} \dd \tau &\le \mathcal{A}_k^{7/6}(T_1,T_3)\mathcal{B}_k(T_1,T_3) .  
\end{aligned}
\end{equation}
Combining \eqref{eq:energy-aux}, \eqref{eq:coefficient_bound}, and \eqref{eq:interpolation}, and using $\ell^{1/3} \le C(1 + \ell)$ we estimate $\mathcal{E}_\ell$ as
\begin{equation*}
\begin{aligned}
    \mathcal{E}_{\ell}(T_2,T_3) \le C\Bigg[1 + \frac{\mathcal{A}_k^{1/2}(T_1,T_3)}{(\ell - k)^{3/4}} + \frac{1}{(T_2 - T_1)(\ell - k)}& + \frac{1 + \ell}{(\ell - k)} + \frac{\ell^2}{(\ell - k)^2}\Bigg] \\
    & \times \mathcal{A}_k^{2/3}(T_1,T_3)\mathcal{B}_k(T_1,T_3).
   \end{aligned}
\end{equation*}

This concludes the proof of Lemma \ref{lem:gain_of_integrability}.

\begin{flushleft}
    \underline{{\bf Step 4: De Giorgi Iteration}}
\end{flushleft}

We now conclude the proof of Proposition \ref{prop:critical_degiorgi} by a De Giorgi iteration. We fix $0 < t < T$ the given times and $K > 0$ the given threshold from the statement of Proposition \ref{prop:critical_degiorgi}. Then, for $\eta > 0$ arbitrary, we define our iteration quantities:
\begin{equation*}
    \ell_n = K + (1 - 2^{-n})\eta, \qquad t_n = t(1 - 2^{-n}), \qquad\text{and} \qquad E_n = \mathcal{E}_{\ell_n}(t_n,T).
\end{equation*}
Then, using Lemma \ref{lem:gain_of_integrability}, monotonicity in $n$, and  we have the following recurrence:
\begin{equation*}
\begin{aligned}
    &E_{n+1} = \mathcal{E}_{\ell_{n+1}}(t_{n+1},T)\\
        &\quad\lesssim  \left(1 + \frac{\mathcal{A}_{\ell_n}(t_n,T)^{1/2}}{(\ell_{n+1} - \ell_n)^{3/4}} + \frac{1}{(t_{n+1} - t_n)(\ell_{n+1} - \ell_n)} + \frac{1 + \ell_n}{(\ell_{n+1} - \ell_n)} + \frac{\ell_n^2}{(\ell_{n+1} - \ell_n)^2} \right)\\
        &\qquad \qquad \qquad \qquad \qquad \qquad \qquad \qquad \qquad \qquad \qquad \qquad \quad \qquad\times \mathcal{A}^{2/3}_{\ell_n}(t_n,T)\mathcal{B}_{\ell_n}(t_n,T)\\
        &\quad\le C_1E_n^{5/3} + C_22^{2n}\left(1 + \frac{1}{\eta t} + \frac{E_0^{1/2}}{\eta^{3/4}} + \frac{K}{\eta} + \frac{K^2}{\eta^2} \right)E_n^{5/3},
\end{aligned}    
\end{equation*}
where $C_1,\, C_2 > 1$ are now fixed constants depending only on $c_0$. We attempt to find a barrier sequence of the form $B_n := B^{-n} E_0$ for some $B > 1$. That is, we wish to find values of $\eta$ and $B$ for which
\begin{equation*}
    B_{n+1} \ge C_1B_n^{5/3} + C_22^{2n}\left(1 + \frac{1}{\eta t} + \frac{E_0^{1/2}}{\eta^{3/4}} + \frac{K}{\eta} + \frac{K^2}{\eta^2}\right)B_n^{5/3}.
\end{equation*}
Inserting the definition of $B_n$, this is implied by
\begin{equation*}
    1 \ge C_1 B^{\frac{-2n + 3}{3}}E_0^{2/3} + C_2 B^{\frac{-2n + 3}{3}}2^{2n} E_0^{2/3}\left(1 + \frac{1}{\eta t} + \frac{E_0^{1/2}}{\eta^{3/4}} + \frac{K}{\eta} + \frac{K^2}{\eta^2} \right).
\end{equation*}
We then pick $B = 8$ so that $B^{\frac{2}{3}} = 4$ and we need only show 
\begin{equation*}
    1 \ge 8(C_1 + C_2)E_0^{2/3} + 8C_2E_0^{2/3}\left(\frac{1}{\eta t} + \frac{E_0^{1/2}}{\eta^{3/4}} + \frac{K}{\eta} + \frac{K^2}{\eta^2} \right).
\end{equation*}
Next, we pick $\eps_0$ so that the first term is less than $1/2$ if $E_0 \le \eps_0$. More precisely, we pick
\begin{equation*}
    \eps_0 = \max\left(1,\frac{1}{64(C_1 + C_2)^{3/2}}\right).
\end{equation*}
We finally pick $\eta$ so that each remaining term is less than $1/8$:
\begin{equation*}
\begin{aligned}
    \eta &= \max\left(64C_2,256C_2^{4/3},8C_2^{1/2}\right)\max\left(E_0^{2/3}t^{-1}, E_0^{14/9}, E_0^{2/3}K, E_0^{1/3}K \right)\\
        &\lesssim_{C_2} 1 + K + E_0^{2/3}t^{-1},
\end{aligned}
\end{equation*}
where the implicit constant depends only on $C_2$.
With these choices of $B$, $\eta$, and $\eps_0$, we find that if $E_0 \le \eps_0$, then $B_n \ge E_n$ for all $n$. More precisely, $B_0 = E_0$, and by our choice of parameters $B_n$ satisfies the reverse inequality so that by induction,
\begin{equation*}
\begin{aligned}
    E_{n+1} &\le C_1E_n^{5/3} + C_22^{2n}\left(1 + \frac{1}{\eta t} + \frac{E_0^{1/2}}{\eta^{3/4}} + \frac{1 + K}{\eta} + \frac{K^2}{\eta^2}\right)E_n^{5/3}\\
        &\le C_1B_n^{5/3} + C_22^{2n}\left(1 + \frac{1}{\eta t} + \frac{E_0^{1/2}}{\eta^{3/4}} + \frac{1 + K}{\eta} + \frac{K^2}{\eta^2}\right)B_n^{5/3}
        \le B_{n+1}.
\end{aligned}
\end{equation*}
Therefore, 
\begin{equation*}
    \norm{f \chi_{\set{f \ge K + \eta}}}_{L^\infty(t,T;L^{3/2}_{9/2})} \le \lim_{n\to\infty} E_n \le \lim_{n\to\infty} B_n = 0,
\end{equation*}
and we conclude $f(\tau,v)\le K + \eta$ pointwise a.e. on $[t,T]\times \R^3$. By the choice of $\eta$, and noting $\eps = E_0 < \eps_0 \le 1$, we have shown the desired pointwise estimate:
\begin{equation*}
    f(\tau,v) \le C^*(K + 1) + C^*\eps^{2/3}t^{-1}  \qquad \text{for almost all }\tau,v \in[t, T]\times \R^3.
\end{equation*}

\section{The local in time \texorpdfstring{$L^{3/2}$}{} theory}\label{sec:short_time1}

\subsection{Propagation of \texorpdfstring{$L^{3/2}_m$}{} for smooth data}

We now wish to propagate the $L^{3/2}_m$ norm of a solution $f$ to \eqref{eq:landau} for a time depending only on the initial data $f_{in}$. This will enable us to use Proposition \ref{prop:critical_degiorgi} to gain unconditional smoothing estimates, which can be used to construct solutions to \eqref{eq:landau}. 
The main result of this section is in the following proposition: 

\begin{proposition} \label{prop:propagation}
    Suppose $f_{in} \in L^{3/2}_m$ is Schwartz class and satisfies
    \begin{equation*}
        \int_{\R^3} \brak{v}^{m}f_{in}^{3/2} \dd v \le H \qquad \text{and} \qquad \int_{\R^3} \brak{v}^{m} (f_{in} - K)_+^{3/2} \dd v = \delta,
    \end{equation*}
    for some $m \ge 9/2$ and $K > 0$.
    Then, there is an constant $\delta_1 = \delta_1(H,m) > 0$ such that if $\delta < \delta_1$ and $f:[0, T^*) \times \R^3 \to \R_+$ is any smooth, rapidly decaying solution to \eqref{eq:landau} with initial datum $f_{in}$, there is $T = T(H, m, K) \in (0,1]$ such that
    \begin{equation}\label{eq:propagation1}
    \begin{aligned}
        \sup_{0 < \tau < \min(T,T^*)} &\int_{\R^3} \brak{v}^{m} (f-K)_+^{3/2} \dd v \\
        &+ \int_0^{\min(T,T^*)}\int_{\R^3} \brak{v}^{m-3}\abs{\nabla (f-K)_+^{3/4}}^2 \dd v \dd \tau < 4\delta,
       \end{aligned}
    \end{equation}
    and, moreover,
    \begin{equation}\label{eq:propagation2}
      \begin{aligned}
        \sup_{0 < \tau < \min(T,T^*)}&\int_{\R^3} \brak{v}^{m} f^{3/2} \dd v\\ 
        &+ \int_0^{\min(T,T^*)}\int_{\R^3}\brak{v}^{m-3} \abs{\nabla f^{3/4}}^2 \dd v \dd\tau \le C(m,K,H).
  \end{aligned}
    \end{equation}
\end{proposition}

The proof of Proposition \ref{prop:propagation} is inspired by preceding results of the authors in \cite[Section 3]{GoldingGualdaniLoher}, but our current setting presents significant complications. First, to use Proposition \ref{prop:critical_degiorgi}, we need more than $f\in L^\infty(0,T;L^{3/2}_m)$; we need additional regularity in time. For example, $f\in C(0,T;L^{3/2}_m)$ or $f\in BV(0,T;L^{3/2}_m)$ would be sufficient. To this end, we introduce the following quantities:
\begin{equation}\label{defn:f}
    f_K := (f - K)_+, \quad    y(t) := \int_{\R^3} \brak{v}^m f_K^{3/2} \dd v, \quad F(t) := \int_{\R^3} \brak{v}^{m-3} \abs{\nabla f_K^{3/4}}^2 \dd v. 
\end{equation}
Since $K$ will be fixed, the dependence of $y$ on $K$ is suppressed. Control of $y(t)$ uniformly in time encodes uniform integrability of the profiles $\set{f(t)}_{t \ge 0}$. Our goal now is to propagate smallness of $y$.

This leads, however, to another complication. We are not able to obtain an equation for the evolution of $y(t)$ using just that $f$ solves $\eqref{eq:landau}$. Instead, we must consider the following auxiliary ``bulk'' quantities:
\begin{equation}\label{defn:g}
    g_{2K} := \min(f, 2K), \quad   z(t) := \int_{\R^3} \brak{v}^m g_{2K}^{3/2} \dd v, \quad G(t) := \int_{\R^3} \brak{v}^{m-3} \abs{\nabla g_{2K}^{3/4}}^2 \dd v.
\end{equation}
Note that $f$ and its derivatives decompose as 
$$
f = f_{K} + g_{K} \quad \text{and}\quad \partial_{t,v_i} f = \partial_{t,v_i} f_K + \partial_{t,v_i} g_K.
$$
Although their derivatives have essentially disjoint supports, $f_K$ and $g_K$ do  \emph{NOT} have disjoint support. Therefore, to study the evolution of $y$, the nonlinear reaction term in \eqref{eq:landau} becomes
\begin{equation*}
    f^2 = f_K^2 + 2f_K g_K + g_K^2.
\end{equation*}
We begin by finding the evolution $f_K$, i.e. show that $y(t)$ satisfies a differential inequality, which essentially appeared already in the proof of Proposition \ref{prop:critical_degiorgi}:
\begin{lemma}\label{lem:level_set_ODE}
    Suppose $f:[0,T]\times \R^3 \to \R^+$ is a smooth, rapidly decaying solution to \eqref{eq:landau} with $A[f] \ge c_0\brak{v}^{-3}$. Let $m \geq 9/2$. 
    Then, for any $K > 0$, the corresponding quantities defined in \eqref{defn:f} and \eqref{defn:g} satisfy the differential inequality:
    \begin{equation}\label{eq:ode-y}
        \frac{\dd y(t)}{\dd t} + F(t) \lesssim_{m} F(t)y^{2/3}(t) + (1 + K)y(t) + y^{7/5}(t) + Kz(t),
    \end{equation}
    where the implicit constant $C$ depends only on $m$ and $c_0$.
\end{lemma}
\begin{proof}
    Applying Lemma \ref{lem:level_set_inequality} (with $p = 3/2$, $\ell = K$, and $m \geq 9/2$ fixed), we have the energy estimate:
    \begin{equation*}
    \begin{aligned}
        \frac{\dd}{\dd t}\int_{\R^3} \brak{v}^m f_K^{3/2} \dd v &+ \frac{c_0}{2}\int_{\R^3}\brak{v}^{m-3}\abs{\nabla f_K^{3/4}}^2 \dd v \\
        &\le C \int_{\R^3} \brak{v}^m f_K^{5/2} \dd v + C(\norm{A[f]}_{L^\infty} + K)\int_{\R^3} \brak{v}^m f_K^{3/2} \dd v\\
            &\qquad + CK^2 \int_{\R^3} \brak{v}^m f_K^{1/2} \dd v.
    \end{aligned}
    \end{equation*}
    where the constant $C$ depends only on $c_0$ and $m$. Bounding the highest order term first, we find by Lemma \ref{lem:sobolev} (with $p = 3/2$, $k = 9/2$, and $q = 5/2$):
    \begin{equation*}
        \int_{\R^3} \brak{v}^{m} f_K^{5/2} \dd v \lesssim \left(\int_{\R^3} \brak{v}^{m} f_K^{3/2} \dd v \right)^{2/3}\left(\int_{\R^3} \brak{v}^{m-3} \abs{\nabla f_K^{3/4}}^2 \dd v\right) \lesssim F(t)y^{2/3}.
    \end{equation*}
    Next, we bound the term with $\norm{A[f]}_{L^\infty}$ using Lemma \ref{lem:coefficient_bounds} (with $q = 9/2$), $\norm{f_K}_{L^1} \le 1$, the Sobolev embedding $H^1(\R^3) \embeds L^6(\R^3)$, and Young's inequality:
    \begin{equation*}
    \begin{aligned}
        \norm{A[f]}_{L^\infty}\int_{\R^3} \brak{v}^{m} f_K^{3/2} \dd v &\le \left(\norm{A[f_K]}_{L^\infty} + \norm{A[g_K]}_{L^\infty}\right)y(t)\\
            &\le C\left(\norm{f_K}_{L^1}^{\frac{4}{7}}\norm{f_K}_{L^{9/2}}^{\frac{3}{7}} + \norm{g_K}_{L^1}^{2/3}\norm{g_K}_{L^\infty}^{1/3}\right)y(t)\\
            & \le C\left(\int_{\R^3} \abs{\nabla f_K^{3/4}}^2\right)^{\frac{2}{7}}y  + CK^{1/3}y\\
            &\le \delta F(t) + C\delta^{-\frac{2}{5}}y^{\frac{7}{5}} + CK^{1/3}y.
    \end{aligned}
    \end{equation*}
    Lastly, we bound the lowest order term using H\"older's inequality and Chebychev's inequality for the measure $\dd \nu = \brak{v}^m\dd v$:
    \begin{equation*}
    \begin{aligned}
        \int_{\R^3}& \brak{v}^m f_{K}^{1/2} \dd v \\
        &= \int_{\set{f \ge 2K}} f_{K}^{1/2} \dd \nu + \int_{\set{K \le f < 2K}} f_{K}^{1/2} \dd\nu \\
            &\le \nu\left(\set{f_K \ge K}\right)^{2/3}\left(\int_{\R^3} f_K^{3/2} \dd \nu\right)^{1/3} + \nu\left(\set{g_{2K} \ge K}\right)^{2/3}\left(\int_{\R^3} g_K^{3/2} \dd \nu\right)^{1/3}\\
            &\le K^{-1}\int_{\R^3} \brak{v}^{m} f_{K}^{3/2} \dd v + K^{-1}\int_{\R^3} \brak{v}^m g_{2K}^{3/2} \dd v = K^{-1}(y + z).
    \end{aligned}
    \end{equation*}
    Picking $\delta$ sufficiently small depending only on $c_0$ and $m$ and combining our estimates, we have shown that $y$ satisfies the differential inequality
    \begin{equation*}
        \frac{\dd y}{\dd t} + F(t) \lesssim_{m,c_0} F(t)y^{2/3} + (1 + K)y + y^{7/5} + Kz,
    \end{equation*}
    where the implicit constant $C$ depends only on $m$ and $c_0$.
\end{proof}

Next, we show that the bulk portion of $f$ evolves according to a relatively simple integral inequality. More precisely, in the following lemma, we show that $z(t)$ solves a \emph{constant coefficient, linear equation} with a forcing term, which is nonlinear in $y(t)$. Nonetheless, Gr\"onwall's lemma enables us to bound $z(t)$ in terms of $y(t)$ and reduce \eqref{eq:ode-y} to a scalar, albeit nonlinear, integral inequality.
\begin{lemma}\label{lem:level_set_ODE2}
    Under the same notations and hypotheses as Lemma \ref{lem:level_set_ODE}, there is a constant $C> 0$ depending only on $m$ and $c_0$ such that
   \begin{equation}\label{eq:ode-z}
   \begin{aligned}
     z(t) + \int_0^t G(s)\dd s \leq &z(0) +  2y(0) + C\left(K + 1\right)\int_0^t z(s) \dd s + C K^{3/2}\int_0^t y(s)^{\beta_1} \dd s\\
     &+  C\int_0^t\left(F^{4/7}(s) + K^{1/3} + K\right)y(s) \dd s,
    \end{aligned}
    \end{equation}
     where 
     $\beta_1 \ge 1$ is explicitly computed and depends only on $m$. For $m = 9/2$, $\beta_1 = 1$. 
\end{lemma}

\begin{proof}
    Using that $\partial_t g_{2K} = \chi_{\set{f < 2K}}\partial_t f = \partial_t f - \partial_t f_{2K}$ and $g_{2K} = 2K$ on the support of $f_{2K}$, we compute
    \begin{equation}\label{eq:time_derivative}
    \begin{aligned}
        \frac{\dd}{\dd t}z(t) &= \frac{3}{2}\int_{\R^3} \brak{v}^m g_{2K}^{1/2}\partial_t g_{2K} \dd v\\
            &= \frac{3}{2}\int_{\R^3} \brak{v}^m g_{2K}^{1/2}\partial_t f \dd v - \frac{3(2K)^{1/2}}{2} \frac{\dd}{\dd t}\int_{\R^3} \brak{v}^m  f_{2K} \dd v.
    \end{aligned}
    \end{equation}
    Our main task will be to use that $f$ solves \eqref{eq:landau} to bound the first term on the right hand side of \eqref{eq:time_derivative}, which is the result of taking $\brak{v}^m g_{2K}^{1/2}$ as a test function in the weak formulation of \eqref{eq:landau}: we use $\nabla g_{2K} = \chi_{\set{f < 2K}}\nabla f$ and find
    \begin{equation}\label{eq:bulk_energy1}
    \begin{aligned}
        \int_{\R^3}& \brak{v}^m g_{2K}^{1/2} \partial_t f\\
        &= - \int_{\R^3} \nabla \left(\brak{v}^m g_{2K}^{1/2} \right) \cdot \left(A[f] \nabla f - \nabla a[f] f\right) \dd v\\
        &= -\int_{\R^3} \brak{v}^m \nabla g_{2K}^{1/2} \cdot A[f] \nabla f \dd v - \int_{\R^3} g_{2K}^{1/2}\nabla \brak{v}^m \cdot A[f]\nabla f \dd v\\
        &\quad + \int_{\R^3 } \brak{v}^m\nabla g_{2K}^{1/2} \cdot \nabla a[f] f \dd v + \int_{\R^3} g_{2K}^{1/2}\nabla \brak{v}^m \cdot \nabla a[f]f  \dd v \\
        &= \sum_{j=1}^4 I_j.
    \end{aligned}
    \end{equation}
    The first term, $I_1$, is our usual coercive term, which we rewrite as
    \begin{equation*}
        I_1 = \frac{-8}{9}\int_{\R^3} \brak{v}^m A[f] \nabla g_{2K}^{3/4} \cdot \nabla g_{2K}^{3/4} \dd v.
    \end{equation*}
    We expand the second term using the identity $f = f_{2K} + g_{2K}$ and that $f_{2K}$ and $\nabla g_{2K}$ have disjoint support
    \begin{equation}\label{eq:bulk_energy_term1}
        \begin{aligned}
            I_2 &= -\frac{4}{3}\int_{\R^3} g_{2K}^{3/4}\nabla \brak{v}^m \cdot A[f]\nabla g_{2K}^{3/4} \dd v -  \int_{\R^3}g_{2K}^{1/2} \nabla \brak{v}^m \cdot A[f]\nabla f_{2K} \dd v\\
                &= I_{2,1} - \int_{\R^3} g_{2K}^{1/2}\nabla \brak{v}^m \cdot \left(\nabla \cdot \left(A[f]f_{2K}\right) - \nabla a[f]f_{2K}\right) \dd v\\
                &= I_{2,1} + \int_{\R^3} g_{2K}^{1/2}f_{2K} A[f]:\nabla^2 \brak{v}^m \dd v + \int_{\R^3} g_{2K}^{1/2}f_{2K} \nabla \brak{v}^m \cdot \nabla a[f] \dd v\\
               &= I_{2,1} + I_{2,2} + I_{2,3}
        \end{aligned}
    \end{equation}
    We bound $I_{2,1}$ using Young's inequality: for any $\delta > 0$,
    \begin{equation*}
    \begin{aligned}
        |I_{2,1}| &\le \delta \int_{\R^3} \brak{v}^m A[f]\nabla g_{2K}^{3/4} \cdot g_{2K}^{3/4} \dd v + C\delta^{-1}\int_{\R^3} g_{2K}^{3/2} A[f]\nabla\brak{v}^m\cdot \nabla \brak{v}^m  \dd v\\
            &\le \delta \int_{\R^3} \brak{v}^m A[f]\nabla g_{2K}^{3/4} \cdot g_{2K}^{3/4} \dd v + C\delta^{-1}\int_{\R^3} |A[f]|\brak{v}^{m-2} g_{2K}^{3/2} \dd v
    \end{aligned}
    \end{equation*}
    We bound $I_{2,2}$ and $I_{2,3}$ using H\"older's inequality and Chebychev's inequality as in \eqref{eq:Chebychev} (with $l= 2K$, $k = K$, $\alpha_1 = 1$, and $\alpha_2 = 3/2$):
    \begin{equation*}
    \begin{aligned}
        |I_{2,2}| + |I_{2,3}| &\le CK^{1/2}\left(\norm{A[f]}_{L^\infty} + \norm{\nabla a[f]}_{L^\infty} \right) \int_{\R^3} \brak{v}^m f_{2K} \dd v\\
            &\le C\left(\norm{A[f]}_{L^\infty} + \norm{\nabla a[f]}_{L^\infty} \right) \int_{\R^3} \brak{v}^m f_{K}^{3/2} \dd v
    \end{aligned}
    \end{equation*}
    Next, we further decompose the terms $I_3$ and $I_4$. Integrating by parts and using that $f < 2K$ on $\{f < 2K\}$, so that $g_{2K}f_{2K} = 2K f_{2K}$ on the support of $f_{2K}$, we find
    \begin{equation}\label{eq:bulk_energy_term2}
        \begin{aligned}
            I_3 + I_4 &= \int_{\R^3} \brak{v}^m g_{2K}^{1/2} f^2 \dd v - \int_{\R^3} \brak{v}^m g_{2K}^{1/2} \nabla a[f] \cdot \nabla f \dd v\\
                &= \int_{\R^3} \brak{v}^m g_{2K}^{1/2}f^2 \dd v - \int_{\R^3} \brak{v}^m g_{2K}^{1/2} \nabla a[f] \cdot \nabla g_{2K} \dd v \\
                &\qquad- \int_{\R^3} \brak{v}^m g_{2K}^{1/2} \nabla a[f] \cdot \nabla f_{2K} \dd v\\
                &= \int_{\R^3} \brak{v}^m g_{2K}^{1/2}\left(g_{2K} + f_{2K}\right)f \dd v - \frac{2}{3}\int_{\R^3} \brak{v}^m f g_{2K}^{3/2} \dd v \\
                &\quad+ \frac{2}{3}\int_{\R^3}  g_{2K}^{3/2}\nabla \brak{v}^m \cdot \nabla a[f] \dd v + \int_{\R^3} \brak{v}^m f_{2K}\nabla g_{2K} \cdot \nabla a[f] \dd v \\
                &\quad- \int_{\R^3} \brak{v}^m f f_{2K}g_{2K}^{1/2} + \int_{\R^3} g_{2K}^{1/2} f_{2K}\nabla \brak{v}^m \cdot \nabla a[f] \dd v\\
                &= \frac{1}{3}\int_{\R^3} \brak{v}^m g_{2K}^{3/2}f \dd v + \frac{2}{3}\int_{\R^3}  g_{2K}^{3/2}\nabla \brak{v}^m \cdot \nabla a[f] \dd v \\
                &\quad+ \int_{\R^3} g_{2K}^{1/2} f_{2K}\nabla \brak{v}^m \cdot \nabla a[f] \dd v\\
                &= I_{3,1} + I_{3,2} + I_{3,3}
        \end{aligned}
    \end{equation}
Note that $I_{3,3} = I_{2,3}$ has already been bounded. Now, we bound $I_{3,1}$ using H\"older's inequality and Chebychev's inequality as in \eqref{eq:Chebychev} (with $l= 2K$, $k = K$, $\alpha_1 = 1$, and $\alpha_2 = 3/2$):
\begin{equation*}
    \begin{aligned}
        \int_{\R^3} \brak{v}^m f_{2K}  g_K^{3/2} \dd v &\leq (2K)^{3/2}\int_{\R^3} \brak{v}^m f_{2K}\dd v \leq (2K)^{3/2} K^{-1/2}\int_{\R^3} \brak{v}^m f_K^{3/2}\dd v \\
        &\le CK\int_{\R^3} \brak{v}^m f_K^{3/2}\dd v.
    \end{aligned}
\end{equation*}
We bound $I_{3,2}$ by using $\nabla a[f] = \nabla \cdot A[f]$, integrating by parts, and applying Young's inequality to obtain for any $\delta > 0$,
\begin{equation*}
    \begin{aligned}
        I_{3,2} &=\frac{2}{3}\int_{\R^3} \nabla\brak{v}^m \cdot \left[\nabla \cdot \left(A[f]g_{2K}^{3/2}\right) - A[f]\nabla g_{2K}^{3/2}\right]  \dd v\\ 
        &=- \frac{2}{3} \int_{\R^3}  g_{2K}^{3/2} A[f]:\nabla^2 \brak{v}^m \dd v - \frac{8}{9} \int_{\R^3} g_{2K}^{3/4}\nabla\brak{v}^m \cdot A[f] \nabla g_{2K}^{3/4}  \dd v\\
        &\leq \delta\int_{\R^3} \brak{v}^m A[f] \nabla g_{2K}^{3/4}\cdot \nabla g_{2K}^{3/4} \dd v + C(\delta^{-1} +1)\int_{\R^3} |A[f]|\brak{v}^{m-2} g_{2K}^{3/2} \dd v. 
    \end{aligned}
\end{equation*}
Collecting these estimates, we have shown that upon choosing $\delta$ sufficiently small
\begin{equation*}
    \begin{aligned}
        \int_{\R^3} \brak{v}^m g_{2K}^{1/2}\partial_t f &+ \int_{\R^3} \brak{v}^{m-3} \abs{\nabla g_{2K}^{3/4}}^2 \dd v \\
        &\le CK\int_{\R^3} \brak{v}^m g_{2K}^{3/2} \dd v + C\int_{\R^3} |A[f]|\brak{v}^{m-2} g_{2K}^{3/2} \dd v\\
            &\qquad + C\left(K + \norm{A[f]}_{L^\infty} + \norm{\nabla a[f]}_{L^\infty}\right)\int_{\R^3} \brak{v}^m f_K^{3/2} \dd v.
    \end{aligned}
\end{equation*}
Set $q = \min\left(\frac{4}{3},\frac{m}{m-2}\right)$ so that $\frac{q}{q-1} \ge 4$. Then, by linearity of $A$ and H\"older's inequality, we find
\begin{equation*}
\begin{aligned}
    \int_{\R^3} &|A[f]|\brak{v}^{m-2} g_{2K}^{3/2} \dd v \\
    &= \int_{\R^3} |A[f_K]|\brak{v}^{m-2} g_{2K}^{3/2} \dd v + \int_{\R^3} |A[g_K]|\brak{v}^{m-2} g_{2K}^{3/2} \dd v\\
        &\le \norm{A[f_K]}_{L^{\frac{q}{q-1}}}\left(\int_{\R^3} \brak{v}^{q(m-2)} g_{2K}^{3q/2} \dd v \right)^{1/q} + \norm{A[g_K]}_{L^\infty}\int_{\R^3} \brak{v}^{m} g_{2K}^{3/2} \dd v.
\end{aligned}
\end{equation*}
Then we use Hardy-Littlewood-Sobolev's inequality, see \cite{lieb}, (recalling that the kernel of $A$ belongs to the Lorentz space $L^{3,\infty}(\R^3)$), to bound
\[
	 \norm{A[f_K]}_{L^{\frac{q}{q-1}}} \leq \norm{f_K}_{L^{q'}},
\]
with
\[
	\frac{1}{q'} = 1 + \frac{q-1}{q} - \frac{1}{3} = \frac{5q-3}{3q}.
\]
Combined with the fact that $g_{2K} \le 2K$, an interpolation of Lebesgue spaces, and Young's inequality we arrive at
\begin{equation*}
\begin{aligned}
	\int_{\R^3}& |A[f]|\brak{v}^{m-2} g_{2K}^{3/2} \dd v \\
  	&\le CK^{\frac{3(q-1)}{2q}}\norm{f_K}_{L^{\frac{3q}{5q-3}}}\left(\int_{\R^3} \brak{v}^m g_{2K}^{3/2} \dd v \right)^{1/q} + C\norm{g_K}_{L^1}^{2/3}\norm{g_K}_{L^\infty}^{1/3}\int_{\R^3} \brak{v}^{m} g_{2K}^{3/2} \dd v\\
        &\le CK^{\frac{3(q-1)}{2q}}\norm{f_K}_{L^1}^{\frac{3q - 3}{q}}\norm{f_K}_{L^{3/2}}^{\frac{3-2q}{q}}\left(\int_{\R^3} \brak{v}^m g_{2K}^{3/2} \dd v \right)^{1/q} + CK^{1/3}\int_{\R^3} \brak{v}^{m} g_{2K}^{3/2} \dd v\\
        &\le CK^{\frac{3}{2}}\norm{f_K}_{L^{3/2}}^{\frac{3-2q}{q-1}} + C\left(1 + K\right)\int_{\R^3} \brak{v}^m g_{2K}^{3/2} \dd v,
\end{aligned}
\end{equation*}
or in terms of $F$ and $z$, we may equivalently write 
\begin{equation}\label{eq:evol-tildef}
\begin{aligned}
    \int_{\R^3}& g_{2K}^{1/2}\partial_t f \dd v + G(t) \\
    &\le C\left(1 + K \right)z(t) + C\left(\norm{A[f]}_{L^\infty} + \norm{\nabla a[f]}_{L^\infty} + K\right) y(t) + CK^{3/2} y(t)^{\beta_1},
   \end{aligned}
\end{equation}
where 
$\beta_1 = \frac{3-2q}{q-1}$ for $q = \min\left(\frac{4}{3},\frac{m}{m-2}\right)$.
We now estimate $\norm{A[f]}_{L^\infty}$ with Lemma \ref{lem:coefficient_bounds}, $\norm{(f-2K)_+}_{L^1} \leq 1$, and Lemma \ref{lem:sobolev}:
\begin{equation}\label{eq:big_a_bound}
    \begin{aligned}
        \norm{A[f]}_{L^\infty}& = \norm{A[f_{K}]}_{L^\infty} + \norm{A[g_{K}]}_{L^\infty} \\
        &\leq C\norm{f_{K}}_{L^1}^{4/7}\norm{f_{K}}_{L^{9/2}}^{3/7} + C \norm{g_{K}}_{L^1}^{2/3}\norm{g_{K}}_{L^\infty}^{1/3}\\
            &\leq C\left(\int_{\R^3} \abs{\nabla f_{K}^{3/4}}^2 \dd v\right)^{2/7} +CK^{1/3}.
    \end{aligned}
\end{equation}  
Similarly, we bound $\norm{\nabla a[f]}_{L^\infty}$ as
\begin{equation}\label{eq:little_a_bound}
    \begin{aligned}
        \norm{\nabla a[f]}_{L^\infty} &\le \norm{\nabla a[f_{K}]}_{L^\infty} + \norm{\nabla a[g_{K}]}_{L^\infty} \\
        &\le C\norm{f_{K}}_{L^1}^{1/7}\norm{f_{K}}_{L^{9/2}}^{6/7} + C\norm{g_{K}}_{L^1}^{1/3}\norm{g_{K}}_{L^\infty}^{2/3}\\
            &\le C\left(\int_{\R^3} \abs{\nabla f_{K}^{3/4}}^2 \dd v\right)^{4/7} + CK^{2/3}.
    \end{aligned}
\end{equation}
Combining \eqref{eq:evol-tildef} with the coefficient bounds \eqref{eq:big_a_bound} and \eqref{eq:little_a_bound}, we obtain the following differential inequality:
\begin{equation*}
    \int_{\R^3} g_{2K}^{1/2}\partial_t f \dd v  + G(t) \lesssim_{m,c_0} \left(1 + K\right)z(t) + \left(F^{4/7}(t) + K^{1/3} + K\right)y(t) + K^{3/2} y(t)^{\beta_1}.
\end{equation*}
Using the identity \eqref{eq:time_derivative}, we have shown
\begin{equation*}
\begin{aligned}
    \frac{\dd}{\dd t} z(t) &+ (2K)^{1/2}\frac{\dd}{\dd t}\int_{\R^3} \brak{v}^m f_{2K} \dd v + G(t) \\
    &\le C\left(1 + K\right)z(t) + C\left(F^{4/7}(t) + K^{1/3} + K\right)y(t) + CK^{3/2} y(t)^{\beta_1}.
\end{aligned}
\end{equation*}
Integrating in time, we find the integral inequality:
\begin{equation*}
\begin{aligned}
    z(t) + \int_0^t G(s) \dd s &\le z(0) + 2K^{1/2}\int_{\R^3} \brak{v}^m f_{2K}(0) \dd v + C\left(1 + K\right)\int_0^t z(s) \dd s\\
        &\qquad + C\int_0^t\left(F^{4/7}(s) + K^{1/3} + K\right)y(s) \dd s + CK^{3/2}\int_0^t y(s)^{\beta_1} \dd s.
\end{aligned}
\end{equation*}
Finally, using H\"older's inequality and Chebychev's inequality as in \eqref{eq:Chebychev} (with $l= 2K$, $k = K$, $\alpha_1 = 1$, and $\alpha_2 = 3/2$):
\begin{equation*}
        K^{1/2}\int_{\R^3} \brak{v}^m f_{2K}(0) \dd v \leq \int_{\R^3} \brak{v}^m f_K^{3/2}(0)\dd v = y(0),
\end{equation*} 
which concludes the proof of \eqref{eq:ode-z}.
\end{proof}

From the previous two lemmas, we obtain a system of differential inequalities for the evolution of $y(t)$ and $z(t)$. We now show that for a quantifiable time, the system propagates smallness of $y(t)$ and boundedness of $z$, the bulk of the $L^{3/2}$ norm of $f$. 

\begin{flushleft}
\underline{{\bf Proof of Proposition \ref{prop:propagation}}}    
\end{flushleft}

Fix initial data $f_{in}$, a weight $m$, and bound $H$. Recall from Lemma \ref{lem:ellipticity} that there is a $c_0 = c_0(H)$ such that for any smooth solution $f$ to \eqref{eq:landau} with initial datum $f_{in}$, $c_0\brak{v}^{-3} \le A[f]$.
For $K > 0$ arbitrary and $f:[0,T^*)\times \R^3\to R^+$ the unique Schwartz class solution to \eqref{eq:landau} with data $f_{in}$ define $y(t)$, $F(t)$, $z(t)$, and $G(t)$ via \eqref{defn:f} and \eqref{defn:g}. From Lemma \ref{lem:level_set_ODE}, for any $K > 0$, there is a constant $\alpha = \alpha(m,c_0) > 0$ such that $y$ solves
\begin{equation}\label{eq:ode_y1}
    \frac{\dd y(t)}{\dd t} + \frac{1}{2} F(t) \lesssim_{c_0,m}  -F(t)\left(1 - \alpha y^{2/3}\right) + (1 + K)y(t) + y^{7/5}(t) + Kz(t).
\end{equation}
We now pick $\delta_1 = \delta_1(m,H)$ as the solution to
\begin{equation*}
    1 - \alpha \left(2\delta_1\right)^{2/3} = 0 \qquad \text{or equivalently}\qquad \delta_1 = \frac{1}{2\alpha^{3/2}}.
\end{equation*}
Finally, fix any $\delta < \delta_1$ and $K > 0$ such that
\begin{equation*}
        y(0) = \int_{\R^3} \brak{v}^{9/2} (f_{in} - K)_+^{3/2} \dd v = \delta,
\end{equation*}
which completely fixes $y$, $z$, $F$, and $G$. From here on, inessential constants are allowed to depend on $m$, $H$, and also $K$.
Now, integrating \eqref{eq:ode-z}, and noting that $z(0) \le \norm{f_{in}}_{L^{3/2}_m} = H$, $z$ solves the integral inequality,
\begin{equation*}
     z(t) + \int_0^t G(s)\dd s \leq 2H + C\int_0^t F^{4/7}(s)y(s) + y(s) + y(s)^{\beta_1} \dd s + C\int_0^t z(s) \dd s,
\end{equation*}
where $\beta_1 = \beta_1(m) \ge 1$ is defined in Lemma \ref{lem:level_set_ODE2}. Since $F$ and $y$ are nonnegative, the first integral on the right hand side is increasing in $t$ and the integral form of Gr\"onwall's inequality implies that for each $0 \le t \le 1$,
\begin{equation}\label{eq:z-explicit}
    z(t) \leq C + C\int_0^t F^{4/7}(s)y(s) + y(s) + y(s)^{\beta_1} \dd s.
\end{equation}
Substituting this explicit inequality for $z$ into the right hand side of \eqref{eq:ode_y1}, we obtain
\begin{equation*}
\begin{aligned}
    \frac{\dd y(t)}{\dd t} +\frac{1}{2} F(t)& \le -CF(t)\left(1-\alpha y^{2/3}\right) + Cy(t) + Cy^{7/5}(t) \\
    &\quad+ C\left(\int_0^t F^{4/7}(s)y(s) + y(s) + y(s)^{\beta_1} \dd s + C \right).
\end{aligned}
\end{equation*}
Integrating in time, and using monotonicity to remove the double integrals in time, for each $0 \le t \le 1$,
\begin{equation*}
\begin{aligned}
    y(t) +\frac{1}{2} \int_0^t F(s)\dd s &\le \delta - C\int_0^t F(s)\left(1-\alpha y^{2/3}(s)\right)\dd s + C\int_0^t y(s) + y^{7/5}(s) \dd s\\
        &\qquad + C\int_0^t\int_0^s F^{4/7}(\tau)y(\tau) + y(\tau) + y(\tau)^{\beta_1} \dd\tau\dd s + Ct \\
    &\le \delta - C\int_0^t F(s)\left(1-\alpha y^{2/3}(s)\right) \dd s + C\int_0^t y(s) + y^{7/5}(s) \dd s\\
        &\qquad + C\int_0^t F^{4/7}(s)y(s) + y(s) + y(s)^{\beta_1} \dd s + Ct.
\end{aligned}
\end{equation*}
Using Young's inequality, we arrive at the final form of our differential inequality for $y$:
\begin{equation}\label{eq:ode_y2}
\begin{aligned}
    y(t) + \frac{1}{4}\int_0^t F(s)\dd s \le &\delta + C_1t - C_1\int_0^t F(s)\left(1-\alpha y^{2/3}(s)\right)\dd s \\
    &+ C_1\int_0^t y(s) + y^{\max(\beta_1,7/3)}(s) \dd s,
\end{aligned}
\end{equation}
where $C_1 > 0$ is now a fixed constant. We then fix $T = T(m,H,K)$ as the minimum of $1$ and the unique solution of
\begin{equation*}
    2\delta = \delta + C_1T\left(1 + 2\delta + (2\delta)^{\max(\beta_1,7/3)}\right)
\end{equation*}
We claim that $y(t) \le 2\delta$ for each $0 < t < \min(T,T^*)$. By assumption $y(0) = \delta < 2\delta$. 
Thus, by continuity of $y$, either there holds $y(t) < 2\delta$ for any $t \in (0, \min(T,T^*))$ or there is $t_0 \in (0, \min(T,T^*))$ such that $y(t_0) = 2\delta$. 

Suppose for the sake of contradiction that $t_0 < \min(T,T^*)$. Then, for any $t \in [0, t_0]$ there holds $y(t) \le 2\delta < 2\delta_1$. By the definition of $\delta_1$, we conclude
\begin{equation*}
    - C\int_0^t F(s)\left(1-\alpha y^{2/3}(s)\right)\dd s \le 0, \qquad \text{for each }0\le t \le t_0.
\end{equation*}
Therefore, \eqref{eq:ode_y2} implies that $y$ satisfies
\begin{equation*}
    y(t) \le \delta + C_1t +  C_1\int_0^t y(s) + y^{\max(\beta_1,7/3)}(s) \dd s \le \delta + C_1t\left(1 + 2\delta + (2\delta)^{\max(\beta_1,7/3)}\right),  
\end{equation*}
for each $0 \le t \le t_0$.
Therefore, evaluating at $t = t_0$,
\begin{equation*}
    2\delta \le  \delta + C_1t_0\left(1 + 2\delta + (2\delta)^{\max(\beta_1,7/3)}\right) < \delta + C_1T\left(1 + 2\delta + (2\delta)^{\max(\beta_1,7/3)}\right) \le 2\delta, 
\end{equation*}
a contradiction. It follows that $t_0 \ge \min(T,T^*)$, i.e. $y(t) \le 2\delta$ for $t\in [0,\min(T,T_0))$. From \eqref{eq:ode_y2},
we conclude that for $0 < t < \min(T,T^*)$, there holds
\begin{equation*}\begin{aligned}
    \frac{1}{4}\int_0^t F(s) \dd s &\le \delta + C_1t + C_1\int_0^t y(s) + y^{\max(\beta_1,7/3)}(s) \dd s \\
    &\le \delta + C_1T\left(1 + 2\delta + (2\delta)^{\max(\beta_1,7/3)}\right) \le 2\delta,
    \end{aligned}
\end{equation*}
which implies \eqref{eq:propagation1}. Therefore, returning to \eqref{eq:z-explicit} and \eqref{eq:ode-z}, we use the bounds on $y(t)$ and $F(t)$ to obtain
\begin{equation*}
    z(t) + \int_0^t G(s)\dd s \le C(K,H,m).
\end{equation*}
Finally, using the bounds on $y$ and $z$ we propagate the $L^{3/2}_{9/2}$ and using the bounds on $F$ and $G$ we show the creation of bounds on the energy:
\begin{equation*}
    \norm{f}_{L^\infty(0,T;L^{3/2}_{m})}^{3/2} \le y(t) + z(t) \le C(K,H,m), 
\end{equation*}
and
\begin{equation*}
 	\norm{\nabla f^{3/4}}_{L^2(0,T;L^2_{m-3})} \le \int_0^tF(s) + G(s) \dd s \le C(K,H,m).
\end{equation*}

\subsection{Construction of solution}
In this section, we prove Theorem \ref{thm:short} (with $p = 3/2$) using an approximation argument similar to \cite[Section 5]{GoldingLoher}. A similar argument will be used in Section \ref{sec:short_time2} to construct solutions for $p > 3/2$.

\begin{flushleft}
    \underline{{\bf Approximation}}
\end{flushleft}

Take $f_{in} \in L^{3/2}_{9/2}(\R^3)$ satisfying \eqref{eq:normalisation} and
such that 
\begin{equation*}
        \int_{\R^3} \brak{v}^{9/2}f_{in}^{3/2} \dd v \le H.
\end{equation*}    
Let $\{f_{in, \eps}\}_{\eps > 0}$ be a family of Schwartz class initial datum such that
\begin{enumerate}[i.]
    \item for each  $\eps > 0$, $f_{in, \eps} \in \mathcal S(\R^3)$,
    \item for each  $\eps > 0$, $f_{in, \eps}$ is normalized through \eqref{eq:normalisation}, 
    \item for each  $\eps > 0$, 
    \begin{equation*}
        \begin{aligned}
            \int_{\R^3} \brak{v}^{9/2}f_{in, \eps}^{3/2} \dd v \le 2H,
        \end{aligned}
    \end{equation*}
    \item $\{f_{in, \eps}\}_{\eps > 0} \to f_{in}$ as $\eps \to 0^+$ strongly in  $L^{3/2}_{9/2} \cap L^1_2$ and pointwise almost everywhere.
\end{enumerate}
Then, by Theorem \ref{thm:golding-loher}, there is a family of local-in-time Schwartz class solutions $f_{\varepsilon} : [0, T_\varepsilon) \times \R^3 \to \R_+$ for \eqref{eq:landau} with initial data $f_{in, \eps}$, existing up to time $T_\varepsilon$, which is either infinite, or else
\begin{equation}\label{eq:continuation-Teps}
    \lim_{t \nearrow T_\varepsilon} \norm{f_{\varepsilon}(t)}_{L^\infty} = +\infty.
\end{equation}

\begin{flushleft}
    \underline{{\bf Uniform lower bound on time}}
\end{flushleft}

We first show that the approximations exist on a uniform time interval, i.e. the maximal time of existence $T_\varepsilon$ is bounded below, uniformly in $\eps$. 
Fix $\delta_1 = \delta_1(H)$ as in Proposition \ref{prop:propagation} and $\eps_0 = \eps_0(H)$ as in Proposition \ref{prop:critical_degiorgi}, which are uniform in $\eps$. Then, because $\brak{v}^{9/2}f_{in,\eps}^{3/2} \to \brak{v}^{9/2}f_{in}^{3/2}$ in $L^1(\R^3)$, the family is uniformly integrable, allowing us to choose $K$ independent of $\eps$ so that
\begin{equation}\label{eq:choice of K}
    \int_{\R^3} \brak{v}^{\frac{9}{2}}(f_{in, \eps} - K)_+^{\frac{3}{2}} \dd v < \frac{\eps_0}{6}.
\end{equation}
Applying Proposition \ref{prop:propagation}, there is a $T = T(K, H) \in (0,1]$ independent of $\eps$ such that
\begin{equation*}
        \sup_{0 < \tau < \min(T,T_{\eps})} \int_{\R^3} \brak{v}^{9/2} (f_{\varepsilon}-K)_+^{3/2} \dd v + \int_0^{\min(T,T_{\eps})}\int_{\R^3} \brak{v}^{3/2}\abs{\nabla (f_{\varepsilon}-K)_+^{3/4}}^2 \dd v < \eps_0.
\end{equation*}
By Proposition \ref{prop:critical_degiorgi}, there is $C = C(K,H) > 0$, independent of $\varepsilon$, such that
\begin{equation}\label{eq:Linfty-eps}
    \norm{f_{\varepsilon}(t)}_{L^\infty(\R^3)} \leq C\left(1 + t^{-1}\right), \qquad 0 < t < \min(T, T_\varepsilon).
\end{equation}
Finally, the continuation criterion \eqref{eq:continuation-Teps} for $f_{\eps}$, implies $T_{\eps} > T$ for each $\eps > 0$, providing a uniform in $\eps$ lower bound on the time of existence of $f_{\eps}$. We will construct solutions on this uniform time interval $[0,T]$. 

\begin{flushleft}
    \underline{{\bf Higher regularity estimates}}
\end{flushleft}

To construct a smooth solution $f$ on $[0, T]$ we need strong compactness, which follows from uniform in $\eps$ higher regularity estimates. The precise form of higher regularity estimates we will use are the local estimates contained in the following lemma, which is shown in \cite[Lemma 5.2]{GoldingLoher}:
\begin{lemma}[Higher regularity]\label{lem:higher-regularity}
The family $f_{\varepsilon}: [0, T] \times \R^3 \to \R_+$ constructed above satisfies for each $R > 0$, $0 < t < T$, $k \in \N$, the following regularity estimates
\begin{equation*}
    \norm{f_{\varepsilon}}_{C^k((t, T) \times B_R)} \leq C(H, K, R, t, k).
\end{equation*}
\end{lemma}
The proof is based on combining the coefficient bounds in Lemma \ref{lem:coefficient_bounds} and \eqref{eq:Linfty-eps} to obtain uniform in $\varepsilon$ bounds. In particular, for positive times \eqref{eq:landau} can be viewed as a linear second order parabolic equation with bounded measurable coefficients, for which bootstrapping argument and classical parabolic regularity theory (viz. DeGiorgi-Nash-Moser estimates and Schauder estimates) yield Lemma \ref{lem:higher-regularity}.

\begin{flushleft}
    \underline{{\bf Passing to the limit}}
\end{flushleft}

As a consequence of the uniform estimates in \eqref{eq:Linfty-eps}, Lemma \ref{lem:higher-regularity}, Proposition \ref{prop:propagation}, the Sobolev embedding \eqref{eq:sobolev}, and standard compactness results, we have a limit $f: [0, T] \times \R^3 \to \R_+$, such that $f_{\varepsilon} \to f$ in the following topologies:
\begin{enumerate}[i.]
    \item strongly in $C^k((t, T) \times B_R)$ for each $k \in \N$, for each $R > 0$ and for $0 < t < T$;
    \item pointwise almost everywhere in $[0, T] \times \R^3$;
    \item strongly in $L^1([0, T] \times \R^3)$;\footnote{The decay of the entropy implies and uniform bound on the energy implies strong compactness in $L^1$.}
    \item weak-starly in $L^\infty(0, T; L^{3/2}_{9/2}(\R^3)) \cap L^{5/2}(0,T;L^{5/2}_{9/2})$;
    \item $\nabla f_{\eps}^{3/4} \to \nabla f^{3/4}$ weakly in $L^2(0,T;L^2_{3/2})$;
    \item and $t^{-1} f_{\varepsilon} \to t^{-1}f$ weak-starly in $L^\infty((0, T) \times \R^3)$.
\end{enumerate}
In particular, the limit function $f$ satisfies $f \geq 0$ and by Fatou's lemma $f\in L^\infty(0,T;L^1_2)$. Therefore, once we show that the limit satisfies \eqref{eq:landau}, $f$ will automatically satisfy the normalization \eqref{eq:normalisation}. 
Moreover, from iv. and v., we see that $f$ satisfies the desired estimates:
\begin{equation*}
    \sup_{0 < t < T}\norm{f(t)}_{L^{3/2}_{9/2}} + \sup_{0 < t < T} t \norm{f(t)}_{L^\infty} \le C(H,K)
\end{equation*}
Finally, $f$ also becomes instantaneously smooth, that is $f \in C^\infty((t, T) \times \R^3)$ for each $0 < t < T$.

\begin{flushleft}
    \underline{{\bf Limit equation}}
\end{flushleft}

It remains to show that the limit $f$ satisfies \eqref{eq:landau} on $[0, T]\times \R^3$ with initial data $f_{in}$ and that $f\in C([0,T];L^{3/2}_{9/2})$.
First, we show that $f$ is a classical solution to \eqref{eq:landau} on $(0,T)\times \R^3$. Because $f$ is smooth, it suffices to show that the equation is satisfied in a weak sense. 
Since $f_{\varepsilon}$ solves \eqref{eq:landau}, there holds for any $\varphi \in C_c^\infty((t, T) \times \R^3)$,
\begin{equation}\label{eq:weak-landau-eps}
    \int_t^T \int_{\R^3} \varphi \partial_t f_{\varepsilon} \dd v \dd t = -\int_t^T \int_{\R^3} \nabla \varphi \cdot \left(A[f_{\varepsilon}]\nabla f_{\varepsilon} - \nabla a[f_{\varepsilon}] f_{\varepsilon}\right) \dd v \dd \tau.
\end{equation}
The diffusion coefficient converges to the limit diffusion coefficient, since by Lemma \ref{lem:coefficient_bounds} and H\"older's inequality, we have
\begin{equation*}
    \norm{A[f_{\varepsilon}-f]}_{L^1(t, T; L^\infty(\R^3))} \leq C \norm{f_{\varepsilon}-f}_{L^1(t, T; L^1(\R^3))}^{\frac{2}{3}}\norm{f_{\varepsilon}-f}_{L^1(t, T; L^\infty(\R^3))}^{\frac{1}{3}},
\end{equation*}
which converges to zero as $\varepsilon \to 0^+$ by iv. and v. above. Similarly, for $\nabla a$, again by Lemma \ref{lem:coefficient_bounds} and H\"older's inequality, we have the bound
\begin{equation*}
    \norm{\nabla a[f_{\varepsilon}-f]}_{L^1(t, T; L^\infty(\R^3))} \leq C\norm{f_{\varepsilon}-f}_{L^1(t, T; L^1(\R^3))}^{\frac{1}{3}}\norm{f_{\varepsilon}-f}_{L^1(t, T; L^\infty(\R^3))}^{\frac{2}{3}}, 
\end{equation*}
which also converges to zero as $\eps \to 0^+$ by iv. and v. above. The global convergence of the coefficients together with the local compactness in $C^1$ given by i. implies in particular for any $\varphi \in C_c^\infty((t, T) \times \R^3)$
\begin{equation*}
    \int_t^T \int_{\R^3} \varphi \partial_t f \dd v \dd \tau = -\int_t^T \int_{\R^3} \nabla \varphi \cdot \left(A[f]\nabla f - \nabla a[f] f\right) \dd v \dd \tau.
\end{equation*}
Thus, $f$ solves \eqref{eq:landau} in a distributional sense, and consequently $f$ solves \eqref{eq:landau} in a classical sense on $(0,T)\times \R^3$.

\begin{flushleft}
    \underline{{\bf Convergence to the initial datum}}
\end{flushleft}
To show that the initial datum is obtained in a suitable strong sense, we integrate by parts in \eqref{eq:weak-landau-eps} to obtain that $f_{\varepsilon}$ satisfies for any $\varphi \in C_c^\infty([0, T) \times \R^3)$
\begin{equation*}
      \int_0^T \int_{\R^3} \varphi \partial_t f_{\varepsilon} \dd v \dd t = \int_0^T \int_{\R^3} \nabla^2 \varphi : A[f_{\varepsilon}] f_{\varepsilon} +2 \nabla \varphi\cdot \nabla a[f_{\varepsilon}] f_{\varepsilon} \dd v \dd t.
\end{equation*}
Then we bound with Hölder's inequality, $f_{\varepsilon} \in L^\infty(0, T; L^{3/2}(\R^3))$, Lemma \ref{lem:coefficient_bounds}, and the smoothing estimate \eqref{eq:Linfty-eps}
\begin{align*}
    \int_0^T &\int_{\R^3} \nabla^2 \varphi : A[f_{\varepsilon}] f_{\varepsilon} \dd v \dd t\\
     &\leq \norm{f_{\varepsilon}}_{L^\infty(0, T; L^{3/2}(\R^3))}\int_0^T \left(\int_{\R^3} \abs{\nabla^2 \varphi}^{3}\dd v\right)^{2/3} \norm{A[f_{\varepsilon}](t)}_{L^\infty(\R^3)} \dd t\\
        &\leq C(H,K) \left(\int_0^T \int_{\R^3} \abs{\nabla^2 \varphi}^{3}\dd v \dd t\right)^{1/3}\left(\int_0^T \norm{f_{\varepsilon}(t)}_{L^1(\R^3)}^{2/3}\norm{f_{\varepsilon}(t)}_{L^\infty(\R^3)}^{1/3}\dd t\right)^{\frac{2}{3}}\\
        &\leq C(H,K) \norm{\nabla^2 \varphi}_{L^{3}((0, T) \times \R^3)}\left(\int_0^T t^{-\frac{1}{2}} \dd t\right)^{\frac{2}{3}}\\
        &\leq C(H,K) T^{\frac{1}{3}}\norm{\varphi}_{L^{3}(0, T; W^{2, 3}( \R^3))}.
    \end{align*}
Similarly, we bound with Hölder's inequality, Lemma \ref{lem:coefficient_bounds} and the smoothing estimate \eqref{eq:Linfty-eps}
\begin{equation*}
    \begin{aligned}
        \int_0^T \int_{\R^3}\nabla &\varphi\cdot \nabla a[f_{\varepsilon}] f_{\varepsilon} \dd v \dd t \\
        &\leq \norm{f_{\varepsilon}}_{L^\infty(0, T; L^{3/2}(\R^3))}\int_0^T\norm{\nabla \varphi(t)}_{L^3(\R^3)}\norm{\nabla a[f_{\varepsilon}](t)}_{L^\infty(\R^3)}  \dd t \\
        &\leq C(H,K)\norm{\varphi}_{L^3(0,T;W^{2,3}(\R^3))} \int_0^T\norm{ f_{\varepsilon}(t)}_{L^{\frac{3}{2}}(\R^3)}^{\frac{1}{2}}\norm{ f_{\varepsilon}(t)}_{L^{\infty}(\R^3)}^{\frac{1}{2}}  \dd t\\
        &\leq C(H,K)\norm{\varphi}_{L^3(0,T;W^{2,3}(\R^3))}\norm{f_{\varepsilon}}_{L^\infty(0, T; L^{3/2}(\R^3))}^{\frac{1}{2}}\left(\int_0^T t^{-\frac{3}{4}}\dd t \right)^{\frac{2}{3}}\\
        &\leq C(H,K) T^{\frac{1}{6}} \norm{ \varphi}_{L^{3}(0, T; W^{2, 3}(\R^3))}.
    \end{aligned}
\end{equation*}
By a simple density and duality argument, we have shown that
\begin{equation*}
    \norm{\partial_t f_{\varepsilon}}_{L^{3/2}(0, T; W^{-2, 3/2}(\R^3))} \leq C(H, K).
\end{equation*}
Combined with $\norm{f_{\varepsilon}}_{L^\infty(0, T; L^{3/2})} \leq C(H, K)$, the Aubin-Lions Lemma gives strong compactness of $f_{\varepsilon}$ in $C(0, T; W^{-1, 3/2})$. Thus, up to extracting a subsequence, $f_{\varepsilon}$ converges to $f$ in $C(0, T; W^{-1, 3/2}(\R^3))$. Together with the fact that  $f_{\varepsilon}(0) = f_{in, \eps}$, for each $\eps > 0$, $f$ is continuous up to time $0$ and $f(0) = f_{in}$ in the sense of distributions. But then, using that $f \in L^\infty(0, T; L^{3/2}_{9/2}(\R^3))$, we conclude that $f$ is weakly continuous in time with values in $L^{3/2}_{9/2}$ and $f(0) = f_{in}$ almost everywhere. Moreover, by lower semicontinuity of the norm in the weak topology, we have
\begin{equation}\label{eq:liminf}
    \norm{f_{in}}_{L^{3/2}_{9/2}(\R^3)} \leq \liminf_{t \searrow 0^+}  \norm{f(t)}_{L^{3/2}_{9/2}(\R^3)}.
\end{equation}
To conclude that $\lim_{t \to 0^+} f(t) = f_{in}$ strongly in $L^{3/2}_{9/2}$, it remains to show that
\begin{equation}\label{eq:limsup}
     \limsup_{t \searrow 0^+}  \norm{f(t)}_{L^{3/2}_{9/2}(\R^3)} \leq \norm{f_{in}}_{L^{3/2}_{9/2}(\R^3)}.
\end{equation}
To this end, we recall the energy estimate from Lemma \ref{lem:level_set_ODE} (with $\ell = 0$) implies 
\begin{equation*}
\begin{aligned}
    \frac{\dd}{\dd t} \int_{\R^3} \brak{v}^{9/2} f_{\eps}^{3/2} \dd v  &+\int_{\R^3} \brak{v}^{3/2} \abs{\nabla f_{\eps}^{3/4}}^2 \dd v\\
    &\le C\int_{\R^3} \brak{v}^{9/2}f_{\eps}^{5/2} \dd v + C\norm{A[f_{\eps}]}_{L^\infty}\int_{\R^3} \brak{v}^{9/2} f_{\eps}^{3/2} \dd v.
\end{aligned}
\end{equation*}
We now bound the first term using Lemma \ref{lem:sobolev}:
\begin{align*}
    \int_{\R^3}& \brak{v}^{9/2}f_{\eps}^{5/2} \dd v \\
    &\le  C\int_{\R^3} \brak{v}^{9/2}(f_{\eps}-K)_+ ^{5/2} \dd v+ CK\int_{\R^3} \brak{v}^{9/2}f_{\eps}^{3/2} \dd v\\
        &\le C\left(\int_{\R^3}\brak{v}^{9/2} (f_{\eps}-K)_+^{3/2} \dd v\right)^{2/3}\left(\int_{\R^3} \brak{v}^{3/2} \abs{\nabla (f_{\eps}-K)_+^{3/4}}^2 \dd v\right) + C(K,H).
\end{align*}
As in (\ref{eq:choice of K}), relying on Proposition \ref{prop:propagation}, we pick $K$ uniform in $\varepsilon$ so that 
$$
\left(\int_{\R^3}\brak{v}^{9/2} (f_{\eps}(t)-K)_+^{3/2} \dd v\right)^{2/3} <<1 \qquad \text{for }0 < t < T(K,H).
$$
combined with Lemma \ref{lem:coefficient_bounds} and (\ref{eq:Linfty-eps}), we find
\begin{equation}
\begin{aligned}
    \frac{\dd}{\dd t} \int_{\R^3} \brak{v}^{9/2} f_{\eps}^{3/2} \dd v   \le  C (1 + t^{-1/3}), \quad \text{for } 0 < t < T(K,H),
\end{aligned}
\end{equation}
where the constant is independent of $\eps$.
Integrating in time, we find 
\begin{equation*}
     \sup_{0 < s < t} \left(\int_{\R^3} \brak{v}^{9/2} f_{\eps}^{3/2}(s) \dd v\right) \le \norm{f_{in,\eps}}_{L^{3/2}_{9/2}} + C(H, K)t^{2/3}.
\end{equation*}
By $f_{\eps} \weakstar f$ in $L^{\infty}(0,T;L^{3/2}_{9/2})$, which implies the same convergence on smaller time domains, and lower semi-continuity of the norm, with respect to the weak star topology, 
\begin{equation*}
     \sup_{0 < s < t} \left(\int_{\R^3} \brak{v}^{9/2} f^{3/2}(s) \dd v\right) \le \norm{f_{in}}_{L^{3/2}_{9/2}} + C(H, K)t^{2/3}.
\end{equation*}
Since $f$ is weakly continuous taking values in $L^{3/2}_{9/2}$, we conclude by lower semi-continuity again that for every $0 < t < T$,
\begin{equation*}
     \left(\int_{\R^3} \brak{v}^{9/2} f^{3/2}(t) \dd v\right) \le \norm{f_{in}}_{L^{3/2}_{9/2}} + C(H, K)t^{2/3}.
\end{equation*}
Taking $t \to 0^+$, we conclude that $f \to f_{in}$ in $L^{3/2}_{9/2}$ using (\ref{eq:liminf}). It remains now to show strong continuity of $f$ at any later time $t>0$.

\begin{flushleft}
    \underline{{\bf Continuity}}
\end{flushleft}
To show $f$ is continuous in the strong topology on $L^{3/2}_{9/2}$, we show that the function
\begin{equation*}
    t \mapsto \int_{\R^3} \brak{v}^{9/2} f(t)^{3/2} \dd v
\end{equation*}
is absolutely continuous as a function of time. To rigorously justify the following computations, we work with a truncation of the norm. For any $R > 0$, take $\varphi_R \in C^\infty_c(\R^3)$ a cutoff function that satisfies:
    \begin{equation*}
        0 \le \varphi_R \le 1, \quad \varphi_R(v) = \begin{cases}
            1 & \text{if } v \in B_R(0)\\
            0 & \text{if } v \notin B_{2R}(0),
        \end{cases}
        \quad \abs{\nabla \varphi_R} \le \frac{C}{R}, \quad \text{and} \quad \abs{\nabla^2 \varphi_R} \le \frac{C}{R^2},
    \end{equation*}
Because $f \in C^\infty$ and $f, \, A[f],\, \nabla a[f] \in L^\infty([s,T] \times \R^3)$ for each $s > 0$, we can use $\varphi_R^2 f^{1/2} \brak{v}^{9/2}$ as a test function for \eqref{eq:landau} and integrate in time for $\tau \in [s,t]$. For positive times, integrating by parts several times and using $\nabla \cdot A[f] = \nabla a[f]$, the same computations as in Lemma \ref{lem:level_set_inequality} imply that for each $0 < s < t < T$,
\begin{equation}\label{eq:norm_derivative_equality}
\begin{aligned}
    \int_{\R^3} &\varphi_R^2 \brak{v}^{9/2} f^{3/2}(t) \dd v - \int_{\R^3} \varphi_R^2 \brak{v}^{9/2} f^{3/2}(s) \dd v\\
        &= - \frac{3}{2}\int_s^t\int_{\R^3} \nabla \left(\varphi_R^2\brak{v}^{9/2} f^{1/2}\right) \cdot \left[A[f]\nabla f - \nabla a[f]f\right] \dd v \dd\tau \\
        &= -\frac{4}{3}\int_s^t\int_{\R^3} \varphi_R^2\brak{v}^{9/2} A[f]\nabla f^{3/4} \cdot \nabla f^{3/4} \dd v + \frac{1}{2}\int_{\R^3} \varphi_R^2\brak{v}^{9/2} f^{5/2} \dd v\\
        &\quad - \int_{\R^3} \left(\nabla^2 \left(\varphi_R^2\brak{v}^{9/2}\right) : A[f]\right)f^{3/2} \dd v\\
        &\quad - 4\int_{\R^3} f^{3/4} \nabla \left(\varphi_R^2\brak{v}^{9/2}\right) \cdot A[f]\nabla f^{3/4} \dd v\\
        &=: I_1 + I_2 + I_3 + I_4.
\end{aligned}
\end{equation}
Because $f \in L^{5/2}(0,T;L^{5/2}_{9/2})$, the term $I_2$ is bounded uniformly in $R$. Similarly, $f\in L^{\infty}(0,T;L^{3/2}_{9/2})$ and $A[f] \in L^{\infty}(s,t;L^\infty)$ together imply $I_3$ is bounded uniformly in $R$. By Young's inequality, we bound $I_4$ as
\begin{equation*}
\begin{aligned}
    I_4 &= - 4\int_{\R^3} f^{3/4} \left[\varphi_R \brak{v}^{9/2} \nabla \varphi_R + \varphi_R^2 \nabla \brak{v}^{9/2}\right] \cdot A[f]\nabla f^{3/4} \dd v\\
        &\le -\frac{1}{2} I_1 + \frac{C}{R^2}\int_{\R^3} \norm{A[f]}_{L^\infty}\brak{v}^{9/2}f^{3/2} \dd v \\
        &\quad+ C\int_{\R^3} \brak{v}^{-9/2} \varphi_R^2 f^{3/2} A[f]\nabla \brak{v}^{9/2} \cdot \nabla \brak{v}^{9/2} \dd v\\
        &= -\frac{1}{2} I_1 + I_5 + I_6,
\end{aligned}
\end{equation*}
where $I_5$ and $I_6$ are bounded uniformly in $R$. Thus, \eqref{eq:norm_derivative_equality} and $f \in L^\infty(0,T;L^{3/2}_{9/2})$ imply $I_1$ is bounded uniformly in $R$. Therefore, using the monotone convergence theorem and the uniform in $R$ bounds, taking the limit as $R\to \infty$ in \eqref{eq:norm_derivative_equality}, for every $0 < s < t < T$,
\begin{equation*}
\begin{aligned}
    \int_{\R^3}& \brak{v}^{9/2} f^{3/2}(t) \dd v - \int_{\R^3} \brak{v}^{9/2} f^{3/2}(s) \dd v\\
        &= -\frac{4}{3}\int_s^t\int_{\R^3} \brak{v}^{9/2} A[f]\nabla f^{3/4} \cdot \nabla f^{3/4} \dd v + \frac{1}{2}\int_s^t\int_{\R^3} \brak{v}^{9/2} f^{5/2} \dd v\\
        &\quad - \int_s^t\int_{\R^3} \left(\nabla^2 \brak{v}^{9/2} : A[f]\right)f^{3/2} \dd v - 4\int_s^t\int_{\R^3} f^{3/4} \nabla \brak{v}^{9/2} \cdot A[f]\nabla f^{3/4} \dd v.
\end{aligned}
\end{equation*}
It follows that $t\mapsto \int_{\R^3} \brak{v}^{9/2} f^{3/2}(t) \dd v$ is the integral of an $L^1$ function and thus is absolutely continuous. Recalling $f \in C_w(0,T;L^{3/2}_{9/2})$, this finishes the proof of continuity and Theorem {\ref{thm:short}} for $p=3/2$.

\section{The global in time \texorpdfstring{$L^{3/2}$}{} theory} \label{sec:long_time1}

\subsection{Pointwise lower bounds}

Here we show the solution constructed in Theorem \ref{thm:short} propagates lower bounds, which will be used to control the Fisher information.
The basic idea is that the function $\psi(t,v) = a\exp(-\eta t^{2/3}) \brak{v}^{-k}$ satisfies
\begin{equation*}
    \partial_t \psi \le A[f]:\nabla^2 \psi + f\psi,
\end{equation*}
provided $\eta$ is chosen suitably large, because by combining Lemma \ref{lem:coefficient_bounds} and our a priori estimate $\norm{f}_{L^\infty}\lesssim t^{-1}$, we see $\norm{A[f]}_{L^\infty} \lesssim t^{-1/3}$. Therefore, we expect that if the initial datum is above $\brak{v}^{-k}$, the solution $f$ should remain above $\psi$. However, we are not studying a parabolic equation with bounded coefficients, meaning that the pointwise approach to the maximum principle fails. Instead, to fully justify that $f \ge \psi$, we take the Stampacchia approach here. We begin with a general weighted estimate:

\begin{lemma}\label{lem:minimum_principle_computation}
    Suppose $f:[0,T]\times \R^3 \to \R^+$ is a smooth, rapidly decaying solution to \eqref{eq:landau}. For any smooth function $\ell:(0,T) \to \R^+$ and for any $k > 0$, define $g(t, v) = \brak{v}^k f(t, v)$ and $g_\ell(t,v) = (\ell(t) - g(t,v))_+$. Then, for $n < -3$ sufficiently large, $g_\ell$ satisfies
    \begin{equation*}
        \frac{\dd}{\dd t} \int_{\R^3} \brak{v}^n g_\ell^{3/2} \dd v \le \big(\ell^{\prime}(t) + C\norm{A[f(t)]}_{L^\infty}\ell(t) \big)\left(\int_{\R^3} \brak{v}^{n-2} g_\ell^{1/2} \dd v\right),
    \end{equation*}
    where $C > 0$ depends only on $n$ and $k$.
\end{lemma}

\begin{remark}
    Although the assumption that $f$ is rapidly decaying implies that $g$ is rapidly decaying, $g_\ell$ is \emph{NOT} rapidly decaying. The role of the negative weights, i.e. $\brak{v}^n$ with $n < -3$ is exactly to ensure that all integrals are finite. Note that in the proof we will frequently integrate by parts, and assume that the boundary terms vanish without comment, because, in each instance, one can use the rapid decay of $g$, $\nabla g_\ell$, or $\brak{v}^n$ for large enough $n$ as a  justification. 
\end{remark}

\begin{proof}
    We begin by using $\partial_t g_\ell = (\ell^\prime(t) - \partial_t g)\chi_{\set{g < \ell}}$ to note that
    \begin{equation*}
        \frac{\dd}{\dd t} \int_{\R^3} \brak{v}^n g_\ell^{3/2} \dd v = \frac{3}{2}\ell^\prime(t)\int_{\R^3} \brak{v}^n g_\ell^{1/2} \dd v - \frac{3}{2}\int_{\R^3} \brak{v}^n g_\ell^{1/2}\partial_t g \dd v.
    \end{equation*}
    The first term is already in the correct form, while for the second we need to use that $f$ solves \eqref{eq:landau}. Recalling that $g = \brak{v}^k f$, we write down an equation for $g$:
    \begin{equation*}
    \begin{aligned}
        \partial_t g &= \brak{v}^{k} \nabla \cdot (A[f]\nabla f - \nabla a[f] f)\\
            &= \brak{v}^{k} \nabla \cdot \left(\brak{v}^{-k} A[f]\nabla g + g A[f]\nabla \brak{v}^{-k} - \brak{v}^{-k}\nabla a[f]g\right).
    \end{aligned}
    \end{equation*}
    Multiplying by $-\brak{v}^{n}g_\ell^{1/2}$, we find
    \begin{equation*}
    \begin{aligned}
        -&\int_{\R^3} \brak{v}^{n} g_\ell^{1/2}\partial_t g \dd v \\
        &= \int_{\R^3} \nabla \left(\brak{v}^{n + k} g_\ell^{1/2}\right) \cdot \left(\brak{v}^{-k} A[f]\nabla g + g A[f]\nabla \brak{v}^{-k} - \brak{v}^{-k}\nabla a[f]g\right) \dd v\\
            &= I_1 + I_2 + I_3.
    \end{aligned}
    \end{equation*}
    We further subdivide $I_1$ as
    \begin{equation*}
        I_1 = \int_{\R^3} \brak{v}^n \nabla g_\ell^{1/2} \cdot A[f]\nabla g \dd v + \int_{\R^3} \brak{v}^{-k}\left(\nabla \brak{v}^{n + k}\right) g_\ell^{1/2} \cdot A[f]\nabla g \dd v = I_{1,1} + I_{1,2}.
    \end{equation*}
    Using that $\nabla g_\ell = -\chi_{\set{g < \ell}}\nabla g$, the term $I_{1,1}$ is our usual coercive term:
    \begin{equation*}
        I_{1,1} = -\frac{8}{9}\int_{\R^3} \brak{v}^n A[f]\nabla g_\ell^{3/4} \cdot \nabla g_\ell^{3/4} \dd v.
    \end{equation*}
    Using that $\nabla \brak{v}^k = k\brak{v}^{k-2}v$, the term $I_{1,2}$ simplifies as:
    \begin{equation*}
    \begin{aligned}
        I_{1,2} &= -\frac{4(n + k)}{3n}\int_{\R^3} g_\ell^{3/4}\nabla \brak{v}^{n} \cdot A[f]\nabla g_\ell^{3/4} \dd v.
    \end{aligned}
    \end{equation*}
    For $I_2$, we integrate by parts and further subdivide as
    \begin{equation*}
    \begin{aligned}
        I_2 &= -\int_{\R^3} \brak{v}^{n + k} g_\ell^{1/2} \nabla g \cdot A[f]\nabla \brak{v}^{-k} \dd v - \int_{\R^3} \brak{v}^{n + k} g_\ell^{1/2}g \left(A[f]:\nabla^2 \brak{v}^{-k}\right) \dd v\\
            &\qquad- \int_{\R^3} \brak{v}^{n + k} g_\ell^{1/2}g \nabla a[f] \cdot \nabla \brak{v}^{-k} \dd v = I_{2,1} + I_{2,2} + I_{2,3}.
    \end{aligned}
    \end{equation*}
    Since we will use some cancellations between $I_1$, $I_2$ and $I_3$, before bounding terms, we integrate by parts and further subdivide $I_3$ as
    \begin{equation*}
    \begin{aligned}
        I_3 &= \int_{\R^3} \brak{v}^{n} g_\ell^{1/2} \nabla g\cdot \nabla a[f] \dd v - \int_{\R^3} \brak{v}^n g_\ell^{1/2}gf \dd v\\
        &\quad + \int_{\R^3} \brak{v}^{n + k} g_\ell^{1/2}g \nabla a[f] \cdot \nabla \brak{v}^{-k} \dd v\\
            &= I_{3,1} + I_{3,2} + I_{3,3}.
    \end{aligned}
    \end{equation*}
    Note that $I_{2,3} = - I_{3,3}$ and $I_{3,2}$, which contains the most singular terms, has a favorable sign, $I_{3,2} \le 0$. Next, we expand $I_{2,1}$ as
    \begin{equation*}
        I_{2,1} = \frac{-4k}{3n}\int_{\R^3} g_\ell^{3/4}\nabla \brak{v}^{n} \cdot A[f]\nabla g_\ell^{3/4} \dd v.
    \end{equation*}
    Then, for $I_{2,2}$, we compute explicitly
    \begin{equation*}
    \begin{aligned}
        I_{2,2} &= -k(k +2)\int_{\R^3}\brak{v}^{n-4} g_\ell^{1/2}g \left(A[f]: v\tens v\right) \dd v + k\int_{\R^3} \brak{v}^{n-2} gg_\ell^{1/2} \mathrm{tr}\left(A[f]\right) \dd v\\
            &\le k\int_{\R^3} \brak{v}^{n-2} \left(\ell(t) g_\ell^{1/2} - g_\ell^{3/2}\right) \mathrm{tr}\left(A[f]\right) \dd v\\
            &\le k\ell(t)\int_{\R^3} \brak{v}^{n-2} g_\ell^{1/2}\mathrm{tr}\left(A[f]\right) \dd v,
    \end{aligned}
    \end{equation*}
    where only one term does not have a favorable sign. We rewrite $I_{3,1}$ as
    \begin{equation*}
    \begin{aligned}
        I_{3,1} &= -\frac{2}{3}\int_{\R^3} \brak{v}^n \nabla g_\ell^{3/2} \cdot \nabla a[f] \dd v = \frac{2}{3}\int_{\R^3} g_\ell^{3/2}\nabla \brak{v}^n  \cdot \nabla a[f] \dd v - \frac{2}{3}\int_{\R^3} \brak{v}^n g_\ell^{3/2} f \dd v\\
            &\le \frac{2}{3}\int_{\R^3} g_\ell^{3/4}\nabla \brak{v}^n  \cdot \left(\nabla \cdot( A[f]g_{\ell}^{3/4}) - A[f]\nabla g_\ell^{3/4}\right) \dd v\\
            &= -\frac{4}{3} \int_{\R^3} g_\ell^{3/4}\nabla \brak{v}^n \cdot A[f]\nabla g_\ell^{3/4} - \frac{2}{3}\int_{\R^3} g_\ell^{3/2}\left(A[f]:\nabla^2 \brak{v}^n\right) \dd v,
    \end{aligned}
    \end{equation*}
    where again the most singular term has a favorable sign. Expanding the last term further and using $g \le \ell(t)$ on the support of $g_\ell$
    \begin{equation*}
    \begin{aligned}
        - \frac{2}{3}\int_{\R^3} g_\ell^{3/2}\left(A[f]:\nabla^2 \brak{v}^n\right) \dd v &= - \frac{2n(n-2)}{3}\int_{\R^3} \brak{v}^{n-4} g_\ell^{3/2}\left(A[f]:v\tens v\right) \dd v \\
        &\quad- \frac{2n}{3}\int_{\R^3}\brak{v}^{n-2} g_\ell^{3/2}\left(\mathrm{tr}A[f]\right) \dd v\\
            &\le - \frac{2n}{3}\ell(t)\int_{\R^3} \brak{v}^{n-2} g_\ell^{1/2}\mathrm{tr}\left(A[f]\right) \dd v,
    \end{aligned}
    \end{equation*}
    where the first term has a favorable sign.
    Therefore, summing over $I_{i,j}$, we summarize our estimates as:
    \begin{equation*}
    \begin{aligned}
        -\int_{\R^3} \brak{v}^{n} g_\ell^{1/2}\partial_t g \dd v &+ \frac{8}{9}\int_{\R^3} \brak{v}^n A[f]\nabla g_\ell^{3/4} \cdot \nabla g_\ell^{3/4} \dd v\\
            &\le - \frac{8(n + k)}{3n}\int_{\R^3} g_\ell^{3/4} A[f]\nabla\brak{v}^{n} \cdot \nabla g_\ell^{3/4} \dd v\\
            &+ \frac{3k-2n}{3}\ell(t)\int_{\R^3} \brak{v}^{n-2} g_\ell^{1/2}\mathrm{tr}\left(A[f]\right) \dd v.
    \end{aligned}
    \end{equation*}
    Now, we bound the terms on the right hand side. Using $g \le \ell(t)$ on the support of $g_\ell$, for any $\delta > 0$,
    \begin{equation*}
    \begin{aligned}
        \int_{\R^3} &g_\ell^{3/4}A[f]\nabla \brak{v}^{n} \cdot \nabla g_\ell^{3/4} \dd v \\
        &\le \delta \int_{\R^3}  \brak{v}^n A[f]\nabla g_\ell^{3/4} \cdot \nabla g_\ell^{3/4} \dd v + \frac{1}{4\delta}\int_{\R^3} g_\ell^{3/2} \brak{v}^{-n}A[f]\nabla \brak{v}^{n} \cdot \nabla \brak{v}^{n} \dd v\\
            &\le \delta \int_{\R^3} \brak{v}^n A[f]\nabla g_\ell^{3/4} \cdot \nabla g_\ell^{3/4} \dd v + \frac{n^2}{4\delta}\int_{\R^3} \brak{v}^{n-4} g_\ell^{3/2}A[f]v\cdot v \dd v\\
            &\le \delta \int_{\R^3} \brak{v}^n A[f]\nabla g_\ell^{3/4} \cdot \nabla g_\ell^{3/4} \dd v + \frac{n^2\norm{A[f(t)]}_{L^\infty}\ell(t)}{4\delta}\int_{\R^3} \brak{v}^{n-2} g_\ell^{1/2} \dd v.
    \end{aligned}
    \end{equation*}
    Taking $\delta$ sufficiently small depending only on $n$ and $k$, we conclude that
    \begin{equation*}
        \frac{\dd}{\dd t} \int_{\R^3} \brak{v}^n g_\ell^{3/2} \dd v \le \left(\ell^{\prime}(t) + C(n,k)\norm{A[f(t)]}_{L^\infty}\ell(t) \right)\left(\int_{\R^3} \brak{v}^{n-2} g_\ell^{1/2} \dd v\right),
    \end{equation*}
    as desired.
\end{proof}

\begin{lemma}\label{lem:minimum_principle}
    Suppose $f:[0,T]\times \R^3 \to \R^+$ is a smooth, rapidly decaying solution to \eqref{eq:landau} such that $\norm{A[f(t)]}_{L^\infty} \le Mt^{-1/3}$ for $t\in (0,T]$. Then, for each $a > 0$, $k > 0$, and $n < -3$ sufficiently negative, there is an $\eta = \eta(n,k,M) > 0$ sufficiently large so that $f$ satisfies on $(0,T)$,
    \begin{equation*}
        \frac{\dd}{\dd t} \int_{\R^3} \brak{v}^n \left(a\exp\left(-\eta t^{2/3}\right) - f\brak{v}^k\right)^{3/2}_+ \dd v \le 0.
    \end{equation*}
\end{lemma}

\begin{proof}
    Fix $k > 0$ and $n < -3$ arbitrary. Then, set $\ell(t) = a\exp(-\eta t^{2/3})$ for $\eta > 0$ to be determined, $g = \brak{v}^k f$, and $g_\ell = (\ell - g)_+$. We compute
    \begin{equation*}
        \ell^\prime(t) = -\frac{2a\eta}{3t^{1/3}}\exp(-\eta t^{2/3}) = -\frac{2\eta}{3t^{1/3}}\ell(t).
    \end{equation*}
    Therefore, by Lemma \ref{lem:minimum_principle_computation}, we find
    \begin{equation*}
    \begin{aligned}
        \frac{\dd}{\dd t} \int_{\R^3} \brak{v}^n g_\ell^{3/2} \dd v &\le \big(\ell^{\prime}(t) + C(n,k)\norm{A[f(t)]}_{L^\infty}\ell(t) \big)\left(\int_{\R^3} \brak{v}^{n-2} g_\ell^{1/2} \dd v\right)\\
            &\le \big(C(n,k)M - 2\eta/3\big)\frac{\ell(t)}{t^{1/3}}\left(\int_{\R^3} \brak{v}^{n-2} g_\ell^{1/2} \dd v\right).
    \end{aligned}
    \end{equation*}
    Picking $\eta \ge 3C(n,k)M/2$, we conclude that the right hand side is negative.
\end{proof}

\subsection{Bound on the Fisher information and proof of Theorem \ref{thm:long}}

Since we do not have uniqueness, we follow the proof of Theorem \ref{thm:short}. Suppose that the initial datum $f_{in} \in L^1_m \cap L^{3/2}_{m_0}$ for some $m > 6$ and $m_0 > 6$ and $f_{in} \ge a\brak{v}^{-k}$ for some $a > 0$ and $k > 0$ as in the statement of Theorem \ref{thm:long}. Then, let $f_{in, \eps} \in \mathcal{S}(\R^3)$ be such that $f_{in, \eps} \to f_{in}$ in $L^1_m \cap L^{3/2}_{m_0}$. By Theorem \ref{thm:short}, $f_{in, \eps}$ give rise to $f_{\eps}:[0,T]\times \R^3 \to \R^+$, Schwartz class solutions to \eqref{eq:landau} on $[0,T]$, for some $0 < T \le 1$ independent of $\eps$.

From the proof of Theorem \ref{thm:short}, we have the uniform bounds
\begin{equation}\label{eq:decay-coercive}
\begin{aligned}
    \sup_{0 < t < T} \norm{f_{\eps}(t)}_{L^{3/2}_{m_0}}^{3/2} + &\int_0^{T}\int_{\R^3} \brak{v}^{m_0 - 3}\abs{\nabla f_{\eps}^{3/4}}^2 \dd v \dd t \\
    &+ \sup_{0 < t < T} t\norm{f_{\eps}(t)}_{L^\infty} \le C(f_{in}, m_0, T).
\end{aligned}
\end{equation}

Moreover, since $\norm{f_{\eps}(t)}_{L^\infty} \lesssim t^{-1}$, by Lemma \ref{lem:coefficient_bounds}, $\norm{A[f_{\eps}(t)]}_{L^\infty} \lesssim t^{-1/3}$. So, for $n < - 3$ sufficiently negative, we can apply Lemma \ref{lem:minimum_principle} and conclude for some $\eta = \eta(n,k,f_{in},T) > 0$, independent of $\eps$:
\begin{equation*}
    \frac{\dd}{\dd t} \int_{\R^3} \brak{v}^n \left(a\exp\left(-\eta t^{2/3}\right) - f_{\eps}\brak{v}^k\right)^{3/2}_+ \dd v \le 0, \qquad \text{on }(0,T).
\end{equation*}
Integrating in time, we see that for each $0 < t < T$,
\begin{equation*}
    \int_{\R^3} \brak{v}^n \left(a\exp\left(-\eta t^{2/3}\right) - f_{\eps}(t,v)\brak{v}^k\right)^{3/2}_+ \dd v \le \int_{\R^3} \brak{v}^n \left(a - f_{in, \eps}\brak{v}^k\right)^{3/2}_+ \dd v.
\end{equation*}
Noting that $a \le f_{in}\brak{v}^k$ and $f_{in, \eps} \to f_{in}$ in $L^1$, the right hand side converges to $0$ as $\eps \to 0^+$. Since $f_{\eps} \to f$ in $L^1([0,T]\times \R^3)$ as $\eps \to 0^+$, we find for almost every $t \in (0,T)$,
\begin{equation*}
    \int_{\R^3} \brak{v}^n \left(a\exp\left(-\eta t^{2/3}\right) - f(t,v)\brak{v}^k\right)^{3/2}_+ \dd v \le 0.
\end{equation*}
Since $f$ is smooth, $f\brak{v}^k \ge a\exp\left(-\eta t^{2/3}\right)$ for every $t \in (0,T)$.
Combined with \eqref{eq:decay-coercive}, we have shown the following upper bound on the Fisher information:
\begin{equation*}
\begin{aligned}
   \int_0^T i(f(t)\dd t &= \int_0^T \int_{\R^3} \frac{\abs{\nabla f(t)}^2}{\sqrt f(t)\sqrt f(t)} \dd v\dd t \\
   &\le C(a,T,k)\int_0^T\int_{\R^3} \brak{v}^{\frac{k}{2}}\abs{\nabla f^{3/4}(t)}^2 \dd v \dd t\le C(a,T,k,f_{in},m_0),
\end{aligned}
\end{equation*}
provided $k \le 2m_0 - 6$. Thus $i(f)$ is bounded almost everywhere in $[0, T]$.
Finally, since $L^1$ moments are propagated by Lemma \ref{lem:moments}, $f(t) \in L^1_m$, where $m > 6$. Thus, applying Theorem \ref{thm:long-cor-GS-GL} with initial data given by the profile $f(s)$, for each $0 < s < T$, there is unique global-in-time smooth solution $g_s:[0,\infty) \times \R^3\to \R^3$ with decreasing Fisher information and $\norm{g_s(t)}_{L^\infty} < \infty$ for each $t$.
Denote by $\tilde f$ the concatenation:
\begin{equation*}
    \tilde f(t) = \begin{cases} f(t) &\text{if }0 < t < T/2,\\ g_{T/2}(t -T/2)&\text{otherwise}.\end{cases}
\end{equation*}
By Fournier's uniqueness result from Theorem \ref{thm:Fournier}, we conclude $g_s(t) = f(t - s)$ for each $0 < s < t < T$. Consequently, $\tilde f$ is the desired global-in-time smooth solution. Moreover, $\tilde f$ has decreasing Fisher information and $L^1$ in time: for each $0 < t < T$
\begin{equation*}
    i(f)(t) \leq C \left(1+t^{-1}\right).
\end{equation*}

\section{The subcritical local in time theory}\label{sec:short_time2}

This section is devoted to the proof of Theorem \ref{thm:short} in the subcritical case, i.e. when $p > 3/2$. The general strategy of proof is similar to the critical $p = 3/2$ case, but the proofs are comparatively simpler partially due to uniform bounds on the diffusion coefficient $A[f]$: the analogue of the $\eps$-regularity criterion in Proposition \ref{prop:critical_degiorgi} is a standard $L^p \to L^\infty$ regularization estimate enabling a significant simplification in of the propagation of weighted norms in Proposition \ref{prop:propagation}. We remark that local well-posedness for initial datum in $L^p \cap L^1_m$ for $p > 3/2$ was considered in \cite{GoldingLoher}, but the methods break down as $p \to 3/2^+$ in the framework of \cite{GoldingLoher}. The alternative weighted $L^p$ framework offers the advantage of simplicity and does not break down as $p$ tends to either $\infty$ or $3/2$. Due to these considerations, we show full proofs of necessary a priori estimates and omit substantial details of the construction of the solution, which mimics Section \ref{sec:short_time1} here or \cite[Section 5]{GoldingLoher}.

By contrast with the $p=3/2$ case, when $p > 3/2$ we can derive a uniqueness result based on the $L^1(0,T;L^\infty)$ uniqueness result of Fournier \cite{Fournier}. The uniqueness follows the strategy in \cite[Section 6]{GoldingLoher}, which was based on the Prodi-Serrin criteria and the novel proof of an $\eps$-Poincar\'{e} inequality by Alonso et al. in \cite{AlonsoBaglandDesvillettesLods_ProdiSerrin}. 

\subsection{Subcritical De Giorgi}

\begin{proposition}\label{prop:subcritical_degiorgi}
Suppose $f:[0,T]\times \R^3 \rightarrow \R^+$ is a smooth, rapidly decaying solution to \eqref{eq:landau} where $\lambda\brak{v}^{-3} \le A[f] \le \Lambda$. For any $m \ge 9/2$ and $3/2 < p < \infty$, set 
\begin{equation*}
    E_0 := \left(\sup_{0 < t < T} \int_{\R^3} \brak{v}^m f^p\right)^{2/5}\left(\int_0^T\int_{\R^3} \brak{v}^{m-3}\abs{\nabla f^{p/2}}^2 \dd v\dd \tau\right)^{3/5}.
\end{equation*}
Then, we have the estimate
\begin{equation*}
    \norm{f}_{L^\infty([t,T]\times \R^3)} \le C\max\left(E_0^{\frac{1}{p}}\left(1 + t^{-\frac{3}{2p}}\right), E_0^{\frac{2}{2p-3}}\right),
\end{equation*}
where $C$ depends only on $p$, $m$, $\lambda$, and $\Lambda$.
\end{proposition}

\begin{flushleft}
    {\bf \underline{Step 1: Introduction of Level Sets and Weighted Energy Functionals}}
\end{flushleft}

For $f$ a solution to \eqref{eq:landau} and for any $\ell \in \R^+$, we denote the level set function
\begin{equation*}
    f_\ell(t,v) := (f(t,v) - \ell)_+ = \max(f(t,v) - \ell, 0).
\end{equation*}
For $m \ge 9/2$ fixed and $p > 3/2$, we define three weighted energy functionals associated to $f$: for $0\le T_1 < T_2 \le T$ and $0 \le \ell < \infty$,
\begin{equation*}
\begin{aligned}
    \mathcal{A}_\ell(T_1,T_2) :&= \sup_{T_1 < t < T_2}\int_{\R^3} \brak{v}^m f_\ell^{p} \dd v\\
    \mathcal{B}_\ell(T_1,T_2) :&= \int_{T_1}^{T_2}\int_{\R^3} \brak{v}^{m-3} \abs{\nabla f_\ell^{p/2}} \dd v\\
    \mathcal{E}_\ell(T_1,T_2) :&= \sup_{T_1 < t < T_2} \int_{\R^3} \brak{v}^m f_\ell^{p/2} + \int_{T_1}^{T_2}\int_{\R^3} \brak{v}^{m-3}\abs{\nabla f_\ell^{p/2}}^{2} \dd v \dd t
\end{aligned}
\end{equation*}
Our first goal is showing the following nonlinear inequality: for any $0 \le k < \ell$ and $0 \le T_1 < T_2 < T_3 \le T$,
\begin{equation*}
\begin{aligned}
    \mathcal{E}_{\ell}(T_2,T_3) &\le C \left[\frac{1}{(\ell - k)^{2p/3}(T_2 - T_1)} + \frac{1}{(\ell - k)^{2p/3 - 1}} + \frac{1 + \ell}{(\ell - k)^{2p/3}} + \frac{\ell^2}{(\ell - k)^{2p/3+1}} \right]\\
    &\qquad \qquad \times \mathcal{A}^{2/3}_k(T_1,T_3)\mathcal{B}_k(T_1,T_3),
   \end{aligned}
\end{equation*}
where $C$ depends on $\lambda$, $\Lambda$, $p$, and $m$.

\begin{flushleft}
    {\bf \underline{Step 2: Differential Inequality}}
\end{flushleft}

\begin{lemma}\label{lem:subcritical_level_set_inequality}
    Under the assumptions of Proposition \ref{prop:subcritical_degiorgi} and notation introduced in Step 1, the following level set inequality holds for each $\ell > 0$,
    \begin{equation*}
    \begin{aligned}
        \frac{\dd}{\dd t}&\int_{\R^3} \brak{v}^m f_\ell^p \dd v + \int_{\R^3}\brak{v}^{m-3}\abs{\nabla f_\ell^{p/2}}^2 \dd v \\
        &\le C \int_{\R^3} \brak{v}^m f_\ell^{p+1} \dd v + C  (1 + \ell) \int_{\R^3} \brak{v}^m f_\ell^p \dd v + C\ell^2 \int_{\R^3} \brak{v}^m f_\ell^{p-1} \dd v,
    \end{aligned}
    \end{equation*}
    where the constant $C$ depends on $p$, $m$, $\lambda$, and $\Lambda$.
\end{lemma}

The proof follows from Lemma \ref{lem:level_set_inequality}, noting that $\norm{A[f]}_{L^\infty} \le \Lambda$.

\begin{flushleft}
    \underline{{\bf Step 3: Gain of Integrability}}
\end{flushleft}

In this step, we use the inequality derived in Lemma \ref{lem:subcritical_level_set_inequality} to obtain a nonlinear inequality for our energy functional $\mathcal{E}$ defined in Step 1.

\begin{lemma}\label{lem:subcritical_gain_of_integrability}
    Under the assumptions of Proposition \ref{prop:subcritical_degiorgi} and notation introduced in Step 1, the energy functional $\mathcal{E}$ satisfies for each $0 \le T_1 < T_2 < T_3 \le T$ and each $0 \le k < \ell$,
    \begin{equation*}
    \begin{aligned}
    \mathcal{E}_{\ell}(T_2,T_3)   &\le C \left[\frac{1}{(\ell - k)^{2p/3}(T_2 - T_1)} + \frac{1}{(\ell - k)^{2p/3 - 1}} + \frac{1 + \ell}{(\ell - k)^{2p/3}} + \frac{\ell^2}{(\ell - k)^{2p/3+1}} \right]\\
    &\qquad \qquad \times \mathcal{A}^{2/3}_k(T_1,T_3)\mathcal{B}_k(T_1,T_3),
    \end{aligned}
    \end{equation*}
    where the constant $C$ depends only on $p$, $\lambda$, $\Lambda$, and $m$.
\end{lemma}

Following the proof of Proposition \ref{prop:critical_degiorgi}, we integrate the inequality from Lemma \ref{lem:subcritical_level_set_inequality} over $\tau \in [t_1, t_2] \subset [0,T]$, take a supremum over $t_2 \in [T_2, T_3]$, and average over $t_1 \in [T_1,T_2]$ to find

\begin{equation}\label{eq:integral_level_set_inequality2}
\begin{aligned}
    \sup_{\tau \in [T_2,T_3]}& \int_{\R^3} \brak{v}^m f_\ell^p(\tau) \dd v + \int_{T_2}^{T_3}\int_{\R^3}\brak{v}^{m-3}\abs{\nabla f_\ell^{p/2}}^2 \dd v \dd \tau\\ &\le \frac{1}{T_2 - T_1}\int_{T_1}^{T_3}\int_{\R^3} \brak{v}^m f_\ell^p \dd v\dd \tau + C\int_{T_1}^{T_3}\int_{\R^3} \brak{v}^m f_\ell^{p+1} \dd v\dd \tau\\
        &\qquad + C(1 + \ell) \int_{T_1}^{T_3}\int_{\R^3} \brak{v}^m f_\ell^p \dd v \dd \tau + C\ell^2 \int_{T_1}^{T_3}\int_{\R^3} \brak{v}^m f_\ell^{p-1} \dd v\dd \tau.
\end{aligned}
\end{equation}
Recall that as a consequence of the definition of the level set functions $f_\ell$ and Chebychev's inequality for the measure $\dd \nu(v) = \brak{v}^m \dd v$, we have from \eqref{eq:Chebychev}, for any $ 0\le k < \ell$ and $0 < \alpha_1 < \alpha_2$,
\begin{equation}\label{eq:Chebychev2}
    \int_{T_1}^{T_3}\int_{\R^3} \brak{v}^m f_{\ell}^{\alpha_1} \dd v \dd \tau \le  (\ell - k)^{\alpha_1 - \alpha_2} \int_{T_1}^{T_3}\int_{\R^3} f_k^{\alpha_2} \dd v \dd \tau.
\end{equation}
Combining \eqref{eq:integral_level_set_inequality2} and \eqref{eq:Chebychev2} (with $\alpha_1 \in \set{p-1,\,p,\, p+1}$ and $\alpha_2 = 5p/3$), we have shown
\begin{equation*}
\begin{aligned}
    &\mathcal{E}_{\ell}(T_2,T_3) \\
    &\le C\left[\frac{1}{(T_2 - T_1)(\ell - k)^{2p/3}} + \frac{1}{(\ell -k)^{(2p - 3)/3}} + \frac{1 + \ell}{(\ell -k)^{2p/3}} + \frac{\ell^2}{(\ell - k)^{(2p + 3)/3}}\right]\\
    &\qquad \qquad \times \int_{T_1}^{T_3}\int_{\R^3} \brak{v}^{m} f_k^{5p/3} \dd v\dd\tau.
\end{aligned}
\end{equation*}
By interpolation and the weighted Sobolev inequality from Lemma \ref{lem:sobolev},
\begin{equation*}
    \int_{\R^3} \brak{v}^m f_\ell^{5p/3} \dd v \le C\left(\int_{\R^3} \brak{v}^{9/2} f_\ell^p \dd v\right)^{2/3}\left(\int_{\R^3} \brak{v}^{m-3} \abs{\nabla f_\ell^{p/2}} \dd v\right),
\end{equation*}
Since $m \ge 9/2$, integrating in time we find:
\begin{equation*}
\begin{aligned}
    \int_{T_1}^{T_3}\int_{\R^3} \brak{v}^m f_\ell^{\frac{5p}{3}} \dd v \dd \tau &\le C\int_{T_1}^{T_3}\left(\int_{\R^3} \brak{v}^m f_\ell^p \dd v\right)^{\frac{2}{3}}\left(\int_{\R^3} \brak{v}^{m-3} \abs{\nabla f_\ell^{p/2}} \dd v\right) \dd \tau\\ 
    &\le C\mathcal{A}_\ell^{2/3}(T_1,T_3)\mathcal{B}_\ell(T_1,T_3).
\end{aligned}
\end{equation*}
This concludes the proof of Lemma \ref{lem:subcritical_gain_of_integrability}.

\begin{flushleft}
    \underline{{\bf Step 4: De Giorgi Iteration}}
\end{flushleft}

We now finish the proof of Proposition \ref{prop:subcritical_degiorgi} using a De Giorgi iteration. We set $0 < t < T$ the times from the statement of Proposition \ref{prop:subcritical_degiorgi}. Then, we define our iteration quantities:
\begin{equation*}
    \ell_n = K(1 - 2^{-n}) \qquad t_n = t(1 - 2^{-n}) \qquad E_n = \mathcal{A}_{\ell_n}^{2/5}(t_n,T)\mathcal{B}_{\ell_n}(t_n,T)^{3/5}.
\end{equation*}
Then, using Young's inequality and Lemma \ref{lem:subcritical_gain_of_integrability}, we have the following recurrence for $E_n$:
\begin{equation*}
\begin{aligned}
    E_{n+1} &\le C\left(\mathcal{A}_{\ell_{n+1}}(t_{n+1},T) + \mathcal{B}_{\ell_{n+1}}(t_{n+1},T)\right) = C\mathcal{E}_{\ell_{n+1}}(t_{n+1},T)\\
            &\le C\Bigg(\frac{1}{(t_{n+1} - t_n)(\ell_{n+1} - \ell_n)^{2p/3}} + \frac{1}{(\ell_{n+1} - \ell_n)^{2p/3 - 1}} \\
            &\qquad \qquad + \frac{1 + \ell_{n+1}}{(\ell_{n+1} - \ell_n)^{2p/3}} + \frac{\ell_{n+1}^2}{(\ell_{n+1} - \ell_n)^{2p/3 + 1}} \Bigg)E_n^{5/3}\\
        &\le C2^{n\kappa}\left(\frac{1}{K^{2p/3}t} + \frac{1}{K^{2p/3 - 1}} + \frac{1 + K}{K^{2p/3}} + \frac{K^2}{K^{2p/3 + 1}} \right)E_n^{5/3}\\
        &\le C^*2^{n\kappa}\left(\frac{1}{K^{\frac{2p}{3}}t} + \frac{1}{K^{\frac{2p}{3}}} + \frac{1}{K^{\frac{2p-3}{3}}}\right)E_n^{5/3},
\end{aligned}
\end{equation*}
where $\kappa = 2p/3 + 1$ and $C^* > 1$ is now a fixed constant. We attempt to find a barrier sequence of the form $B_n := B^n E_0$ for some $B > 1$. That is, we seek values of $K$ and $B$ for which
\begin{equation*}
    B_{n+1} \ge C^*2^{n\kappa}\left(\frac{1}{K^{\frac{2p}{3}}t} + \frac{1}{K^{\frac{2p}{3}}} + \frac{1}{K^{\frac{2p-3}{3}}}\right)B_n^{5/3}.
\end{equation*}
Inserting the definition of $B_n$, this is implied by
\begin{equation*}
    1 \ge C^* B^{\frac{2n - 3}{3}}2^{n\kappa} E_0^{2/3} \left(\frac{1}{K^{\frac{2p}{3}}t} + \frac{1}{K^{\frac{2p}{3}}} + \frac{1}{K^{\frac{2p-3}{3}}}\right).
\end{equation*}
We first pick $B$ sufficiently large so that $B^{\frac{2}{3}} \ge 2^{\kappa}$. Second, we pick $K$ so that each remaining term is less than $1/3$:
\begin{equation*}
    K = \widetilde{C}(C^*,B,p)\max\left(E_0^{\frac{1}{p}}t^{-\frac{3}{2p}}, E_0^{\frac{1}{p}}, E_0^{\frac{2}{2p-3}} \right)
\end{equation*}
With these choices of $K$ and $B$, since $B_0 = E_0$, we find inductively:
\begin{equation*}
\begin{aligned}
    E_{n+1} &\le C^*2^{n\kappa}\left(\frac{1}{K^{\frac{2p}{3}}t} + \frac{1}{K^{\frac{2p}{3}}} + \frac{1}{K^{\frac{2p-3}{3}}}\right)E_n^{5/3}\\
        &\le C^*2^{n\kappa}\left(\frac{1}{K^{\frac{2p}{3}}t} + \frac{1}{K^{\frac{2p}{3}}} + \frac{1}{K^{\frac{2p-3}{3}}}\right)B_n^{5/3} \le B_{n+1}.
\end{aligned}
\end{equation*}
Therefore, 
\begin{equation*}
    \norm{f \chi_{\set{f \ge K}}}_{L^\infty(t,T;L^p_m)} \le \lim_{n\to\infty} E_n \le \lim_{n\to\infty} B_n = 0,
\end{equation*}
and we conclude $f(\tau,v)\le K$ pointwise a.e. on $[t,T]\times \R^3$. By the choice of $K$, we have shown the estimate:
\begin{equation*}
    f(\tau,v) \le C\max\left(E_0^{\frac{1}{p}}\left(1 + t^{-\frac{3}{2p}}\right), E_0^{\frac{2}{2p-3}}\right) \qquad \text{for almost all }\tau,v \in[t, T]\times \R^3.
\end{equation*}
This concludes the proof of Proposition \ref{prop:subcritical_degiorgi}.

\subsection{Propagation of \texorpdfstring{$L^p_m$}{Lpm} for smooth data}

\begin{lemma}\label{lem:subcritical-propagation}
Suppose $p > 3/2$ and $m \ge 9/2$, $f_{in} \in \S(\R^3)$, and $f:[0,T^*) \to \R^+$ is a smooth rapidly decaying solution to \eqref{eq:landau} with initial data $f_{in}$. Then, there is a time $T = T(\norm{f_{in}}_{L^p_m},m,p) > 0$ such that $f$ satisfies 
\begin{equation*}
    \sup_{0 < t < \min(T,T^*)} \norm{f(t)}_{L^p_m}^p + \int_0^{\min(T,T^*)}\int_{\R^3} \brak{v}^{m-3}\abs{\nabla f^{p/2}}^2 \dd v \dd \tau \le 4\norm{f_{in}}_{L^p_m}^p.
\end{equation*}
\end{lemma}

\begin{proof}
    From Lemma \ref{lem:subcritical_level_set_inequality} with $\ell = 0$, we find
    \begin{equation}\label{eq:diff-ineq-subcritical}
    \begin{aligned}
        \frac{\dd}{\dd \tau} \int_{\R^3} \brak{v}^m f^p \dd v &+ \int_{\R^3} \brak{v}^{m-3}\abs{\nabla f^{p/2}} \dd v \\
        &\le C\int_{\R^3}\brak{v}^m f^{p+1} \dd v + C\int_{\R^3} f^p \brak{v}^m \dd v.
      \end{aligned}
    \end{equation}
    By the interpolation estimate in Lemma \ref{lem:sobolev} with $q = p+1$ (since $m \geq \frac{9}{2}$) and Young's inequality (since $3/2p < 1$ by assumption), we find for any $\delta > 0$,
    \begin{equation*}
    \begin{aligned}
        \int_{\R^3}\brak{v}^m f^{p+1} \dd v &\le C\left(\int_{\R^3} \brak{v}^m f^p \dd v\right)^{1 - 1/2p}\left(\int_{\R^3} \brak{v}^{m-3}\abs{\nabla f^{p/2}}^2 \dd v\right)^{3/2p}\\
        &\leq \delta\left(\int_{\R^3} \brak{v}^{m-3}\abs{\nabla f^{p/2}}^2 \dd v\right) + C\delta^{-\frac{2p}{2p-3}} \left(\int_{\R^3} \brak{v}^m f^p \dd v\right)^{\frac{2p-1}{2p-3}}.
    \end{aligned}
    \end{equation*}
    Taking $\delta$ sufficiently small, depending only on $\lambda$, $\Lambda$, $p$, and $m$, and setting $\alpha = \frac{2p-1}{2p-3} > 1$, $y(t) = \norm{f(t)}_{L^p_m}^p$, and $F(t) = \norm{\nabla f(t)^{p/2}}_{L^2_m}^2$, we find
    \begin{equation*}
        \frac{\dd}{\dd t} y(t) + F(t) \le Cy^\alpha(t) + Cy(t).
    \end{equation*}
    This is a Bernoulli-type differential inequality for $y$, which can be explicitly solved by a suitable change of coordinates:
    Setting $v = e^{-C(1-\alpha)t}y^{1-\alpha}$, we find $v$ satisfies
    \begin{equation*}
        \frac{\dd}{\dd t} v = (1-\alpha)e^{-C(1-\alpha)t}y^{-\alpha} y^\prime(t) - C(1-\alpha)v \ge C(1-\alpha)e^{-C(1-\alpha)t}.
    \end{equation*}
    Integrating then gives that
    \begin{equation*}
        v(t) \ge v(0) - \left[\exp(C(\alpha - 1)t)  - 1\right].
    \end{equation*}
    This implies the following bound on $y(t)$:
    \begin{equation*}
        y(t) \le \exp(Ct)\left[y_0^{1-\alpha} + 1 -  \exp(C(\alpha - 1)t) \right]^{\frac{1}{1-\alpha}}
    \end{equation*}
    Since $\alpha > 1$,  we find a $T = T(y_0,\alpha, C)$ sufficiently small such that $y(t) \le 2y_0$ on $[0,\min(T,T^*)]$.
    Consequently, we integrate to find
    \begin{equation*}
    \int_0^{\min(T,T^*)} F(t) \dd t \le y_0 + C(y_0^\alpha + y_0)T \le 2y_0,
    \end{equation*}
    where the last inequality follows from picking $T$ possibly smaller.
\end{proof}

\subsection{Construction of solution}

We now construct local-in-time solutions. For $f_{in} \in L^p_m \cap L^1_2$, we take $f_{in,\eps} \in \S(\R^3)$ so that $f_{in,\eps}$ converges to $f_{in}$ in $L^1_2 \cap L^p_m$. Using local-in-time well-posedness for $f_{in,\eps}$, there are corresponding solutions $f_{\varepsilon}:[0,T^*_{\eps})\times \R^3 \to \R^+$. From Lemma \ref{lem:subcritical-propagation}, there is a uniform-in-$\eps$ time $T$ such that for $0 < t < \min(T,T^*_{\eps})$, we have the uniform estimate
\begin{equation*}
    \left(\sup_{0 < \tau < t} \int_{\R^3} \brak{v}^m f^p_{\eps} \dd v\right)^{2/5}\left(\int_{0}^t\int_{\R^3} \brak{v}^{m-3} \abs{\nabla f^{p/2}_{\eps}}^2 \dd v \dd \tau\right)^{3/5} \lesssim 1.
\end{equation*}
Using Proposition \ref{prop:subcritical_degiorgi}, we find that $T^*_{\eps} > T$ for each $\eps > 0$ and there holds the uniform-in-$\eps$ estimate:
\begin{equation*}
    \norm{f_{\varepsilon}(t)}_{L^\infty} \lesssim \left(1 + \frac{1}{t^{3/2p}}\right) \qquad \text{for each }0 < t < T.
\end{equation*}
Passing to the limit, $f_{\eps} \to f$, where these bounds guarantee $f \in C(0,T;L^p_m)$ is smooth solution for positive time and satisfies the claimed smoothing estimates. Moreover, by uniqueness result of Fournier, Theorem \ref{thm:Fournier}, the limit $f$ is the unique $L^1(0,T;L^\infty)$ solution with initial data $f_{in}$.

\subsection{Uniqueness}
Let $g: [0, T] \times \R^3 \to \R^+$ be a weak solution of \eqref{eq:landau} with initial datum $g_{in} \in L^1_2 \cap L\log L$, attained in the sense of distributions. We suppose further that $g \in L^r(0, T; L^p_{9/2}(\R^3)) \cap L^\infty(s,T;W^{1,\infty}(\Omega))$ for each $0 < s < T$ and each open, bounded $\Omega \subset \R^3$ and some pair $(r,p)$ satisfying $p > 3/2$ and $r > \frac{2p}{2p-3}$.

Our goal is to derive a differential inequality for $y(t) := \norm{g(t)}_{L^p_{9/2}(\R^3)}^p$ much as in Lemma \ref{lem:subcritical-propagation} to obtain weighted $L^r(0,T;L^p) \to L^\infty(t,T;L^p)$ smoothing estimates. Formally, the idea is to test \eqref{eq:landau} with $\brak{v}^{9/2} g^{p-1}$ and use the Sobolev embedding in Lemma \ref{lem:sobolev} to control the highest order term. 
Formally testing \eqref{eq:landau} with $g^{p-1}\brak{v}^{9/2}$, Lemma \ref{lem:subcritical_level_set_inequality} (with $\ell = 0$ and $m = 9/2$) implies the following differential inequality:
\begin{equation}\label{eq:diff-ineq-subcritical-g}
    \begin{aligned}
       \frac{\dd}{\dd t}\int_{\R^3} \brak{v}^{9/2} g^p \dd v &+ \int_{\R^3}\brak{v}^{3/2}\abs{\nabla g^{p/2}}^2 \dd v\\
        &\le C \int_{\R^3} \brak{v}^{9/2} g^{p+1} \dd v + C\int_{\R^3} \brak{v}^{9/2} g^{p} \dd v.
    \end{aligned}
\end{equation}
Then, we interpolate and use the Sobolev embedding from Lemma \ref{lem:sobolev} and apply Young's inequality to obtain: for any $\varepsilon \in (0,1)$
    \begin{align*}
        \int_{\R^3} \brak{v}^{9/2} g^{p+1}\dd v &\le \norm{g}_{L^p_{9/2}}^{p - \frac{1}{2}}\norm{g}_{L^{3p}_{9/2}}^{\frac{3}{2}}\le \norm{g}_{L^p_{9/2}}^{p - \frac{1}{2}}\norm{g^{p/2}}_{L^{6}_{9/2}}^{\frac{3}{p}}\\
            &\le C\norm{g}_{L^p_{9/2}}^{p - \frac{1}{2}}\left(\int_{\R^3}\brak{v}^{3/2}\abs{\nabla g^{p/2}}^2 \dd v\right)^{\frac{3}{2p}}\\
            &\le C(\eps) \norm{g}_{L^p_{9/2}}^{\left(p- \frac{1}{2}\right)\frac{2p}{2p-3}} + \eps \left(\int_{\R^3}\brak{v}^{3/2}\abs{\nabla g^{p/2}}^2 \dd v\right).
    \end{align*}
Picking $\eps$ sufficiently small, the second term can absorbed by left hand side of \eqref{eq:diff-ineq-subcritical-g}. Denoting by $y(t) := \norm{g(t)}_{L^p_m}^p$, for $t \in (0, T]$ we find:
\begin{equation*}
        \frac{d}{dt} y + \norm{\nabla g^{p/2}(t)}_{L^2_{3/2}}^2 \le C(\eps)  y^{1+\frac{2}{2p-3}}(t) + y(t).
    \end{equation*}
Since $y$ is only absolutely continuous for positive times, we integrate over $[s, t]$ for $0 < s < t < T$ to obtain: 
\begin{align}\label{eq:subcritical-ode-y}
    y(t) +\int_s^t \norm{\nabla g^{p/2}(\tau)}_{L^2_{3/2}}^2 \dd \tau \leq y(s) + C(\eps)  \int_s^t y^{1+\frac{2}{2p-3}}(\tau) \dd \tau + \int_s^t y(\tau) \dd \tau.
\end{align}
Gr\"{o}nwall's inequality and $\frac{2p}{2p-3} \le r$ then implies the explicit bound:
\begin{align*}
    \sup_{s \leq t \leq T} y(t) \leq y(s)\exp\left( C (T-s) + C \int_s^T y^{\frac{2}{2p-3}}(\tau) \dd \tau\right) \le Cy(s).
\end{align*}
Rearranging powers and averaging over $s \in [0, t]$, we find 
\begin{align*}
    \norm{g(t)}_{L^p_{9/2}}^r \leq C t^{-1} \int_0^t \norm{g(s)}_{L^p_{9/2}}^r \dd s \leq C t^{-1}. 
\end{align*}
We plug this estimate back into \eqref{eq:subcritical-ode-y}, to obtain for each $0 < t < T$,
\begin{equation}\label{eq:decay-g}
\begin{aligned}
    \sup_{t< \tau < T} &\norm{g(\tau)}_{L^p_{9/2}}^p + \int_t^T \norm{\nabla g^{p/2}(\tau)}_{L^2_{3/2}}^2 \mathrm{d}\tau \\
    &\leq \norm{g(t)}_{L^p_{9/2}}^p + C(\eps)  \int_t^T \norm{g(\tau)}_{L^p_{9/2}}^{p+r} \mathrm{d}\tau + \int_t^T \norm{g(\tau)}_{L^p_{9/2}}^p \mathrm{d}\tau\\
    &\leq C t^{-\frac{p}{r}}.
\end{aligned}
\end{equation}

In a second step, we use smoothing estimates to show that $g$ belongs to Fournier's uniqueness class, i.e. $g \in L^1(0, T; L^\infty)$. 
Using the smoothing estimate of Proposition \ref{prop:subcritical_degiorgi} - which can be rigorously justified since $g \in W^{1, \infty}_{loc}$ satisfies the global (in $v$) bound \eqref{eq:decay-g} - we obtain
\begin{equation*}
    \norm{g}_{L^\infty([t,T]\times \R^3)} \leq C \max\left\{t^{- \frac{1}{r}}\left(1 + t^{-\frac{3}{2p}}\right), t^{-\frac{2p}{r(2p-3)}}\right\} \leq Ct^{-\left(\frac{1}{r} + \frac{3}{2p}\right)} + Ct^{-\frac{2p}{r(2p-3)}}. 
\end{equation*}
Since $r > \frac{2p}{2p-3}$, we see that $g \in L^1(0, T; L^\infty(\R^3))$. Since this holds for any locally Lipschitz $g \in L^r(0,T;L^p_{9/2})$, by Fournier's uniqueness result in Theorem \ref{thm:Fournier}, we conclude that any two such solutions $g$ with coinciding initial data are equal.

We conclude this section by noting that while the above computations are merely formal, they can be made rigorous. Although $g$ is locally Lipschitz, $g$ does not have sufficient (global) regularity to merit its use as a test function in \eqref{eq:landau}. To rigorously justify the formal argument presented here, one should truncate the norm, by using $\brak{v}^{9/2}\varphi^2 g^{p-1}$ as a test function, where $\varphi$ is a Lipschitz cutoff to $B_R$. To ensure that the formal arguments presented transfer to the local truncations (uniformly in $R$), we need a more precise version of the interpolation from Lemma \ref{lem:sobolev} used above:
\begin{lemma}\label{lem:eps_poincare}
    Suppose $g$, $\varphi$ are any Schwartz class functions. Then, for any $3/2 < q < \infty$, there is a universal constant $C = C(q)$ such that for any $\eps > 0$, there holds:
\begin{equation*}\begin{aligned}
    \int_{\R^3} &\brak{v}^{9/2}\varphi^2 g^{p+1} \dd v \\
    &\le \eps \int_{\R^3} \brak{v}^{3/2} \abs{\nabla \left(\varphi g^{p/2}\right)}^2 \dd v + C\left(\eps^{-3}\norm{g}_{L^q_{9/2}}^{2q}\right)^{\frac{1}{2q-3}}\int_{\R^3} \brak{v}^{9/2}\varphi^2 g^p \dd v.\end{aligned}
\end{equation*}
\end{lemma}
Such estimates were first applied to the study of the homogeneous Landau equation in the work of the second author and Guillen \cite{GualdaniGuillen2}, where they were termed $\eps$-Poincar\'{e} inequalities. More recently, a new $\eps$-Poincar\'{e} inequality appeared in the work of Alonso et al. \cite{AlonsoBaglandDesvillettesLods_ProdiSerrin}, enabling the authors to obtain a condition for propagation and creation of $L^p$ norms. The $\eps$-Poincar\'{e} inequality of \cite{AlonsoBaglandDesvillettesLods_ProdiSerrin} was the crucial tool in \cite{GoldingLoher} for turning a priori smoothing arguments into a rigorous uniqueness proof. We postpone a proof of Lemma \ref{lem:eps_poincare} to Appendix \ref{sec:eps_poincare} and refer the reader to \cite[Section 6]{GoldingLoher} for details on how to use Lemma \ref{lem:eps_poincare} to provide a rigorous justification of the formal argument above.

\section{The subcritical global in time theory}\label{sec:long_time2}

\subsection{Pointwise lower bounds}

Here we show the solution constructed in Theorem \ref{thm:short} propagates lower bounds, which will be used to control the Fisher information. When $p > 3/2$, $\norm{A[f]}_{L^\infty} \lesssim 1$ so that we can show that the function $\psi(t,v) = a\exp(-\eta t) \brak{v}^{-k}$ satisfies
\begin{equation*}
    \partial_t \psi \le A[f]:\nabla^2 \psi + f\psi,
\end{equation*}
provided $\eta$ is chosen suitably large. We still take Stampacchia approach to show $f_{in} \ge \brak{v}^{-k}$ implies $f \ge \psi$, but take less care in justifying computations, particularly when using integration by parts. Because we have uniqueness in the $p > 3/2$ case, a priori estimates can easily be extended to rough data. We begin by proving that $\psi$ as above is indeed a subsolution to the linearized equation:
\begin{lemma}\label{lem:psi}
    Suppose $f:[0,T]\times \R^3 \to \R^+$ is a smooth solution to \eqref{eq:landau} with
    \begin{equation*}
       C_1\brak{v}^{-3} \le A[f] \le C_2.  
    \end{equation*}
    Then, for each $k > 5$ and $a > 0$, there is $\eta = \eta(C_1,C_2,k)$ sufficiently large so that $\psi(t,v) = a\exp(-\eta t) \brak{v}^{-k}$ satisfies
    \begin{equation*}
        \partial_t \psi \le A[f]:\nabla^2 \psi + f\psi.
    \end{equation*}
\end{lemma}

\begin{proof}
    Define $\psi(t,v) = a\exp(-\eta t) \brak{v}^{-k}$. Then, we compute derivatives of $\psi$.
    \begin{align*}
        &\partial_t \psi = -\eta \psi, \qquad \nabla \psi = a\exp(-\eta t) \left(-k \brak{v}^{-k-2}v\right), \\
        &\nabla^2 \psi = a\exp(-\eta t) \left[k(k+2) \brak{v}^{-k-4}v\tens v - k\brak{v}^{-k-2}Id\right] \\
        &\qquad= \left[k(k+2)\frac{v\tens v}{\brak{v}^{4}} - k\frac{Id}{\brak{v}^2} \right]\psi.
    \end{align*}
    So, using $f \geq 0$, $\mathrm{tr}\left(A[f]\right) \le C_1$, $A[f] \ge C_2\brak{v}^{-3}$, by Young's inequality
    \begin{equation*}
    \begin{aligned}
        \partial_t \psi - A[f]:\nabla^2 \psi - f\psi & \le \left[\frac{C_1k}{\brak{v}^2} -\eta - \frac{C_2k(k+2)|v|^2}{\brak{v}^{7}}\right]\psi\\ 
            &\le \left[\frac{\tilde C_1}{\brak{v}^2} - \eta - \frac{\tilde C_2}{\brak{v}^5}\right]\psi\\
            &\le \left[\tilde C_1\left(\frac{5}{2\delta}\right)^{10/3} + \tilde C_1\delta \brak{v}^{-5} - \eta - \tilde C_2 \brak{v}^{-5}\right]\psi.
    \end{aligned}
    \end{equation*}
    Therefore, picking $\delta = \frac{\tilde C_2}{\tilde C_1}$ and then picking $\eta > \tilde C_1(\frac{5}{2\delta})^{10/3}$, concludes the proof.
\end{proof}
\begin{lemma}\label{lem:lowerbound-f-subcritical}
    Suppose $f:[0,T]\times \R^3 \to \R^+$ is a smooth solution to \eqref{eq:landau} with initial data $f_{in} \in L^p_{9/2}$ for $p > 3/2$ and $f_{in} \ge a\brak{v}^{-k}$ for $k > 5$ and $a > 0$. Suppose further that $f\in L^\infty(0,T;L^p) \cap L^1(0,T;L^\infty)$. Then, there is an $\eta = \eta(p,\norm{f_{in}}_{L^p}) > 0$ sufficiently large such that
    \begin{equation*}
        f(t) \ge a\exp(-\eta t) \brak{v}^{-k} \qquad \text{for }0 \le t \le T.
    \end{equation*}
\end{lemma}

\begin{proof}
    Let $\psi$ be as in Lemma \ref{lem:psi}, so that $f_{in}(v) \ge \psi(0,v)$ and
    \begin{equation*}
        \partial_t \psi \le A[f]:\nabla^2\psi + \psi f.
    \end{equation*}
    Set $\varphi = \psi - f$ so that by linearity,
    \begin{equation*}
        \partial_t \varphi \le A[f]:\nabla^2 \varphi + \varphi f.
    \end{equation*}
    Then, performing an $L^p$ estimate, we find
    \begin{equation*}
    \begin{aligned}
        \frac{\dd }{\dd t}\int_{\R^3} \varphi_+^{p} &= -\int_{\R^3} \nabla \varphi_+^{p-1} \cdot \left(A[f]\nabla \varphi - \nabla a[f]\varphi\right) \dd v\\
            &= -\frac{4(p-1)}{p^2}\int_{\R^3} \nabla \varphi_+^{p/2} \cdot A[f]\nabla \varphi^{p/2}_+ \dd v + \frac{p-1}{p}\int_{\R^3} \nabla\varphi_+^p \cdot \nabla a[f] \dd v\\
            &\le \int_{\R^3} \varphi_+^p f \dd v\le \int_{\R^3} \varphi_+^p \psi \dd v \le \norm{\psi(t)}_{L^\infty}\int_{\R^3} \varphi_+^p\dd v.
    \end{aligned}
    \end{equation*}
    Since $\psi \in L^1(0,T;L^\infty)$ and $\varphi_+(0) = 0$, by Gr\"onwall's inequality, $\varphi_+ = 0$ on $[0,T]\times \R^3$, which completes the proof.
\end{proof}

\subsection{Bound on the Fisher information and proof of Theorem \ref{thm:long}}
Fix initial datum $f_{in} \in L^1_m \cap L^{p}_{m_0}$ for some $p\in (3/2, +\infty)$, with $m$ and $m_0$ as stated in the theorem. 
From Theorem \ref{thm:short}, we have a solution $f:[0,T]\times \R^3 \to \R^+$ satisfying the bounds
\begin{equation}\label{eq:H1_upper_bound2}
\begin{aligned}
    \sup_{0 < t < T} \norm{f(t)}_{L^{p}_{m_0}}^{p} + \int_0^{T}\int_{\R^3} \brak{v}^{m_0 - 3}\abs{\nabla f^{p/2}}^2 \dd v \dd t + \sup_{0 < t < T} &t^{3/2p}\norm{f(t)}_{L^\infty}\\
    & \le C(f_{in}, m_0, p, T).
    \end{aligned}
\end{equation}
Since $\norm{f(t)}_{L^p} \lesssim 1$, by Lemma \ref{lem:coefficient_bounds}, $\norm{A[f(t)]}_{L^\infty} \lesssim 1$. So, by Lemma \ref{lem:lowerbound-f-subcritical}, for some $\eta = \eta(k,p,\norm{f_{in}}_{L^p},T) > 0$ there holds:
\begin{equation*}
    f(t,v) \ge a\exp(-\eta t) \brak{v}^{-k} \qquad \text{for each }(t,v) \in [0,T]\times\R^3.
\end{equation*}
Combined with \eqref{eq:H1_upper_bound2}, we have shown the following upper bound on the Fisher information: 
\begin{equation*}\begin{aligned}
    \int_0^T\int_{\R^3}f^{p-2} \frac{\abs{\nabla f}^2}{f^{p-1}} \dd v\dd t &\le C(a,T,k)\int_0^T\int_{\R^3} \brak{v}^{k(p-1)}\abs{\nabla f^{p/2}}^2 \dd v\dd t \\
    &\le C(a,T,k,f_{in},m_0,p)
    \end{aligned}
\end{equation*}
provided $k(p-1) \le m_0 - 3$. We conclude as in the $p = 3/2$ case of Theorem \ref{thm:long-cor-GS-GL} that $f$ can be continued to a global-in-time smooth solution with decreasing Fisher information.

\appendix 
\section{Proof of Lemma \ref{lem:sobolev}}
\label{appendix}
In the section, we provide proofs for some of the basic technical tools used throughout. 

\begin{flushleft}
    \underline{Proof of the weighted Sobolev inequality \eqref{eq:sobolev}}
\end{flushleft}

\begin{proof}
The case $k = 3$ follows from the Sobolev embedding $\dot{H}^1(\R^3) \embeds L^6(\R^3)$ and Hardy's inequality.

For $k > 3$, we first compute derivatives of our weight $\brak{v}^m$ for $m\in \R$ arbitrary.
\begin{equation*}
    \nabla \brak{v}^m = m \brak{v}^{m-2}v \quad \text{and}\quad \Delta \brak{v}^m = 3m \brak{v}^{m-2} + m(m-2)\brak{v}^{m-4}|v|^2
\end{equation*}
Now, we apply the Sobolev embedding to the function $\varphi = \brak{v}^m f$ and find for $C_1 > 0$ the Sobolev constant,
\begin{equation*}
\begin{aligned}
    \Bigg(\int_{\R^3}& \brak{v}^{6m} f^6\Bigg)^{1/3} \\
    &\le C_1\int_{\R^3} \abs{\nabla \brak{v}^mf}^2\\
        &= C_1 \int_{\R^3} (f\nabla \brak{v}^m + \brak{v}^m \nabla f)^2\\
        &= C_1 \int_{\R^3} \brak{v}^{2m}\abs{\nabla f}^2 \dd v +  \int_{\R^3} m^2 f^2 \brak{v}^{2(m-2)}|v|^2 + 2m f\brak{v}^{2m-2}\nabla f \cdot v\\
        &= C_1 \int_{\R^3} \brak{v}^{2m}\abs{\nabla f}^2 \dd v + C_1\int_{\R^3} m^2 f^2 \brak{v}^{2(m-2)}|v|^2 + \frac{1}{2}\nabla f^2 \cdot \nabla \brak{v}^{2m}\\
        &= C_1 \int_{\R^3} \brak{v}^{2m}\abs{\nabla f}^2 \dd v + C_1\int_{\R^3} m^2 f^2 \brak{v}^{2(m-2)}|v|^2 - \frac{1}{2} f^2 \Delta \brak{v}^{2m}\\
        &\le C_1\int_{\R^3} \brak{v}^{2m} \abs{\nabla f}^2 + C_2 \int_{\R^3} f^2 \brak{v}^{2m-2}.
\end{aligned}
\end{equation*}
For $m > 0$, we use $|v|^2 \le \brak{v}^2$ to compute $C_2$ as
\begin{equation*}
    C_2 = C_1(m^2 - \frac{1}{2}\left(6m + 2m(2m-2)\right)) = C_1( m^2 - 3m + 2m^2 + 2m) = -C_1(m^2 + m) < 0.
\end{equation*}
Finally, setting $2m = k-3$, for $k > 3$, we conclude
\begin{equation*}
    \left(\int_{\R^3} \brak{v}^{3k - 9} f^6\right)^{1/3} \le C_1\int_{\R^3} \brak{v}^{k-3} \abs{\nabla f}^2 - \frac{C_1}{4}(k-3)(k-1) \int_{\R^3} f^2 \brak{v}^{k-5}.
\end{equation*}
\end{proof}

\begin{flushleft}
    \underline{Proof of the interpolation estimate \eqref{eq:interpolation-1}}
\end{flushleft}

\begin{proof}
    We use interpolation of weighted Lebesgue spaces (i.e. H\"older's inequality) to bound $\norm{f}_{L^{q}_k}$ as
    \begin{equation*}
        \norm{\brak{v}^{\frac{k}{q}}f}_{L^{q}} \le \norm{\brak{v}^{\frac{m}{p}}f}_{L^p}^{\theta}\norm{\brak{v}^{\frac{k-3}{p}}f}_{L^{3p}}^{1-\theta},
    \end{equation*}
    provided $\theta \in (0,1)$ and $m \in \R$ satisfy the relations
    \begin{equation*}
        \frac{1}{q} = \frac{\theta}{p} + \frac{1-\theta}{3p} \qquad \text{and} \qquad \frac{k}{q} = \frac{m\theta}{p} + \frac{(k-3)(1-\theta)}{p}.
    \end{equation*}
    Solving this system of constraints for $\theta$ and $m$, we obtain
    \begin{equation*}
        \theta = \frac{3p-q}{2q}, \qquad 1-\theta = \frac{3q - 3p}{2q},\qquad \text{and} \qquad m = \frac{2kp - (k-3)(3q-3p)}{3p-q}.
    \end{equation*}
    Therefore, since $k \ge 3$, we use the weighted Sobolev inequality \eqref{eq:sobolev} to obtain
    \begin{equation*}\begin{aligned}
        \norm{f}_{L^{q}_k}^{q} &\le \norm{f}_{L^p_m}^{\frac{3p-q}{2}}\norm{\brak{v}^{\frac{k-3}{p}}f}_{L^{3p}}^{\frac{3}{2}(q-p)} \\
        &\le \norm{f}_{L^p_m}^{\frac{3p-q}{2}}\norm{f^{p/2}}_{L^6_{3k-9}}^{\frac{3}{p}(q-p)} \le C_{k,p}\norm{f}_{L^p_m}^{\frac{3p-q}{2}}\norm{\nabla f^{p/2}}_{L^2_{k-3}}^{\frac{3}{p}(q-p)},
  \end{aligned}
    \end{equation*}
    which completes the proof. 
\end{proof}

\section{Proof of Lemma \ref{lem:cutoff}}\label{appendix2}
In this section, we construct the smooth cutoff used in the proof of Theorem \ref{thm:long-cor-GS-GL}.
\begin{proof}
    We define first an auxiliary function $\tilde \varphi$ as
    \begin{equation*}
        \tilde \varphi(x) = \begin{cases} \exp(-x^{-1}) &\text{if } x > 0\\ 0 &\text{otherwise.} \end{cases}
    \end{equation*}
    From $\tilde \varphi$, we translate and rescale $x \mapsto \frac{\tilde \varphi(x)}{\tilde \varphi(x) + \tilde \varphi(1 - x)}$ and define the radial profile of the cutoff:
    \begin{equation*}
        \begin{aligned}
           \varphi(x) &= \left(\frac{\tilde \varphi\left(2- x/R\right)}{\tilde \varphi\left(2-x/R\right) + \tilde \varphi\left(1 - (2-x/R)\right)}\right)^2 \\
           &= \begin{cases}1, \qquad &\text{if } x \leq R, \\
               \left( \frac{\exp\left(-(2-x/R)^{-1}\right)}{\exp\left(-(2-x/R)^{-1}\right) + \exp\left(-(x/R-1)^{-1}\right)}\right)^2, \qquad &\text{if } R < x < 2R,\\
                0 \qquad\qquad &\text{if }  2R \leq x.
             \end{cases}
        \end{aligned}
    \end{equation*}
    The cutoff we seek is $\eta(v) = \varphi(|v|)$. Note that because $\tilde \varphi \in C^\infty$, $\eta \in C^\infty$ and we can easily verify $0 \le \eta \le 1$, $\eta = 1$ on $B_R$, and $\eta = 0$ on $\R^3 \setminus B_R$.
    It remains only to verify the differential inequalities, which are nontrivial only when $R < |v| < 2R$. 
    
    Because $\eta$ is radial, $\nabla \eta(v) = \frac{v}{\abs{v}}\varphi^\prime(\abs{v})$, so that $\abs{\nabla \eta(v)} = \abs{\varphi'(\abs{v})}$.
    For $R < \abs{v} < 2R$, by explicit computation,
    \begin{equation*}
    \begin{aligned}
        \varphi^\prime(x) &= 2\sqrt{\varphi(x)}\left(\frac{\tilde\varphi(2-x/R)}{\tilde \varphi(2 - x/R) + \tilde \varphi(1 - (2-x/R))}\right)^\prime\\
            &= -\frac{2\sqrt{\varphi(x)}}{R}\left(\frac{\tilde\varphi^\prime(2-x/R)\tilde \varphi(1-(2- x/R)) + \tilde\varphi(2-x/R) \tilde \varphi^\prime(1-(2- x/R))}{\left[\tilde \varphi(2 - x/R) + \tilde \varphi(1-(2- x/R))\right]^2}\right).
    \end{aligned}
    \end{equation*}
    We bound the term in parentheses by computing it explicitly and realizing that for $R < x < 2R$ the denominator is minimized at $x = 3R/2$, and that the numerator is bounded since $x \mapsto e^{-x^{-1}} x^{-2}$ is bounded, as the exponential dominates:
    \begin{equation*}
        0 \le \frac{ e^{-(2 - x/R)^{-1}}e^{-(1- (2 - x/R))^{-1}}\left[(2 - x/R)^{-2} + (1 - (2 - x/R))^{-2} \right]}{\left[e^{-(2 - x/R)^{-1}} + e^{-(1- (2 - x/R))^{-1}}\right]^2} \le C.
    \end{equation*}
    Therefore, $\varphi$ is decreasing and
    \begin{equation}\label{eq:eta'}
        \begin{aligned}
            \abs{\nabla \eta(v)} = \abs{\varphi^\prime(\abs{v})} = -\varphi^\prime(\abs{v}) \le \frac{2C\sqrt{\varphi(\abs{v})}}{R} = \frac{2C\sqrt{\eta}}{R}.
                 \end{aligned}
    \end{equation}
    From \eqref{eq:eta'}, we deduce one differential equality for $\eta$. On the other hand, we see for $R < x < 2R$,
    \begin{equation*}
        \begin{aligned}
            1 - \varphi(x) &= \frac{\tilde \varphi\left(1 - (2 - x/R)\right) \left(\varphi(1 - (2-x/R)) + 2 \tilde \varphi(2-x/R)\right)}{\left[\tilde \varphi\left(2-x/R\right) + \tilde \varphi\left(1 - (2-x/R)\right)\right]^2}\\
                &= \frac{e^{-(1 -(2-x/R))^{-1}}\left[e^{-(1 -(2-x/R))^{-1}}+2e^{-(2-x/R)^{-1}}\right] }{ \left(e^{-(2-x/R)^{-1}} + e^{-(1- (2 - x/R))^{-1}}\right)^2},
        \end{aligned}
    \end{equation*}
    so that the explicit expressions for $\varphi$ yield
    \begin{equation*}
    \begin{aligned}
        \frac{\abs{\varphi'(x)}^2}{1 - \varphi(x)} &= \frac{ 4R^{-2}e^{-(2 - x/R)^{-1}}}{\left[e^{-(1 -(2-x/R))^{-1}}+2e^{-(2-x/R)^{-1}}\right]}\\
            &\qquad\times \left[{\frac{ e^{-(2 - x/R)^{-1}}e^{-1/2(1 - (2- x/R))^{-1}}\left[(2-x/R)^{-2} + (1- (2-x/R))^{-2} \right]}{\left[e^{-(2 - x/R)^{-1}} + e^{-(1- (2- x/R))^{-1}}\right]^2}}\right]^2\\
            & \leq \frac{C}{R^2},
    \end{aligned}
    \end{equation*}
    where we bounded the term in brackets as before, minimizing the denominator and again using that the exponential dominates any polynomial. We conclude
    \begin{equation}\label{eq:eta'2}
        \abs{\nabla \eta(v)} \le \frac{C\sqrt{1 - \eta(v)}}{R},
    \end{equation}
    and \eqref{eq:eta'2} is the second differential inequality claimed.
\end{proof}

\section{Proof of Lemma \ref{lem:eps_poincare}}\label{sec:eps_poincare}

In this section, we provide a proof of the $\eps$-Poincar\'{e} inequality in Lemma \ref{lem:eps_poincare}:
\begin{proof}
    We prove first an unweighted version of the desired inequality. Using Holder's inequality and the Gagliardo-Nirenberg-Sobolev inequality, we obtain:
    \begin{align*}
        \int_{\R^3} \varphi^2 g^{p+1} \dd v &\le \left(\int_{\R^3} g^q \dd v\right)^{\frac{1}{q}}\left(\int_{\R^3} \left(\varphi^2 g^{p}\right)^{\frac{q}{q-1}}\right)^{\frac{q-1}{q}} = \norm{g}_{L^q} \norm{\varphi g^{p/2}}_{L^{\frac{2q}{q-1}}}^2\\
            &\le C\norm{g}_{L^q}\norm{\nabla(\varphi g^{p/2})}_{L^2}^{2\theta}\norm{\varphi g^{p/2}}_{L^2}^{2(1-\theta)}\\
            &\le \eps\left(\int_{\R^3} \abs{\nabla(\varphi g^{p/2})}^2 \dd v\right) + C\eps^{-\frac{\theta}{1-\theta}}\norm{g}_{L^q}^{\frac{1}{1 - \theta}}\left(\int_{\R^3}\varphi^2 g^p \dd v\right),
    \end{align*}
    where $\frac{q-1}{2q} = \frac{1}{2} - \frac{\theta}{3}$. Solving for $\theta$, we find $\theta = \frac{3}{2q} \in (0,1)$ for $q > 3/2$. Inserting this value of $\theta$, we find
    \begin{equation}\label{eq:eps_poincare_unweighted}
        \int_{\R^3} \varphi^2 g^{p+1} \dd v \le \eps\left(\int_{\R^3} \abs{\nabla(\varphi g^{p/2})}^2 \dd v\right) + C\eps^{-\frac{3}{2q-3}}\norm{g}_{L^q}^{\frac{2q}{2q-3}}\left(\int_{\R^3}\varphi^2 g^p \dd v\right).
    \end{equation}
    Repeating the above argument with homogeneous weights of the form $\abs{v}^m$ and using the Caffarelli-Kohn-Nirenberg inequality (see \cite{CaffarelliKohnNirenberg} or later improvements in \cite{DuarteSilva,Lin}) in place of the Gagliardo-Nirenberg-Sobolev inequality, we find:
    \begin{align*}
        \int_{\R^3}& \abs{v}^{9/2} \varphi^2 g^{p+1} \dd v \\
        &\le \left(\int_{\R^3} \abs{v}^{9/2}g^q \dd v\right)^{\frac{1}{q}}\left(\int_{\R^3} \abs{v}^{9/2}\left(\varphi^2 g^{p}\right)^{\frac{q}{q-1}}\right)^{\frac{q-1}{q}}\\
            &= \norm{\abs{v}^{\frac{9}{2q}}g}_{L^q} \norm{\abs{v}^{\frac{9q}{4(q-1)}}\varphi g^{p/2}}_{L^{\frac{2q}{q-1}}}^2\\
            &\le C\norm{\abs{v}^{\frac{9}{2q}}g}_{L^q}\norm{\abs{v}^{3/4}\nabla(\varphi g^{p/2})}_{L^2}^{2\theta}\norm{\abs{v}^{9/4}\varphi g^{p/2}}_{L^2}^{2(1-\theta)}\\
            &\le \eps\left(\int_{\R^3} \abs{v}^{3/2}\abs{\nabla(\varphi g^{p/2})}^2 \dd v\right) + C\eps^{-\frac{\theta}{1-\theta}}\norm{\abs{v}^{\frac{9}{2q}}g}_{L^q}^{\frac{1}{1 - \theta}}\left(\int_{\R^3}\abs{v}^{9/2}\varphi^2 g^p \dd v\right),
    \end{align*}
    where $\theta$ is now given by
    \begin{equation*}
        \frac{q-1}{2q} + \frac{\frac{9(q-1)}{4q}}{3} = \theta \left(\frac{1}{2} - \frac{1 - 3/4}{3} \right) + (1-\theta)\left(\frac{1}{2} + \frac{9/4}{3}\right).
    \end{equation*}
    Solving for $\theta$, we once again find $\theta = \frac{3}{2q} \in (0,1)$ for $q > 3/2$.
    Consequently, we obtain
    \begin{equation}\label{eq:eps_poincare_homogeneous}
    \begin{aligned}
        \int_{\R^3} \abs{v}^{9/2}\varphi^2 g^{p+1} \dd v \le &\eps\left(\int_{\R^3} \abs{v}^{3/2}\abs{\nabla(\varphi g^{p/2})}^2 \dd v\right) \\
        &+ C\eps^{-\frac{3}{2q-3}}\norm{\abs{v}^{\frac{9}{2q}}g}_{L^q}^{\frac{2q}{2q-3}}\left(\int_{\R^3} \abs{v}^{9/2}\varphi^2 g^p \dd v\right).
        \end{aligned}
    \end{equation}
    Summing \eqref{eq:eps_poincare_unweighted} and \eqref{eq:eps_poincare_homogeneous}, we obtain the desired bound.
    \end{proof}

\bibliographystyle{plain}
\bibliography{Landau-bib}

\end{document}